\documentclass[11pt]{amsart}
\usepackage[top=1.2in,bottom=1.2in,left=1.5in,right=1.5in]{geometry}

\usepackage[colorlinks]{hyperref}
\usepackage{amssymb,amsmath,verbatim,amsthm,mathabx,xfrac}
\usepackage{cleveref}
\usepackage{afterpage}
\usepackage{subcaption}
\usepackage{standalone}
\usepackage{enumitem}
\usepackage{mathtools}
\usepackage[mathscr]{euscript}

\newcommand{\nc}{\newcommand}
\nc{\rnc}{\renewcommand}
\nc{\nt}{\newtheorem}
\nc{\dmo}{\DeclareMathOperator}

\usepackage{pinlabel}
\usepackage[pdftex,dvipsnames]{xcolor}
\usepackage[colorinlistoftodos,prependcaption,textsize=small]{todonotes}
\usepackage{tikz}
\usetikzlibrary{patterns}

\nt{theorem}{Theorem}[section]
\nt{proposition}[theorem]{Proposition}
\nt{lemma}[theorem]{Lemma}
\nt{corollary}[theorem]{Corollary}

\dmo{\diam}{diam}
\dmo{\Mod}{Mod}
\dmo{\PMF}{PMF}
\dmo{\MF}{MF}
\nc{\PL}{P\!\mathscr{L}}
\dmo{\CAT}{CAT}

\nc{\ga}{\gamma}
\nc{\de}{\delta}

\nc{\C}{\mathcal{C}} 
\nc{\F}{\mathcal{F}} 
\nc{\A}{\mathcal{A}} 
\nc{\G}{\mathcal{G}} 
\dmo{\D}{D} 
\dmo{\B}{B} 

\nc{\HH}{\mathbb{H}} 
\nc{\R}{\mathbb{R}} 
\nc{\Z}{\mathbb{Z}} 

\nc{\Qcurve}{Q_\mathrm{curve}}
\nc{\Qtrack}{Q_\mathrm{track}}
\nc{\Qforce}{Q_\mathrm{force}}
\nc{\Qfit}{Q_\mathrm{fit}}
\nc{\Qdt}{Q_\mathrm{DT}}
\nc{\Qrect}{Q_\mathrm{rect}}

\renewcommand{\P}{\mathcal{P}} 
\renewcommand{\H}{\mathcal{H}} 

\dmo{\flex}{flex}
\dmo{\slope}{slope}

\renewcommand{\int}{\operatorname{int}}

\renewcommand\split{\rightharpoonup}
\nc\fold{\mapsto}
\nc{\transverse}{\pitchfork}
\nc{\carried}{\prec}
\nc{\fullycarried}{\mathrlap{\, \, \, \,\,\cdot}\carried}

\dmo{\hull}{\mathbf{r}}

\nc{\wt}{\widetilde}
\nc{\wh}{\widehat}

\dmo{\width}{width}
\dmo{\height}{height}

\Crefname{theorem}{Theorem}{Theorems}
\newtheorem*{main}{Main Theorem}
\Crefname{lemma}{Lemma}{Lemmas}
\Crefname{corollary}{Corollary}{Corollaries}
\Crefname{proposition}{Proposition}{Propositions}

\nc{\p}[1]{\medskip\paragraph{{\em #1}}}
\nc{\margin}[1]{\marginpar{\scriptsize #1}}

\author{Dan Margalit, Bal\'azs Strenner, Samuel J. Taylor, S. \"Oyk\"u Yurtta\c{s}}
\thanks{This material is based upon work supported by the the National Science Foundation under Grant Nos. DMS-1057874, DMS-1811941 DMS-2102018, and DMS-2203431.  We also acknowledge support through a Sloan Research Fellowship and a Fulbright Senior Scholar Award.}
\title{Quadratic-time computations for pseudo-Anosov mapping classes}

\begin{document}

\maketitle

\begin{abstract}
We give a quadratic-time algorithm to compute the stretch factor and the invariant measured foliations for a pseudo-Anosov element of the mapping class group.  As input, the algorithm accepts a word (in any given finite generating set for the mapping class group) representing a pseudo-Anosov mapping class, and the length of the word is our measure of complexity for the input.  The output is a train track and an integer matrix where the stretch factor is the largest real eigenvalue and the unstable foliation is given by the corresponding eigenvector.  This is the first algorithm to compute stretch factors and measured foliations that is known to terminate in sub-exponential time.   
\end{abstract}



\section{Introduction}

In his famous survey paper from 1982, ``Three dimensional manifolds, Kleinian groups and hyperbolic geometry,'' William Thurston laid out a series of 24 problems that would shape the field of low-dimensional topology for the ensuing decades \cite{Thurston}.  Problem~20 is:
\begin{quote}
\emph{Develop a computer program to calculate the canonical form for a general diffeomorphism of a surface, and to calculate the action of the group of diffeomorphisms on $\PL_0(S)$. Use this to implement an algorithm for the word problem and the conjugacy problem in the group of isotopy classes of diffeomorphisms of a surface.}
\end{quote}
Here, $\PL_0(S)$ is the space of (compactly supported) projective measured laminations on a surface $S$.  One reason for Thurston's interest in this question is revealed by his Problem~21: \emph{Develop a computer program to calculate hyperbolic structures on 3-manifolds.} Because of the central place of the theory of surfaces in low-dimensional topology, Problem~20 has importance in its own right.

Significant progress has been made on Thurston's problem (see below), but it has not been completely solved; there are algorithms to compute the canonical form in general, but not computer programs that work for all compact surfaces.  In this paper we give, for any compact surface, an easily-implementable and efficient algorithm for computing the canonical form in the case where the diffeomorphism is assumed to be of a certain type (namely, pseudo-Anosov).

To state our main result requires some setup.  Let $S=S_{g,p}$ be the surface obtained from the connected sum of $g$ tori by deleting $p$ points; the complexity of $S$ is $\xi(S) = 3g+p-3$. The Nielsen--Thurston classification theorem says that the elements of the mapping class group $\Mod(S)$ fall into three categories: periodic, reducible, and pseudo-Anosov, and that the last category is exclusive from the other two.  To a pseudo-Anosov element $f$ there are two pieces of data associated: a pair of measured foliations on $S$ and a real number $\lambda > 1$.  These are called the stable and unstable foliations of $f$ and the stretch factor.  Any foliation in $S$ can be represented by a train track $\tau$ in $S$ along with a vector in the vector space $V(\tau)$ spanned by the branches of $\tau$.  

A \emph{Nielsen--Thurston package} for $f \in \Mod(S)$ is a pair $(\tau,D)$, where $\tau$ is train track in $S$ and $D$ is a rational matrix that represents a linear map of $V(\tau)$, that has the stretch factor of $f$ as its largest real eigenvalue, and that has the unstable foliation of $f$ as a corresponding positive eigenvector.  

\begin{main}
Fix a surface $S=S_{g,p}$ and let $\Gamma$ be any finite generating set for $\Mod(S)$.  There is an explicit algorithm whose input is any word $w$ in $\Gamma$ that represents a pseudo-Anosov element of $\Mod(S)$, whose output is a Nielsen--Thurston package $(\tau,D)$  for that element, and whose running time is quadratic in the length of $w$ and polynomial time in the complexity $\xi(S)$.
\end{main}

This main theorem is stated more explicitly as Corollary~\ref{cor:main} below. We consider this result to be a step towards solving Thurston's problem in general.  Our algorithm is in the spirit of Thurston's problem: we directly compute (part of) the action of a mapping class on the space of measured foliations on a surface $S$, which is canonically isomorphic to $\PL_0(S)$.

There is no other algorithm for computing stretch factors and foliations whose running time is known to be faster than exponential.  Our paper does not address the issue of computing the eigendata for the matrix $D$, since there are many well-known (and fast) algorithms for this.

We can imagine many possible applications for our algorithm, for instance in determining the smallest stretch factor in $\Mod(S_g)$ for small $g$.  Equally important are the various tools we develop in the paper.  For example, with our theory of slopes (Section~\ref{section:slope}) and our forcing lemma (Proposition~\ref{prop:forcing}) we give a new criterion for a curve to be carried by a train track.  We also expect that our work can be used toward a fast algorithm for the conjugacy problem in the mapping class group.  

\p{Computing model and running time.} To say that an algorithm runs in quadratic time (as opposed to polynomial time) requires us to specify our model. We work in the bit model, that is, we count bit operations on a deterministic multitape Turing machine (and we refer to the number of bit operations as the run time).  Arithmetic is charged at bit cost: adding two \(t\)-bit integers costs \(O(t)\) bit operations, and multiplying a \(t\)-bit integer by an \(s\)-bit integer costs \(O(ts)\) bit operations (schoolbook multiplication).  No unit-cost arithmetic is assumed.

Let $n$ be the length of the word $w$ in the main theorem and let $\xi = \xi(S)$. The heart of our algorithm is to multiply together $O\!\bigl(\xi^2n\bigr)$ square matrices. The size of each matrix is $O\!\bigl(\xi\bigr)$ and the entries have fixed bit size. This takes $O\!\bigl(n^2 \xi^d \bigr)$ bit operations for some $d \leq 8$. The algorithm also uses $O\!\bigl(\xi\bigr)$ applications of the Zassenhaus algorithm whose running time is $O\!\bigl(n^2\xi^d \bigr)$ for some $d \leq 9$.

%
%

\p{The idea of our algorithm} The group $\Mod(S)$ acts in a piecewise-linear fashion on $\MF(S)$, the space of measured foliations on $S$.  Our main technical theorem, Theorem~\ref{thm:main}, says that there is a number $Q=Q(S)=O\!\bigl(\xi^2\bigr)$, a cell structure on $\MF(S)$, and a point $c$ in $\MF(S)$ so that for any pseudo-Anosov $f \in \Mod(S)$ the point $f^Q(c)$ lies in a region of $\MF(S)$ that has two properties: (1) it contains the unstable foliation for $f$ and (2) the action of $f$ on the region is linear.  From these two properties we may readily extract the desired Nielsen--Thurston package, as the unstable foliation is an eigenvector of the linear map from item (2).  We can compute $f^Q(c)$ and the action of $f$ at $f^Q(c)$ in quadratic time; as above, this essentially reduces to the problem of multiplying together $Qn$ matrices, where $n$ is the word length of $f$, and the entries of the matrices are of bounded size.    

\p{Prior work} Our algorithm has its origins in work of Thurston and Mosher (as exposited by Casson--Bleiler \cite[p. 146]{CassonBleilerNotes}), who gave a method for explicitly computing the action of $\Mod(S)$ on $\MF(S)$.  In his famous research announcement, Thurston suggested that under iteration by a pseudo-Anosov mapping class, points in $\MF(S)$ converge to the unstable foliation ``rather quickly.''  Our Theorem~\ref{thm:main} is an affirmation of this. 

Bell--Schleimer proved a result that, in their words, shows Thurston's suggestion to be false \cite{BellSchleimer}.  Their theorem says that, within the linear region described above, the convergence rate to the ray of unstable foliations is arbitrarily slow.  In other words, for the corresponding matrices, the absolute value of the ratio of the two eigenvalues of largest absolute value can be arbitrarily close to 1.  There is no contradiction: our result says that the sequence $c, f(c), f^2(c),\dots$ quickly arrives at the linear region containing the unstable foliation, and the Bell--Schleimer result says that, once there, the convergence of the sequence to the ray of unstable foliations can be slow.  

Other works have taken advantage of the fact that the sequence $f^k(c)$ eventually---and quickly---lands in the ``right'' linear region for $f$; see the work of Menzel--Parker \cite{MenzelParker}, Moussafir \cite{Moussafir06}, Finn--Thiffeault \cite{FinnThiffeault}, Hall (who wrote the computer program \href{http://pcwww.liv.ac.uk/maths/tobyhall/software/}{Dynn}) \cite{dynn}, the fourth-named author with Hall \cite{HallYurttas09}, and Bell (who wrote the computer programs \href{http://markcbell.co.uk/software.html#flipper-software}{flipper} and \href{http://markcbell.co.uk/software.html#curver-software}{curver}) \cite{BellThesis}.  Our Theorem~\ref{thm:main} is the first result to quantify how quickly the sequence $f^k(c)$ lands in the right linear region. 

There are also algorithms to compute stretch factors and foliations that do not use the action on $\MF(S)$ directly.  Most notably, Bestvina--Handel gave an algorithm that uses train tracks and is based on the Stallings theory of folding \cite{BH}.  The expectation is that this algorithm has exponential time complexity.  Brinkmann \cite{Brinkmann} implemented the algorithm for once-punctured surfaces in the computer program \href{https://github.com/nettoyeurny/xtrain}{Xtrain} and Hall implemented it for punctured disks in the program \href{http://pcwww.liv.ac.uk/maths/tobyhall/software/}{Trains}. 


\subsection{The main technical result} In order to state our main technical result, Theorem~\ref{thm:main}, we require some notation and setup.  As above, the surfaces in this paper are all of the form $S = S_{g,p}$.  Before we start in earnest, we give an abstract topological definition that will be used in the discussion of the space of measured foliations.  

\p{Integral cone complexes} In this paper, an integral cone complex is a sort of cell complex where a cell is a positive cone in some $\R^d$ defined as the set of solutions to a finite set of integer linear equations and integer inequalities; we will refer to such a cone as an integral linear cone.  Such a cell has a finite set of faces, each being a cell in the same sense.  Cells are glued along their faces, the gluing maps being linear.

\begin{center}
\begin{figure}[hbt!]
 \includegraphics[width=0.4\textwidth]{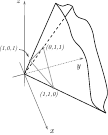}
\caption{A geometric realization of the integral cone complex $Y_\Delta$}\label{fig:ydelta}
\end{figure}
\end{center}

We now present a basic example of an integral cone complex $Y_\Delta$ that we will treat as a running example.  Let $V_1$, $V_2$, and $V_3$ each be the first quadrant in $\R^2$.  We glue the face of $V_i$ on the $x$-axis to the face of $V_{i+1}$ on the $y$-axis  (subscripts  considered modulo 3) by the map $(t,0) \mapsto (0,t)$.  The complex $Y_\Delta$ is homeomorphic to $\R^2$. 

The integral cone complex $Y_\Delta$ can be geometrically realized in $\R^3$ as the boundary of the first octant (by a geometric realization, we mean an embedding into $\R^N$ mapping cells to cells).  The complex $Y_\Delta$ is also geometrically realized as the set of points in the first octant of $\R^3$ whose coordinates satisfy a degenerate triangle inequality, meaning that one coordinate is the sum of the other two. See Figure~\ref{fig:ydelta} for an illustration.

With the definition of an integral cone complex in hand, we have several further definitions to make:
\begin{enumerate}[label=(\alph*)]
\item integral linear maps, integral cone complex maps, and linear regions,
\item the projectivized complex, vertex rays, and induced maps,
\item the radius,
\item subdivisions for generating sets,
\item PL-eigenvectors, 
\item PL-eigenregions,
\item projectively source-sink integral cone complex maps, and
\item extended dynamical maps and sink packages.
\end{enumerate}
After making these definitions, we explain how to compute PL-eigenvectors explicitly.

\bigskip \noindent (a) In the category of integral cone complexes, the natural maps $f : Y \to Y$ are integral cone complex maps, which we now define.  First, suppose $U$ is a subset of some cell $V$ of $Y$; we say that a map $U \to Y$ is integral linear if the image is contained in a cell $W$ and---using the $V$- and $W$-coordinates---the map is given by the restriction of an integral matrix.  

With this definition in hand, a map $f : Y \to Y$ is an integral cone complex map if there is an integral cone complex subdivision $Y_f$ of $Y$ so that the restriction of $f$ to each cell of $Y_f$ is an integral linear map.  If $V_i$ is a cell of $Y$ we denote the cells of such a subdivision of $V_i$ by $V_{i,k}^f$.  We will refer to the cells of any such subdivision as linear pieces for $f$. These pieces are not canonical; the important feature for us is that each $V_{i,k}^f$ is determined by finitely many integer linear inequalities in the $V_i$-coordinates.

If $U$ is any integral linear cone contained in a cell $V$ and $U$ is also $f$-linear, then $U$ can be extended to a subdivision $Y_f$ as above.  Therefore, it makes sense to refer to any such $U$ as a linear piece for $f$.  

We may describe the action of $f$ on an $f$-linear piece $U$ by a 4-tuple $(A,V_i,V_j,U)$, where $V_i$ and $V_j$ are cells of $Y$, the domain $U$ is an integral linear cone contained in $V_i$, the image $f(U)$ is contained in $V_j$, and $A$ is the matrix describing the action of $f$ on the linear piece $U$ in the given coordinates of $V_i$ and $V_j$.

In our algorithms, it will suffice to record only the triple $(A,V_i,V_j)$, omitting the domain.  Indeed, by multiplying the matrices from these triples, we can obtain a matrix for the product of mapping classes without knowing the domain.  In the case where $V_i$ is equal to $V_j$, we may write $(A,V_i)$ instead of $(A,V_i,V_j)$. 

To illustrate, we return to our running example $Y_\Delta$.  Let
\begin{align*}
V_{1,1} &=  \{ (x,y) \in V_1 \mid y \leq x \}, \text{and} \\
V_{1,2} &=  \{ (x,y) \in V_1 \mid x \leq y \}.
\end{align*}
There is an integral cone complex map $f_1$ given by the four 4-tuples
\[
\left( A_{1,1},V_1,V_1,V_{1,1}\right), \left(A_{1,2},V_1,V_2,V_{1,2}\right), \left(A_2,V_2,V_3,V_2\right), \text{ and } \left(A_3,V_3,V_3,V_3\right),
\]
where
\[
A_{1,1} = \left(\begin{smallmatrix*}[r] 1 & -1 \\ 0 & 1 \end{smallmatrix*}\right), \ \ 
A_{1,2} = \left(\begin{smallmatrix*}[r] 1 & 0 \\ -1 & 1 \end{smallmatrix*}\right), \ \ 
A_2 = \left(\begin{smallmatrix*}[r] 1 & 1 \\ 0 & 1 \end{smallmatrix*}\right), \ \ 
A_3 = \left(\begin{smallmatrix*}[r] 1 & 0 \\ 1 & 1 \end{smallmatrix*}\right), \ \ 
\]
There is another cone complex map $r$ given by permuting the $V_1$ cyclically, mapping each to the next by the identity matrix.

\bigskip \noindent (b) If $Y$ is an integral cone complex, there is a projectivized version of $Y$, denoted $\P(Y)$, which is a cell complex in the usual sense.  For instance the projectivized complex for $\MF(S)$ is $\P(\MF(S)) = \PMF(S)$ (Thurston  described the $\Mod(S)$-action on $\PMF(S)$ as opposed to $\MF(S)$).  If $f : Y \to Y$ is an integral cone complex map, then there is an induced cellular map $\P(Y_1) \to \P(Y_1)$, where $Y_1$ is an integral cone complex subdivision of $Y$.  We also define a vertex ray of an integral cone complex $Y$ to be the ray corresponding to a vertex of $\P(Y)$.

For our running example $Y_\Delta$, the cell complex $\P(Y_\Delta)$ is a triangle.  Under the geometric realization given by Figure~\ref{fig:ydelta}, $\P(Y_\Delta)$ is identified with the triangle shown in the picture.  There are three vertex rays in $Y_\Delta$.  The subdivision associated to the map $f_1$ is given by adding a vertex to $\P(Y_\Delta)$ corresponding to the ray spanned by $(1,1)$ in $V_1$; this is the midpoint of the $V_1$-edge of $\P(Y_\Delta)$.  

\bigskip \noindent (c) We now set out to define the radius of an integral cone complex $Y$---which we also refer to as the radius of $\P(Y)$.  First, we define a graph $\G_Y$ whose vertices are the top-dimensional cells of $\P(Y)$ and whose edges connect cells of $\P(Y)$ with nonempty intersection.  For a set of vertices $W$ of $\G_Y$, we define the radius of $Y$ at $W$ to be the supremal distance in $\G_Y$ from a vertex of $\G_Y$ to $W$.  For a vertex $c$ of $\P(Y)$, we define the radius of $Y$ at $c$ to be the radius of $Y$ at $W_c$, the set of top-dimensional cells of $\P(Y)$ containing $c$.

For the running example $Y_\Delta$, the graph $\G_Y$ is a triangle (dual to $\P(Y_\Delta)$), and so the radius with respect to any $W$ is at most 1.  As another example, if $Y$ is any triangulation of $S^2$ given by the equator and any number of longitudes, then the radius of $Y$ with respect to the vertex at the north pole is 1.  As we will see, this example mirrors the situation for $\MF(S)$.  

\bigskip \noindent (d) Suppose that $G$ is a group that acts on the integral cone complex $Y$ by integral cone complex maps, and say that $G$ is generated by $\Gamma = \{h_\ell\}$.  We denote by $Y_\Gamma$ the collection of integral cone complex subdivisions $\{ Y_{h_\ell} \}$ (so there is one subdivision for each generator).  As above, we denote the cells of the $h_\ell$-subdivision of $V_i$ by $V_{i,k}^{h_\ell}$.  

If $\Gamma$ is finite, then there is a single subdivision (namely, the common refinement of the $Y_{h_\ell}$)  with the property that each cell of the subdivision is a linear piece for each $h_\ell$ and $h_\ell^{-1}$.  In practice, it is convenient to use a different subdivision for each generator since this reduces the number of calculations required.  

\bigskip \noindent (e) Let $Y$ be an integral cone complex and let $f : Y \to Y$ be an integral cone complex map.  For a point $v$ in $Y$ and a $\lambda \geq 0$, the point $\lambda v$ is a well-defined point of $Y$.  With this in mind, a PL-eigenvector for $f$ is a point $v \in Y$ so that $f(v) = \lambda v$ for some $\lambda \geq 0$.  We emphasize that if $A$ represents $f$ on a cone $V$ of $Y$, a real eigenvector $v$ for $A$ is not necessarily a PL-eigenvector for $f$.  This is because $v$ might not lie in the cone of $\R^d$ corresponding to $V$.

In our running example, the point $(1,0) \in V_1$ is a PL-eigenvector of the integral cone complex map $f_1$ described above, since it is an eigenvector of the matrix $A_{1,1}$ and since the domain of this linear piece is contained in the codomain.  By the same token, $(0,1) \in V_1$ is not a PL-eigenvector of $f_1$: it is an eigenvector of $A_{1,2}$ but its image is $(0,1) \in V_2$, which is not identified with (any multiple of) $(0,1) \in V_1$.

\bigskip \noindent (f) Next, a PL-eigenregion for the integral cone complex map $f$ is a linear piece $U$ for $f$ that contains a PL-eigenvector $v$.  The linear transformation associated to $f|U$ will recognize $v$ as an eigenvector.  More specifically, if $f|U$ has the triple $(A,V_i,V_j)$ with $i=j$ then $v$ will be a (classical) eigenvector for $A$.  When $i \neq j$, we can choose coordinates on $V_i$ and $V_j$ that agree on the cell $V_{ij} = V_i \cap V_j$.  In these coordinates, we again have that $v$ is an eigenvector for $A$.

As part of our algorithm, we will find a point $w$ with the property that it is contained in the interior of a union of PL-eigenregions for $v$.  In such a case, there may be several different triples associated to the action of $f$ at $w$, and hence several different matrices.  However, all of these matrices will have $v$ as an eigenvector.  

In the running example, the map $f_1$ two PL-eigenregions for the PL-eigenvector $(1,0) \in V_1$, namely, $V_{1,1}$ and $V_3$.  We already discussed the former region. Similarly, $V_3$ is a PL-eigenregion for $f_1$ since $(0,1)$ is an eigenvector for $A_3$ and the third piece of $f_1$ has the same domain and codomain, which contains $(0,1)$.  

\bigskip \noindent (g) Next, an integral cone complex map $f : Y \to Y$ is projectively source-sink if the induced cellular map $\P(Y) \to \P(Y)$ has source-sink dynamics.  This means that the induced map has exactly two fixed points, the source and the sink, and that the $f$-iterates of all points outside of the source converge to the sink.  In particular, the only two PL-eigenvectors for $f$ (up to scale) are representatives of the source and the sink.  

In the running example, if we take $f_2 = r \circ f_1 \circ r^{-1}$, then the integral cone complex map
\[
f_2^{-1} \circ f_1
\]
is projectively source-sink.  The sink and source correspond to the vectors $(\varphi,1) \in V_3$ and $(1,\varphi) \in V_1$, where $\varphi$ is the golden mean.  The corresponding triples are
\[
(D,V_3,V_3) \text{ and } (E,V_1,V_1)
\]
where
\[
D = \left(\begin{smallmatrix*}[r] 2 & 1 \\ 1 & 1 \end{smallmatrix*}\right) \text{ and }
E = \left(\begin{smallmatrix*}[r] 2 & -1 \\ -1 & 1 \end{smallmatrix*}\right).
\]

\bigskip \noindent (h) Let $Y$ be an integral cone complex and $f : Y \to Y$ a projectively source-sink integral cone complex map.  Let $U$ be a PL-eigenregion for the sink of $f$ with triple $(A,V_i,V_j)$.  There exists a maximal subcone $W$ of $V_i \cap V_j$ invariant under $f$, and the map $f|W$ is a dynamical map in the sense that $f$ maps $W$ to itself.  The cone $W$ is nontrivial since it contains the ray corresponding to the sink.  In Section~\ref{sec:alg} we give an algorithmic procedure to find $W$ and to compute the linear map associated to $f|W$.  We also show that, in any $\R^d$-coordinates on $V_i$, the cone $W$ is the intersection of $V_i$ with a subspace of $\R^d$.  Hence it makes sense to project $V_i$ onto $W$.

We prove in Lemma~\ref{lem:specrad} that (under a certain condition about the location of the sink) the largest real eigenvalue of $f|W$ is the eigenvalue for the sink, and any corresponding positive eigenvector is the sink.  If we precompose $f|W$ with a projection map $V_i \to W$ (really a projection of the corresponding subspaces of $\R^d$), we obtain a linear map $V_i \to W$, which we call an extended dynamical map for the sink; say that $D$ is a matrix for this map.  The nonzero real eigenvalues of $D$ are the same as that of $f|W$.  As such, the pair $(V_i,D)$ contains the data of the sink; we refer to any such pair as a sink package.

In the running example, we have $i=j$  (this is the generic scenario); specifically $W = V_i = V_j = V_3$.  As such, the sink package for the map $f_2^{-1} \circ f_1$ is $(V_3,D)$, where $D$ is the matrix given above.

\p{The space of measured foliations...}  Turning back now to the mapping class group, we denote by $\MF(S)$ the space of measured foliations on $S$, again up to Whitehead moves and isotopy.  See the book by Fathi--Laudenbach--Po\'enaru \cite{FLP} for background on $\MF(S)$.

Let $S=S_{g,p}$ and let $\xi(S)=3g-3+p$.  Thurston proved that $\MF(S)$ is homeomorphic to $\R^{2\xi(S)}$.  He further proved that $\MF(S)$ has the structure of an integral cone complex \cite{Th}.   Thurston actually described the projectivized complex for $\PMF(S)$, which can be described in terms of train track coordinates as follows.

A train track in $S$ is a certain type of embedded finite graph in $S$ with a well-defined tangent line at each point; see the book of Penner--Harer \cite{PennerHarer92} for details.  Thurston's cells in the cell structure for $\PMF(S)$---equivalently, the projectivized integral cone complex structure for $\MF(S)$---correspond naturally to train tracks in $S$.  
Conversely, every train track gives a cone in $\MF(S)$ by considering the measured foliations that are carried by $\tau$.  Given a train track $\tau$, the associated cone will be denoted $V(\tau)$.  For any given $S$, we may choose once and for all a cell decomposition $Y$ of $\MF(S)$ with cells $V(\tau_1),\dots,V(\tau_N)$. 

There are infinitely many choices for the set of $V(\tau_i)$, that is, infinitely many choices of structures as an integral cone complex where the cells correspond to train tracks.  But they all give the same PL structure.  While there are integral cone complex structures on $\MF(S)$ that do not arise from train tracks, we will only make use of the ones that do; in what follows we refer to these as \emph{train track cell structures}.

As an example, let $S$ be $S_{0,4}$, or equivalently the thrice-punctured plane.  The train tracks $\tau_1$, $\tau_2$, and $\tau_3$ shown in Figure~\ref{fig:coords} give a cell structure on $\MF(S_{0,4})$.  The weights $x_i$ and $y_i$ on $\tau_i$ given an identification of each cell $V(\tau_i)$ with the first quadrant in $\R^2$.  The resulting integral cone complex structure on $\MF(S_{0,4})$ is isomorphic to our running example $Y_\Delta$.  The geometric realization given in Figure~\ref{fig:ydelta} has a geometric description, namely, it is given by the geometric intersection numbers of a weighted train track with three rays to infinity that emanate from the three marked points.  

\begin{figure}[hbt]
\labellist
\small\hair 2pt
\pinlabel {$\tau_1$} [ ] at 200 460
\pinlabel {$x_1$} [ ] at 185 525
\pinlabel {$y_1$} [ ] at 340 525
\pinlabel {$\tau_2$} [ ] at 40 180
\pinlabel {$x_2$} [ ] at 140 145
\pinlabel {$y_2$} [ ] at 195 240
\pinlabel {$\tau_3$} [ ] at 500 200
\pinlabel {$x_3$} [ ] at 355 225
\pinlabel {$y_3$} [ ] at 400 145
\pinlabel {$a$} [ ] at -10 460
\pinlabel {$b$} [ ] at 205 60
\endlabellist  
\centering
 \includegraphics[width=.5\textwidth]{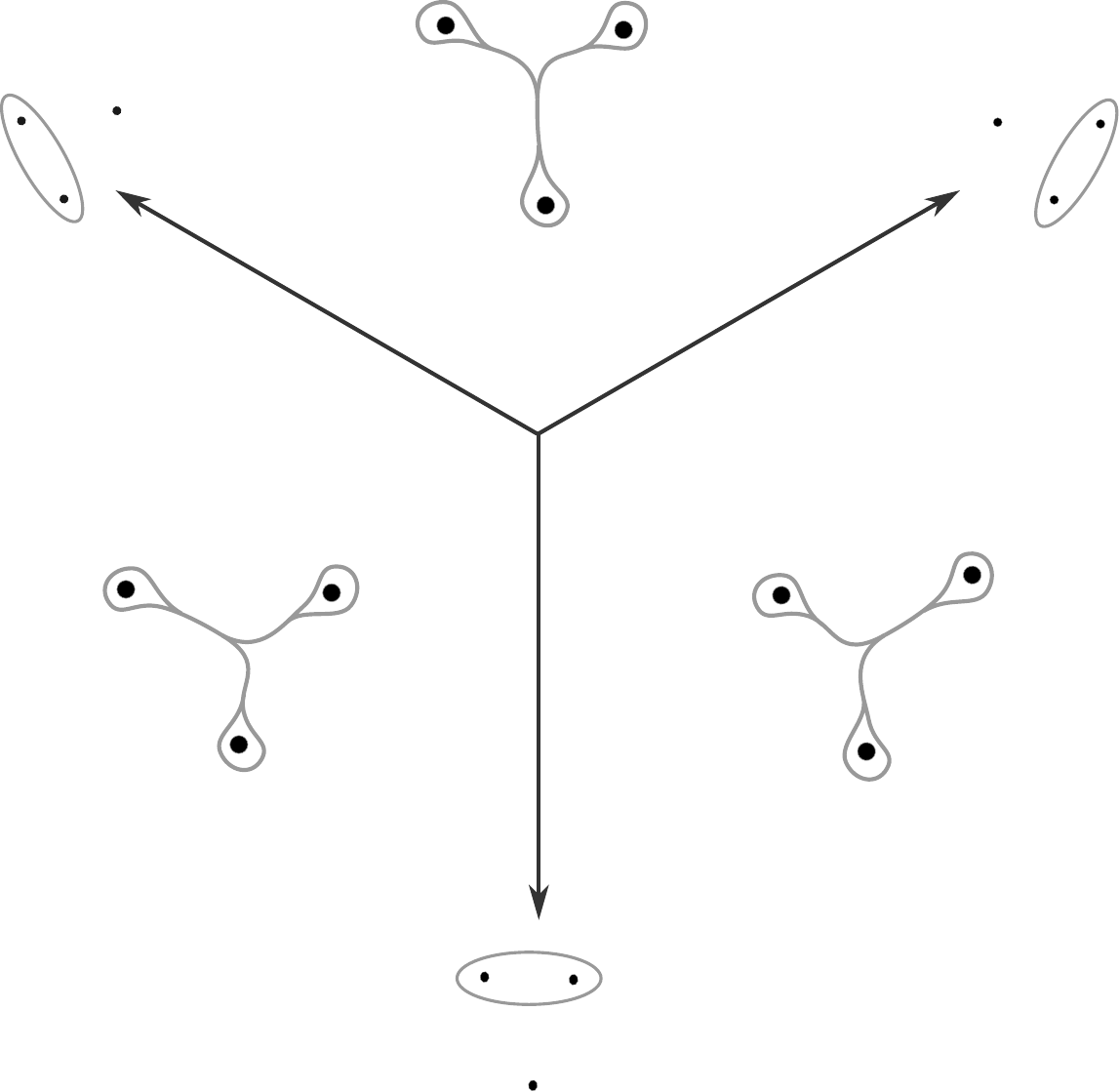}
\caption{Train tracks giving a cell structure on $\MF(S_{0,4})$}\label{fig:coords}
\end{figure}

There is a convenient class of train track cell structures on $\MF(S)$ that applies to an arbitrary $S$, namely, the ones arising from Dehn--Thurston train tracks.  These are defined as follows.  First, we choose a pants decomposition of $S$; we will refer to the curves of the pants decomposition as base curves.  From there, we decompose $S$ into a collection of annuli and pairs of pants.  The core curves of the annuli are the base curves and the pairs of pants are the closures of the complementary regions.  Each pair of pants has between one and three boundary components, with the other holes accounted for by punctures.  As in the book by Penner--Harer \cite[{\S}2.6]{PennerHarer92}, there are four model (sub)-train tracks on each annulus and at most four on each pair of pants.  By gluing together all possible models, we obtain the collection of Dehn--Thurston train tracks.  We prove in Lemma~\ref{lem:rad} that if $Y$ is a Dehn--Thurston cell structure on $\MF(S)$ and $c$ is one of the base curves, then the radius of $Y$ at $c$ is~1.  

\p{...and the mapping class group action} There is a natural action of $\Mod(S)$ on $\MF(S)$; again see \cite{FLP}. As explained by Thurston, the action of $\Mod(S)$ on $\MF(S)$ is by integral cone complex maps (in his setup, piecewise integral projective maps).

Given an integral cone complex structure $Y$ for $\MF(S)$ and a finite generating set $\Gamma$ for $\Mod(S)$, there is (as above) an associated collection of subdivisions $Y_\Gamma$ (as in our general discussion, there are canonical choices for these subdivisions, but in practice we compute subdivisions which may be finer than those). Moreover, each subdivision is a train track cell structure.  If $V(\tau_i)$ is a top-dimensional cell of $Y$ and $h_\ell$ is an element of $\Gamma$, we denote by $V(\tau_{i,1}^{h_\ell}), \dots, V(\tau_{i,N}^{h_\ell})$ the cells of the corresponding subdivision of $V(\tau_i)$.  Each $\tau_{i,j}^{h_\ell}$ is obtained from the corresponding $\tau_i$ by a sequence of splittings; again see \cite{PennerHarer92} for details.

An additional property of the $\Mod(S)$-action is that each pseudo-Anosov element acts on $\MF(S)$ with projective source-sink dynamics.  In particular, such an element has exactly two PL-eigenvectors.  The eigenvalues are $\lambda$ and $1/\lambda$, for the source and sink, respectively.  We will refer to a PL-eigenregion for either of these PL-eigenvectors as a Nielsen--Thurston eigenregion.

From a Nielsen--Thurston eigenregion, we obtain an extended dynamical map, and a corresponding sink package $(V(\tau_i),D)$, or simply $(\tau_i,D)$.  Any such sink package $(\tau_i,D)$ will be called a Nielsen--Thurston package.  This is the desired output of our algorithm.  

In our running example, the left-handed half-twist $H_a$ about the curve $a$ in Figure~\ref{fig:coords} (obtained by a clockwise $\pi$-rotation of the two leftmost marked points) is given by the map $f_1$ described above (using the above identification of $\MF(S_{0,4})$ with $Y_\Delta$), and similarly the half-twist $H_b$ is given by $f_2$.  The mapping class $H_b^{-1} H_a$ is pseudo-Anosov with stretch factor $\varphi^2$.  By the above calculation, the sink package for $H_b^{-1} H_a$ is $(D,V(\tau_3))$.

\p{Statement of the main technical theorem} Let $S=S_{g,p}$.  For the statement of Theorem~\ref{thm:main}, we must define a constant $Q=Q(S)$.  We build up the definition of $Q$ by defining a sequence of preliminary constants.  We do this in order to give a sense of how $Q$ arises, and to make feasible to improve our $Q$ by by improving the intermediate constants (and hence improve the implementation).  In the end, we give a definition of $Q$ using the intermediary constants, and also a formula given directly in terms of $g$ and $p$; the impatient reader is encouraged to skip to the latter definition of $Q$.

We begin with a basic constant related to the topology of $S$: 
\[
K(S) = 4\xi(S) = 4(3g-3+p).
\]
The number $K=K(S)$ is the maximum number of separatrices for a measured foliation on $S$; we refer to $K$ as the separatrix complexity of $S$.  In other words, this is the maximum possible sum of the degrees of the singularities of a foliation on $S$.

Next we have the constants
\[
D(S) = 1\quad \text{and} \quad E(S) = 14.
\]
The number $D=D(S)$ is the radius (as defined above) of any Dehn--Thurston integral cone complex structure on $\MF(S)$, specifically the radius about one of the curves of the underlying pants decomposition; see Lemma~\ref{lem:rad}.  The number $E(S)$ is an upper bound on the diameter in $\C(S)$ of the set of vertex cycles for a train track in $S$.  Aougab \cite[Section 5]{Aougab} proved that for all but finitely many $S$ there is an upper bound of 3 for all train tracks in $S$.   As observed by Tang--Webb \cite[Section 2.3]{TangWebb}, it follows from the work of Bowditch \cite{Bowditch} that there is an upper bound of 14 for all train tracks in all surfaces.  Masur--Minsky \cite{MasurMinsky99} had previously conjectured that the diameter is bounded above by 3 for all train tracks in all surfaces, but this is still an open problem.  

From the above constants, we now define several constants that appear in the statements of the main propositions in the paper or in the proof of Theorem~\ref{thm:main}:
Next, we have 
\begin{align*}
\Qcurve(S) &= 2K^2 \\
\Qtrack(S) &= E \cdot \Qcurve \\
\Qforce(S) &= \Qtrack + (2K+6) \cdot 2K\\
\Qfit(S) &= \Qforce(S) + (E+3) \cdot \Qcurve\\
\Qdt(S) &= (D+1) \Qtrack
\end{align*}
The first three of these constants appear in the statements of Propositions~\ref{prop:boundedslope}, \ref{prop:forcing}, and~\ref{prop:squeeze}, which we call the bounded slope lemma, the forcing lemma, and the fitting lemma.  The number $2K+6$ appears in the statement of Lemma~\ref{lem:covering_rectangles}.

Finally, we define 
\[
Q(S) = \Qforce(S) + \Qfit(S) + \Qdt(S) + 1. 
\]
This gives
\begin{align*}
Q(S_{g,p}) &= 2,464\, \xi^2 + 96\,\xi + 1 \\
&= 2,464(3g-3+p)^2 + 96(3g-3+p) + 1
\end{align*}
and in particular 
\[
Q(S_2) = 22,465 \quad \text{and} \quad Q(S_3) = 89,281.
\]
The coefficient 2,464 arises as $32(5E+7)$.  If one proves the Masur--Minsky conjecture that $E(S)$ can be replaced with 3, then we obtain
\[
Q(S_{g,p}) = 704\, \xi^2 + 96\,\xi + 1, \quad  Q(S_2) = 6,625, \quad \text{and} \quad Q(S_3) = 25,921.
\]
We may now state our main technical result in terms of Dehn--Thurston cell structures, Nielsen--Thurston eigenregions, and the number $Q=Q(S)$ just defined.  

\begin{theorem}
\label{thm:main}
Fix a surface $S=S_{g,p}$, a generating set $\Gamma$ for $\Mod(S)$, a Dehn--Thurston cell structure $Y$ of $\MF(S)$, and a base curve $c \in \MF(S)$ for the Dehn--Thurston coordinates.  Let $Q=Q(S)$.  Let $f \in \Mod(S)$ be pseudo-Anosov with unstable foliation $\F$ and stretch factor $\lambda$.  Then
\begin{enumerate}
\item $f^Q(c)$ lies in the interior of a union of Nielsen--Thurston eigenregions for $\F$, and
\item if $(V(\tau_i),D)$ is a Nielsen--Thurston package for $f$, then $\lambda$ is the largest real eigenvalue of $D$ and any corresponding positive eigenvector is a positive multiple of $\F$.  
\end{enumerate}
\end{theorem}

We may think of the first statement of Theorem~\ref{thm:main} as saying that the Nielsen--Thurston eigenregion is large in some sense (Proposition~\ref{prop:squeeze} even more directly has this interpretation).  For a general PL-action of a group on a manifold, we would expect the eigenregion for a sink to become arbitrarily small as the word length of $f$ becomes increasingly large.  As such, Theorem~\ref{thm:main} is not a general fact about PL-actions; our arguments rely heavily on technology specific to the setting of the mapping class group.

The point of the first statement of Theorem~\ref{thm:main} is that $f^Q(c)$ finds a Nielsen--Thurston eigenregion for $f$ (or perhaps several).  From the action of $f$ on any such eigenregion we may compute a Nielsen--Thurston package $(\tau,D)$.  Then the second statement of the theorem allows us to extract the stretch factor and unstable measured foliation for $f$.  

We emphasize that the cell structure on $\MF(S)$ is fixed once and for all.  If the cell structure---and hence the notion of linearity---were allowed to depend on $f$, then the corresponding statement would hold for $Q=1$, because there exists an invariant train track for $f$ that carries the curve $c$.  We also emphasize that our exponent $Q$ does not depend on the generating set for $\Mod(S)$.  We will return to this point in Section~\ref{sec:alg}.

\subsection{The algorithm}\label{sec:alg} We now explain in detail the algorithm that comes from Theorem~\ref{thm:main}.  As above, the algorithm runs in quadratic time with respect to the length of the spelling of $f$ as a word in any fixed generating set for $\Mod(S)$, and it runs in polynomial time with respect to $\xi(S)$. And the output is a train track and a matrix.  We state this result as Corollary~\ref{cor:main} below.

There are three sub-routines in the main algorithm: the linear piece algorithm, the basic computation algorithm, and the sink package algorithm.  We explain these in turn.  The first of these is used in the second; the second and third are both utilized in the main algorithm.

\p{Ledgers and the linear piece algorithm} Let $M$ be a space that admits an integral cone complex structure $Y$, and let $G$ be a finitely generated group acting on $M$ by integral cone complex maps.  Suppose we have chosen
\begin{enumerate}
\item a generating set $\Gamma = \{h_1,\dots,h_m\}$ for $G$ and
\item a cell structure for $Y$ on $M$ with cells $\{V_1,\dots,V_N\}$.
\newcounter{ledger}
\setcounter{ledger}{\value{enumi}}
\end{enumerate}
A linear piece algorithm for the pair $(Y,\Gamma)$ is an algorithm that takes as input an $h_\ell \in \Gamma$ and a cell $V_i$ of $Y$, and produces as output two pieces of data:
\begin{enumerate}
\setcounter{enumi}{\value{ledger}}
\item a subdivision $\{V_{i,k}^{h_\ell}\}$ of $V_i$ corresponding to $h_\ell$ and
\item the integer matrices $\A = \{A_{i,k}^{h_\ell}\}$ so each $h_\ell | V_{i,k}^{h_\ell}$  is described by $A_{i,k}^{h_\ell}$.
\end{enumerate}
Here the $V_{i,k}^{h_\ell}$ are described in terms of linear inequalities in the $V_i$-coordinates.

In some applications it may be expedient to write down once and for all the set of subdivisions $Y_\Gamma = \{Y_{h_\ell}\}$ for all cells $V_i$ and all generators $h_\ell \in \Gamma$, as well as the set of all matrices $\A = \{A_{i,k}^{h_\ell}\}$. We would refer to any such quadruple $(\Gamma,Y,Y_\Gamma,\A)$ as a \emph{ledger} for the action of $G$ on $M$.  As above, once $\Gamma$ and $Y$ are chosen, the $Y_\Gamma$ and the $\A$ are derived from those choices. In Section~\ref{section:example} we demonstrate our algorithm through a particular example, and in that case we start by giving the entire ledger.

\p{Linear piece algorithms for mapping class groups} For mapping class groups, there are many examples of linear piece algorithms (and ledgers) already in the literature.  In his thesis \cite{Penner}, Penner gave formulas---which give a constant-time linear piece algorithm---in terms of Dehn--Thurston train track coordinates (explicit matrices are not given, but they can be derived from the computations there).  Penner's formulas contain minor mistakes, corrected in his subsequent paper \cite{Penner06}. There is still one other mistake, pointed out by Yandi Wu: in the second elementary move (Theorem~A.2), the $\lambda$'s and $\kappa$'s should be switched \cite{yandi}.  For punctured surfaces, a constant-time linear piece algorithm is given by Hamidi-Tehrani--Chen \cite{htc}.

Again, for the case of the $\Mod(S)$-action on $\MF(S)$, the cells of $Y_\Gamma$ correspond to train tracks.  Further, we can choose the ledger so that each subdivision $V_{i,k}^{h_\ell}$ is equal to some $V(\tau_{i,k}^{h_\ell})$, where $\tau_{i,k}^{h_\ell}$ is obtained from train track $\tau_i$ by a sequence of splittings; see the book by Penner--Harer for the definition of a splitting \cite{PennerHarer92}.

In the case of $\Mod(S)$, the number of cells $V_i$ is exponential in $\xi(S)$. Therefore, in order to obtain an algorithm that runs in polynomial time with respect to $\xi(S)$, we cannot compute the entire ledger. Fortunately, this is not necessary. Penner's ledger (or, formulas) gives us a linear piece algorithm that runs in constant time with the Lickorish generators for $\Mod(S)$. This is because (a) each generator only interacts with a bounded number of coordinates, (b) there is a finite set of closed formulas describing the action of a generator on the corresponding coordinates, and (c) the corresponding linear inequalities have bounded bit size.

\p{Basic computational algorithm} Say that $G$ is a group with finite generating set $\Gamma=\{h_i\}$ and that $G$ acts on a space $M$ that admits the structure of an $m$-dimensional integral cone complex $Y$. Suppose we are given some word
\[
w = f_n \cdots f_1
\]
in the generating set $\Gamma$, and let $c \in M$ be a point, written as an integer point in some cell $V_i$ of $Y$.  Say that $w$ represents $f \in G$.  We would like to compute:
\begin{itemize}
\item $f(c) \in M$ and
\item a triple $(A,V_i,V_j)$ describing $f$ in a partial neighborhood of $c$.
\end{itemize}
By the latter, we mean (as above), a triple $(A,V_i,V_j)$ where $V_i$ and $V_j$ are cells of $Y$ and $A$ is a matrix describing $f$ on a linear piece $U \subseteq V_i$ containing $c$ (the matrix $A$ uses the $V_i$ and $V_j$ coordinates).  We now give the algorithm for performing these computations.  

\medskip

\noindent \emph{Setup.} A space $M$ and group $G$ as above, with a choice of linear piece algorithm for the choices of $Y$ and $\Gamma$.

\medskip

\noindent \emph{Input.} A word $w = f_n \cdots f_1$ in the generating set $\Gamma$ (so each $f_\ell$ is some $h_{i_\ell}^{\pm 1}$)

\medskip

\noindent \emph{Initialization.} Let $c_0=c$, with $c_0 \in V_{i_0}$, let $A_0 = I_m$ (the identity matrix), and let $\ell=1$.

\medskip

\noindent \emph{For $\ell = 1,\dots,n$}:

\begin{enumerate}
\item Apply the linear piece algorithm to $V_{i_{\ell-1}}$ and $f_\ell$, and determine a top-dimensional subdivision cell $V_{i_{\ell-1},k}^{f_\ell}$ that contains $c_{\ell-1}$ and on which $f_\ell$ is linear, has codomain in some $V_{i_\ell}$, and is represented by a matrix $A_{i,k}^{f_\ell}$.
\item Compute $c_{\ell} = A_{i,k}^{f_\ell} c_{\ell-1}$ and $A_{\ell} =  A_{i,k}^{f_\ell} A_{\ell-1}$.
\end{enumerate}

\medskip

\noindent \emph{Output.} $f(c) = c_n$ and $(A_n,V_{i_0},V_{i_n})$ 

\medskip

We emphasize that there can be different outputs for the triple $(A_n,V_{i_0},V_{i_n})$, since at each stage there may be more than one linear piece of $f_\ell$ containing $c$.  We also point out that step (1) is where the generating set for $G$ is most relevant to the complexity of the algorithm; if we choose a generating set that is not well-suited to the cell structure on $Y$, then the linear piece algorithm can be more complicated.

\p{The sink package algorithm} We now give an algorithm to find the sink package for a projectively source-sink integral cone complex map $f : Y \to Y$.  We suppose that we have in hand a triple $(A,V_i,V_j)$ for the action of $f$ on a PL-eigenregion for the sink $v \in Y$.  Both $V_i$ and $V_j$ are cones in $\R^d$, and we may assume without loss of generality that whatever faces are shared in $Y$ are also shared in $\R^d$.  In this case we say that the coordinates on $V_i$ and $V_j$ are compatible.  

\medskip

\noindent \emph{Input.} A triple $(A,V_i,V_j)$ for a PL-eigenregion for the sink of a projectively source-sink integral cone complex map.  We assume that the matrix $A$ is written in terms of compatible $\R^d$-coordinates on $V_i$ and $V_j$.  

\medskip

\noindent \emph{Initialization.} Let $W_0$ be the subspace of $\R^d$ spanned by $V_i \cap V_j$ (in the compatible coordinates).

\medskip

\begin{enumerate}
\item \emph{For $k = 1,\dots, m$}: \newline
\hspace*{1em} let $W_k = A(W_{k-1}) \cap W_{k-1}$
\item Choose a $d \times d$ matrix $B$ that projects $\R^d$ to $W_m$; set $D=AB$.
\end{enumerate}

\medskip

\noindent \emph{Output.}  $(V_i,D)$

\medskip

In general, projection onto an integral linear subspace gives a matrix with rational entries.  So, generically, we expect the output matrix $D$ to have rational entries.  

We will of course apply the sink package algorithm in the case of the mapping class group, in which case $V_i$ can be replaced by the train track defining the cell of $\MF(S)$ corresponding to $V_i$.  Again, we prove in Lemma~\ref{lem:specrad} that the largest real eigenvalue of $D$ is the eigenvalue for the sink, and any corresponding positive eigenvector is the sink.  

\p{Property Q} We will state the main algorithm for any group action on an integral cone complex that satisfies the same property as the one guaranteed for the mapping class group by the first statement of Theorem~\ref{thm:main}.  Specifically, we say that the action of a finitely generated group $G$ on an integral cone complex $Y$ satisfies property Q if there is a natural number $Q=Q(Y)$ so that, for any $f \in G$ acting with projectively source-sink dynamics and any $c \in Y$ lying on a vertex ray, the point $f^Q(c)$ lies in the interior of a union of PL-eigenregions for $f$ containing the sink. 

\p{The main algorithm} Say that $G$ is a group with finite generating set $\Gamma=\{h_i\}$ and that $G$ acts on a space $M$ that admits the structure of an $m$-dimensional integral cone complex $Y$. Fix a constant time linear piece algorithm for the choices of Y and $\Gamma$, and fix some nonzero $c \in M$ lying on a vertex ray of $Y$.  Suppose that the $G$-action on $M$ satisfies property Q for some specific $Q=Q(Y)$.  

\medskip

\noindent \emph{Input.} A word $w = f_n \cdots f_1$ in the generating set $\Gamma$ that represents a projectively source-sink element $f$ of $G$.  

\medskip

\noindent \emph{Step 1.} Applying the basic computation algorithm to $w$, compute $f^Q(c)$.

\medskip

\noindent \emph{Step 2.} Applying the basic computation algorithm to $w$, compute a triple $(A,\tau_i,\tau_j)$ associated to an $f$-linear piece $U$ of $Y$ containing $f^Q(c)$.

\medskip

\noindent \emph{Step 3.} Compute a sink package $(\tau_i,D)$ for $f$ using $V(\tau_i)$-coordinates.  

\medskip

\noindent \emph{Output.} $(\tau_i,D)$

\medskip

We are finally ready to state our corollary of Theorem~\ref{thm:main}, which gives the algorithm from the Main Theorem. 

\begin{corollary}
\label{cor:main}
Fix a surface $S$ and let $\Gamma$ be any finite generating set for $\Mod(S)$.  The main algorithm computes a Nielsen--Thurston package for any pseudo-Anosov $f \in \Mod(S)$, given as a word $w$ in $\Gamma$.  Assuming that there is a constant-time linear piece algorithm for the choice of $\Gamma$ and the choice of integral cone complex structure on $\MF(S)$, the running time of the main algorithm is quadratic in the length of $w$ and polynomial in $\xi(S)$.
\end{corollary}

Corollary~\ref{cor:main} follows immediately from Theorem~\ref{thm:main}, the standard running times for matrix multiplication and the Zassenhaus algorithm, and the existence of a constant time linear piece algorithm for the mapping class group. (The Zassenhaus algorithm here boils down to row reduction of matrices of size $O\bigl(\xi\bigr)$ and bit size $O\bigl(n\bigr)$, and this row reduction problem has the complexity stated in the introducition; see \cite[Lecture X, Theorem 5]{FPAA00}).

\p{The Nielsen--Thurston problem} The more general problem that our work fits into is the Nielsen--Thurston problem.  This is the algorithmic problem of determining, from a product of generators of $\Mod(S)$, the Nielsen--Thurston type of the resulting $f \in \Mod(S)$ and the associated data: the period if $f$ is periodic, the canonical reduction system if $f$ is reducible, and the stretch factor and foliations when $f$ is pseudo-Anosov.  

Bell--Webb gave a polynomial-time algorithm for the reducibility problem, that is, the problem of determining whether an infinite order mapping class (given as a word in the generators) is reducible \cite{BellWebb}.  In this case, their algorithm gives the canonical reduction system.  Otherwise, their algorithm computes the asymptotic translation length in the curve graph.  (It is known classically that if $f$ is periodic, then this---plus the period---can be determined in quadratic time.   Indeed, this follows from the upper bound on the period (depending only on $S$) and the quadratic-time solution to the word problem for $\Mod(S)$.)  Thus, combining our work with theirs, we obtain a polynomial-time algorithm to solve the Nielsen--Thurston problem.  

Baroni \cite{Baroni} recently gave an alternate version of the Bell--Webb algorithm.  Instead of just being polynomial in the word length, his algorithm is polynomial in the word length and the complexity of the surface.  On the other hand, Baroni's algorithm does not explicitly determine any reducing curves.  

As pointed out to us by Agol, there is an approach to the Nielsen--Thurston problem using the classical theory of 3-manifolds.  Given a mapping class, the algorithm proceeds by producing a triangulation of the corresponding mapping torus, and then using the Jaco--Tollefson algorithm to find the JSJ-decomposition \cite[Section 8]{JT}.  (The key point is that a pseudo-Anosov mapping class has a trivial JSJ-decomposition.)


\subsection{New tools and an outline of the proof}

We now give an outline of the paper and the proofs of the main results, as well as descriptions of the new tools introduced in the paper.  We begin in Section~\ref{section:example} by illustrating our algorithm through a concrete example.  

The rest of the paper is devoted to the proof of Theorem~\ref{thm:main}.  The keys to our proof are Propositions~\ref{prop:forcing} and~\ref{prop:squeeze}, called the forcing lemma and the fitting lemma.  Both propositions are stated in terms of the notion of slope for a curve in a surface, which is also new.  We now turn to developing this idea, beginning with the required background.  

\p{Horizontal and vertical foliations} Let $f \in \Mod(S)$ be a pseudo-Anosov mapping class.  In order to orient ourselves (and the reader), we will refer to the stable and unstable measured foliations, $(\F^s,\mu^s)$ and $(\F^u,\mu^u)$, as $(\F_v,\mu_v)$ and $(\F_h,\mu_h)$.  The `\emph{v}' and `\emph{h}' stand for ``vertical'' and ``horizontal'' and in our pictures we will (away from singularities) draw the foliations so that $\F_v$ is vertical, $\F_h$ is horizontal, and the transverse measures are $|dx|$ and $|dy|$.  In such a picture, the pseudo-Anosov representative of $f$ stretches horizontally and compresses vertically.

\p{Singular Euclidean structures} The two measured foliations $\F_v$ and $\F_h$ induce a singular Euclidean structure $X$ on $S$, the singular points being the singularities of the foliations. 
Let $\bar X$ be the closed singular Euclidean surface defined by taking the completion of $X$ with the induced metric. There is a set of marked points $\Sigma \subseteq \bar X$ so that $\bar X \setminus \Sigma$ is equal to the surface obtained from $X$ by deleting the singularities of $\F_v$ and $\F_h$.  In other words, $\Sigma$ is the set of points that either correspond to punctures of $X$ or singularities of the foliations (or both). 

We remark that $\F_h$ and $\F_v$ are also the horizontal and vertical measured foliations associated to an integrable meromorphic quadratic differential $q$ on $\bar X$, regarded as a closed Riemann surface.  

\p{Section~\ref{section:slope}: Saddle connections and slope} Let $X$ be a singular Euclidean surface.  A saddle connection in the surface $\bar X$ is a closed geodesic arc that meets $\Sigma$ in exactly two points, its endpoints.  The corresponding arc in $X$ is also called a saddle connection.  

Let $\sigma$ be such a saddle connection in $X$.  We define the \emph{slope} of $\sigma$ as its height over width:
\[
\slope_X(\sigma) = \frac{\mu_h(\sigma)}{\mu_v(\sigma)}.
\]
For any collection of saddle connections $Z$, we define $\slope_X(Z)$ to be the set of slopes of elements of $Z$.  

Let $c$ be a simple closed curve in $X$.  By pulling $c$ tight in $X$, the curve $c$ degenerates to an immersed path in $\bar X$ contained in a set of pairwise non-crossing saddle connections.  Here and throughout when we say that saddle connections are non-crossing, we mean that their interiors are disjoint.  Similarly, if $c$ is a measured foliation on $S$, we can pull tight each nonsingular leaf.  Each leaf gives an immersed path contained in a set of pairwise non-crossing saddle connections.  Further, since the leaves are pairwise disjoint, there is a single collection of pairwise non-crossing saddle connections so that each leaf gives an immersed path contained in the union.  For $c$ either a simple closed curve or a foliation, let $Z=\{\sigma_1, \ldots, \sigma_k\}$ be the resulting collection of pairwise non-crossing saddle connections.  We define the slope of $c$ as the finite set
\[
  \slope_X(c) = \slope_X(Z) = \{\slope_X(\sigma_1), \ldots, \slope_X(\sigma_k)\}.
\]
Finally, for a train track $\tau$, we define $\slope_X(\tau)$ to be the union of the slopes of the vertex cycles.  

\p{Section~\ref{section:slope} (cont.): Horizontality} It is sometimes convenient for us to use an alternate formulation of the slope.  Let $f : S \to S$ be a pseudo-Anosov mapping class with stretch factor $\lambda$ and singular Euclidean surface $X$, and let $c$ (or $\tau$) be as above.   We define
\[
\H_f(c) = -\frac{1}{2} \log_\lambda \slope_X(c).
\]
We refer to $\H_f(c)$ as the horizontality of $c$.  We are using an abuse of notation here, as $\H_f$ is not completely determined by $f$.  Indeed, different choices of $X$ lead to horizontality functions that differ by a constant.  For any such choice, horizontality  satisfies the equation
\[
\H_f(f^k(c)) = \H_f(c) + k
\]
for any $k$.  We also prove in Proposition~\ref{prop:boundedslope} that for any set of disjoint curves, arcs, or leaves of a foliation $c$, the set $\H_f(c)$ has diameter bounded above by $\Qcurve = 2K^2$, the key point being that this bound depends only on $S$ (and not on $f$).  

For two sets of real numbers $A$ and $B$ we say that $A < B$ if $\max A < \min B$.  We will use this notion to compare the slopes and horizontalities of the various objects discussed above. 

As an application of our upper bound on $\diam \H_f(c)$, we give a new approach to, and---in the case of a closed surface---a slight improvement of the result of Gadre--Tsai on stable translation distances in the curve graph; see Section~\ref{sec:GT}.

\p{Section~\ref{sec:rec}: Approximating rectangles} In Section~\ref{sec:rec} we prove Lemma~\ref{lem:covering_rectangles}, a technical lemma that will be used to prove the forcing lemma in Section~\ref{sec:forcing}.  The basic idea of the lemma is to approximate a configuration $(\gamma_1,\gamma_2,\ell)$, where the $\gamma_i$ are crossing saddle connections and $\ell$ is a horizontal leaf, by a configuration $(\gamma_1,\gamma_2,R)$, where $R$ is a maximal rectangle.  For the latter to be an approximation, we mean that the $\gamma_i$ cross the horizontal sides of $R$ and that the slope of $R$---defined to be the slope of either diagonal---is not too small.

\p{Section~\ref{sec:forcing}: The forcing lemma} Let $f$, $\F_h$, and $X$ be as above, and let $\tau$ be a filling, transversely recurrent train track in $S$.  We write $c \carried \tau$ and $\F \carried \tau$ to mean that a curve $c$ or a foliation $\F$ is carried by $\tau$, and in this case we write $\tau[c]$ and $\tau[\F]$ for the sub-train track of $\tau$ consisting of all branches traversed by $c$ or $\F$, respectively. 

Suppose that $c$ is a simple closed curve or a leaf of a measured foliation in $S$.  In the forcing lemma (Proposition~\ref{prop:forcing}) we assume that
\[
\H_f(c) \geq \H_f(\tau) + \Qforce.
\]
Under this assumption, the forcing lemma gives the following implications:
\begin{enumerate}
\item If $c \carried \tau$, then $\F_h \carried \tau[c]$.
\item If $\F_h \carried \tau$, then $c$ is carried by a diagonal extension of $\tau[\F_h]$.  
\end{enumerate}
Again, for the relevant train track-related definitions, including carrying and diagonal extensions, see the book by Penner--Harer \cite{PennerHarer92}.

\p{Section~\ref{sec:back}: The fitting lemma} Suppose $f$ is a pseudo-Anosov mapping class and that $\tau_1$ and $\tau_2$ are two train tracks carrying the horizontal foliation $\F_h$.  The fitting lemma (Proposition~\ref{prop:squeeze}) gives an $f$-invariant train track $\nu$---meaning that $f(\nu) \carried \nu$---with two properties:
\begin{enumerate}
\item $\nu \carried \tau_i$ for $i \in \{1,2\}$ and
\item $\H_f(\nu) \leq \H_f(\tau_i) + \Qfit$ for $i \in \{1,2\}$.
\end{enumerate}
The train tracks $\tau_1$ and $\tau_2$ will be Dehn--Thurston train tracks coming from the cell structure on $\MF(S)$.  On an intuitive level, if a train track $\nu$ is carried by $\tau_i$ and has much larger horizontality, then $V(\nu)$ is a very small sub-region of $V(\tau_i)$.  So the $\nu$ in the conclusion of Proposition~\ref{prop:squeeze} is fitted in the sense that $V(\nu)$ is both contained in $V(\tau_i)$ and is large (intuitively, we think of $V(\nu)$ as being almost the size of $V(\tau_i)$, although it may be of smaller dimension).  

\begin{center}
\begin{figure}[hbt!]
\labellist
\small\hair 2pt
\pinlabel {$\H_f(\tau_i)$} [ ] at 80 30
\pinlabel {$\H_f(\nu)$} [ ] at 135 30
\pinlabel {$\H_f(f^Q(c))$} [ ] at 375 30
\endlabellist  
 \includegraphics[scale=0.9]{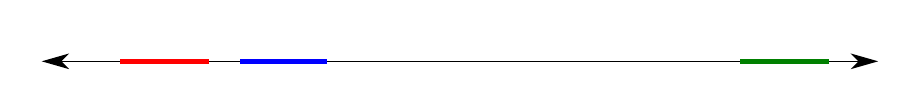}
\caption{In the proof of Theorem~\ref{thm:main}, we choose $Q$ so $\H_f(f^Q(c))$ is very large, we find a coordinate train track $\tau_i$ carrying, $\F_h$, and we choose an $f$-invariant $\nu$ carried by $\tau_i$ with $\H_f(\nu)$ small}\label{fig:balazs}
\end{figure}
\end{center}

\p{Section~\ref{sec:proof}: proof of the main technical theorem} We sketch the idea of the proof of Theorem~\ref{thm:main}.  Choose any curve $c$ lying on a vertex ray of $\MF(S)$.  Our $Q$ and $\Qdt$ are defined exactly so that  
\[
\H_f(f^Q(c)) \geq \H_f(\tau_i) + \Qforce + \Qfit + 1
\]
for all train tracks $\tau_i$ arising in the chosen cell decomposition of $\MF(S)$.  Suppose that $f^Q(c)$ is carried by the coordinate train track $\tau_i$.  By the forcing lemma, $\F_h$ is carried by $\tau_i$.  By the fitting lemma, there is an $f$-invariant train track $\nu$ so that 
\[
\H_f(f^Q(c)) \geq \H_f(\nu) + \Qforce + 1
\]
Then for every curve $d$ near $f^Q(c)$, the forcing lemma (plus the continuity of horizontality) gives a diagonal extension $\nu_d$ of $\nu$ that carries $d$.  Each such $V(\nu_d)$ is a Nielsen--Thurston eigenregion.  Since the $V(\nu_d)$ form an open neighborhood of $f^Q(c)$ in $\MF(S)$, every linear piece for $f$ that contains $f^Q(c)$ must overlap one of the Nielsen--Thurston eigenregions $V(\nu_d)$, and hence is contained in a Nielsen--Thurston eigenregion, as desired.

\p{Acknowledgments} We are grateful to Ian Agol, James Belk, Mark Bell, Joan Birman, Tara Brendle, Martin Bridson, Alex Eskin, Benson Farb, Juan Gonz\'alez-Meneses, Toby Hall, Chris Leininger, Livio Liechti, Scott MacLachlan, Kasra Rafi, Roberta Shapiro, Saul Schleimer, Richard Webb, and Alex Wright for helpful conversations.  Webb in particular suggested that the running time of our algorithm should be polynomial in terms of the complexity of the surface. We are also grateful to several anonymous referees for helpful comments, especially for help with algorithmic complexity issues and especially for pointing out the connection between our work and Thurston's question.


\section{An Example}\label{section:example}

In this section we apply (a version of) our algorithm to a specific example of a pseudo-Anosov mapping class.  Our example lies in the braid group $B_3$.  This group is naturally isomorphic to the mapping class group of $\D_3$, the closed disk with three marked points in the interior. It is also naturally isomorphic to the mapping class group of the surface obtained from a torus by deleting the interior of a closed, embedded disk.  

The standard generating set for $B_3$ is $\Gamma = \{\sigma_1,\sigma_2\}$.  Each generator $\sigma_i$ is a half-twist.  We consider the pseudo-Anosov braid
\[
f=\sigma_2\sigma_1^{-1}\sigma_2\sigma_1\sigma_1\sigma_1
\]
We will use our algorithm to show that the stretch factor of $f$ is
\[
\lambda = \frac{5 + \sqrt{21}}{2}
\]
and to obtain a description of the unstable foliation for $f$ in terms of train tracks.  

Our method in this section diverges from the main algorithm of this paper in that we do not make use of the explicit $Q$ from Theorem~\ref{thm:main}.  Instead, we use a guess-and-check approach.  Specifically, we use the basic computation algorithm to compute $f(c)$ and a triple $(A,V_i,V_j)$ for the action of $f$ at $f(c)$.  Then we find an eigenvector $v$ for $A$ and check directly that it is a PL-eigenvector for $f$ with eigenvalue $\lambda > 1$; it follows that $v$ represents the sink (if $v$ were not a PL-eigenvector, we would proceed by computing the action of $f$ at $f^2(c)$, etc.).  This method always terminates, and in practice it can be more efficient than our main algorithm (since $Q$ is so large).  Our Theorem~\ref{thm:main} guarantees that this method will terminate in quadratic time.

The discussion in this section has one other small difference from our general setup, namely, the fact that $\D_3$ has boundary; on the other hand, we could delete the boundary and consider $f$ as an element of $\Mod(S_{0,4})$ without changing the stretch factor or foliations.  Note that
\[
Q(S_{0,4}) = 2,464+96+1 = 2,561
\]
which is much larger than the power of $f$ required for the guess-and-check method, which in this case is 1.

\subsection{The ledger} We start by setting up the ledger for the $\B_3$-action on $\MF(\D_3) \cong \R^2$.  Since we already defined the generating set $\Gamma$, it remains to choose a cell decomposition $Y$ for $\MF(\D_3)$ and then find the collection of subdivisions $Y_\Gamma$ and the resulting collection of matrices $\A$.  

\p{The cell decomposition} The four train tracks $\tau_1$, $\tau_2$, $\tau_3$, and $\tau_4$ shown in the left-hand side of  Figure~\ref{fig:tx} give a cell decomposition $Y$ of $\MF(\D_3)$ as an integral cone complex.  The top-dimensional cells are the $V_i = V(\tau_i)$, and the faces are given by the rays through four curves shown.  For further explanation of this picture, see the book by Farb and the first author \cite[Chapter 15]{FarbMargalit12}  (the train tracks there are slightly different).  Here we have another difference from our general setup, as the train tracks here are not Dehn--Thurston train tracks; on the other hand these train tracks are especially suited to the braid group.  Also, this is a different structure on $\MF(\D_3) \cong \MF(S_{0,4})$ given in the introduction; the calculation can be carried out just as well with that setup.
  
\begin{figure}[hbt!]
\labellist
\small\hair 2pt
\pinlabel {$\tau_1$} [ ] at 245 244
 \pinlabel {$\tau_4$} [ ] at 251 91
 \pinlabel {$\tau_3$} [ ] at 9 93
 \pinlabel {$\tau_2$} [ ] at 14 246
 \pinlabel {$\scriptstyle{x}$} [ ] at 46 238
 \pinlabel {$\scriptstyle{y}$} [ ] at 36 178
 \pinlabel  {$\scriptstyle{y}$} [ ] at 31 87
 \pinlabel {$\scriptstyle{x}$} [ ] at 39 27
 \pinlabel  {$\scriptstyle{x}$}  [ ] at 220 27
 \pinlabel {$\scriptstyle{y}$}  [ ] at 225 87
 \pinlabel {$\scriptstyle{y}$} [ ] at 219 178
 \pinlabel {$\scriptstyle{x}$} [ ] at 216 238
 
  \pinlabel {$\scriptstyle{y}$}  [ ] at 493 255
 \pinlabel  {$\scriptstyle{x}$}  [ ] at 563 245
 \pinlabel  {$\scriptstyle{x}$} [ ] at 584 152
 \pinlabel {$\scriptstyle{y}$} [ ] at 608 185

 \pinlabel  {$\tau_{1,2}$}  [ ] at 560 265

 \pinlabel  {$\tau_{1,1}$}  [ ] at 586 208
  \pinlabel {$\tau_{2,1}$} [ ] at 337 208
  \pinlabel {$\tau_{2,2}$} [ ] at 370 265
  \pinlabel {$\tau_{3,1}$} [ ] at 330 70
   \pinlabel {$\tiny{\tau_{3,2}}$} [ ] at 363 25
    \pinlabel {$\tau_{4,1}$} [ ] at 620 70
    
  \pinlabel {$\tau_{4,2}$} [ ] at 568 25
 
 \pinlabel {$\scriptstyle{y}$}  [ ] at 612 106
 \pinlabel {$\scriptstyle{x}$} [ ] at 584 108

 \pinlabel {$\scriptstyle{y}$}  [ ] at 490 1
 \pinlabel {$\scriptstyle{x}$} [ ] at 492 14

 \pinlabel {$\scriptstyle{y}$} [ ] at 442 255
 \pinlabel {$\scriptstyle{x}$} [ ] at 377 245

 \pinlabel  {$\scriptstyle{x}$} [ ] at 353 150
 \pinlabel {$\scriptstyle{y}$} [ ] at 328 180

\pinlabel  {$\scriptstyle{y}$} [ ] at 326 111
\pinlabel {$\scriptstyle{x}$}  [ ] at 350 111
 
 \pinlabel {$\scriptstyle{y}$} [ ] at 445 5

 \pinlabel  {$\scriptstyle{x}$}  [ ] at 443 19
\endlabellist  
\includegraphics[width=.925\textwidth]{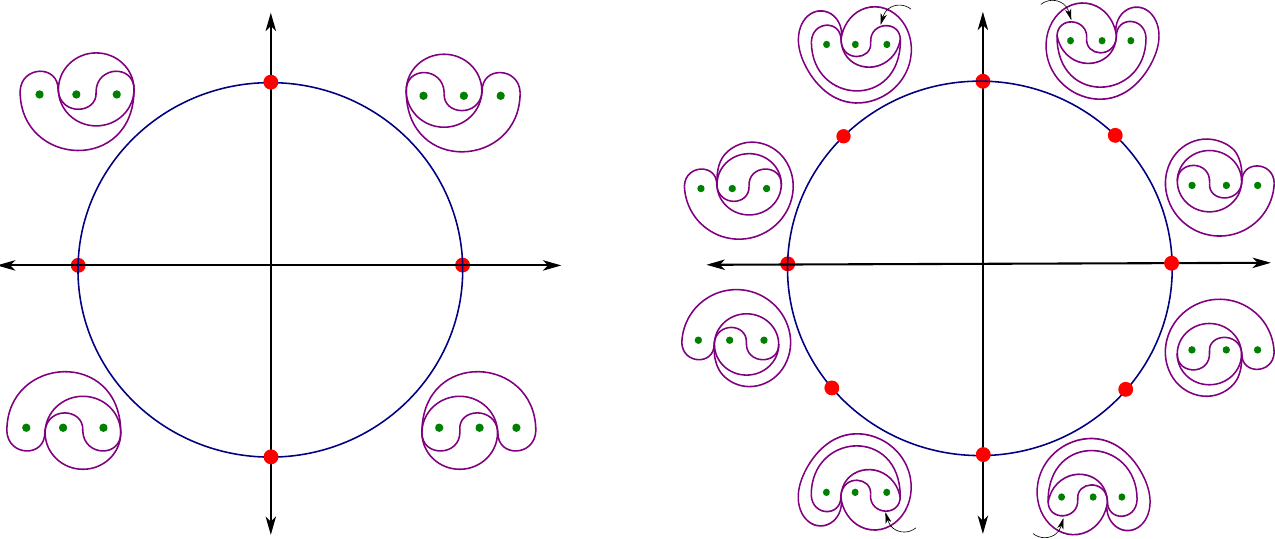}
\caption{The integral cone complex structure $Y$ on $\MF(\D_3)$ and the common refinement of $Y_\Gamma$}\label{fig:tx}
\end{figure}

\p{The subdivisions and the matrices} The eight train tracks $\tau_{1,1}, \tau_{1,2}, \dots, \tau_{4,1},  \tau_{4,2}$ shown in the right-hand side of Figure~\ref{fig:tx} give the common refinement of the collection of subdivisions $Y_\Gamma$ of $Y$ associated to $\Gamma$.  Indeed, for each $f \in \Gamma^{\pm}$ and each $\tau_{ij}$ there is a $k$ so that $f(V_{i,j}) \subseteq V_k$ and $f|V_{i,j}$ is linear.  We illustrate this claim in one case in Figure~\ref{fig:a} (for an example of a similar computation, see again the book by Farb and the first author).  The figure shows $\tau_{1,2}$, $\sigma_1(\tau_{1,2})$, and $\tau_2$, along with the weights corresponding to the $xy$-coordinates on $V_{1,2}$.  We see from the figure that
\[
\sigma_1(V_{1,2}) \subseteq V_2.
\]
For each generator (or inverse of a generator) $h$, the subdivision $Y_h$ lies between $Y$ and the common subdivision.  For instance, the subdivision $Y_{\sigma_1}$ has top-dimensional cells $V_{1,1}$, $V_{1,2}$, $V_2$, $V_3$, and $V_4$.  In each of the other three cases, the subdivision $Y_h$ is again obtained by subdividing one cell of $Y$.  

The calculations that allow us to check the inclusions $f(V_{i,j}) \subseteq V_k$ also allow us to compute the collection of matrices $\A$.  From the weights in Figure~\ref{fig:a} we see that the matrix $A_{1,2}^{\sigma_1}$ describing the action of $\sigma_1$ on $V_{1,2}$ is
\[
A_{1,2}^{\sigma_1} = \left( {\begin{array}{rr}
    -1& 1\\
    1& 0 \\
  \end{array} } \right)
\]
Since the coordinates on $V_{1,2}$ are compatible with the coordinates on $V_2$ (that is, the inclusion is given by the identity matrix), the matrix $A_{1,2}^{\sigma_1}$ does (as required) describe the map $\sigma_1|V_{1,2}$ in terms of the $V_1$- and $V_2$-coordinates.  As such, we can describe the action of $\sigma_1$ on its linear piece $V_{1,2}$ by the triple $(A_{1,2}^{\sigma_1},V_1,V_2)$.  The computations for the other cases are similar. 

From the refinement given in Figure~\ref{fig:tx}, we also obtain notation for each subdivision $Y_{\sigma_i^{\pm 1}}$.  For example, $Y_{\sigma_1}$ is the subdivision given by $V_{1,1}$, $V_{1,2}$, $V_2$, $V_3$, and $V_4$.  While we will not list the subdivisions for the other 7 generators and inverses of generators, the subdivisions are implicit from the calculations below: if we write $A_{1,2}^{\sigma_1}$ this signifies that $V_{1,2}$ is a linear piece for $\sigma_1$, and if we write $A_2^{\sigma_1}$ this signifies that $V_2$ is a linear piece for $\sigma_1$.  

\begin{figure}[hbt!]
\labellist
\small\hair 2pt
\pinlabel  {$x$} [ ] at 153 145
 \pinlabel  {$y$} [ ] at 55 125
 \pinlabel {$\sigma_1$} [ ] at 190 120
 \pinlabel  {$y$} [ ] at 285 123
 \pinlabel  {$x$} [ ] at 375 150
 \pinlabel  {$x$} [ ] at 470 135
 \pinlabel {$y-x$} [ ] at 620 150
\endlabellist  
 \includegraphics[width=.8\textwidth]{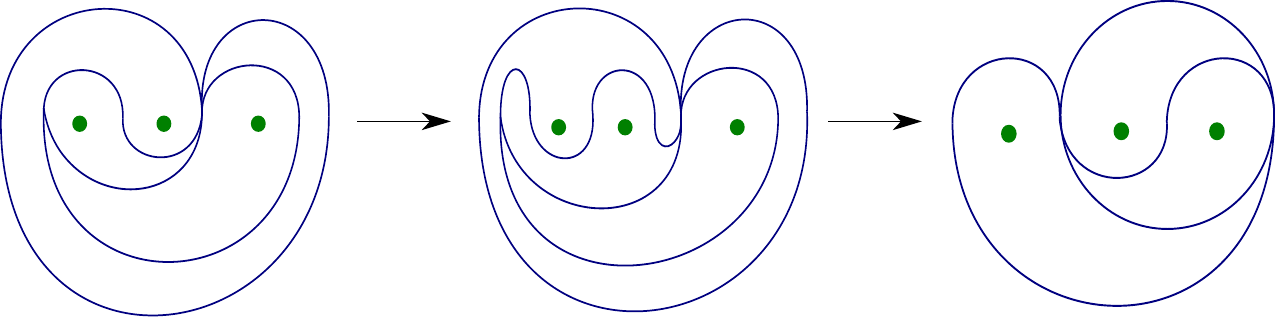}
\caption{A part of the computation of the linear transformation $V_{12} \to V_2$ induced by $\sigma_1$}\label{fig:a}
\end{figure}

\subsection{Applying the algorithm} As above, let $f=\sigma_2\sigma_1^{-1}\sigma_2\sigma_1\sigma_1\sigma_1$.  Let $c$ be the curve described by the vector $(1,2)$ in $V_1$-coordinates (again, contrary to our general setup, this curve does not lie on a vertex ray).  We first use the basic computation algorithm to give a detailed description of how to compute $f(c)$, and then use the modified version of the main algorithm to find the stretch factor and foliations.  

\p{Applying the basic computation algorithm} Thinking of $f$ as a word $w$ in $\Gamma$ we have $f_1 = \sigma_1$, $f_2 = \sigma_1$, etc.  As in the description of the basic computation, we initialize by setting $c_0=c$.  

The following table gives the details of the steps of the basic computation for finding $f(c)$.  In each row we give:
\begin{enumerate}
\item the step number $\ell$, 
\item the curve $c_{\ell-1}$ from the previous step,
\item the cell $V_{i,k}$ of the subdivision $Y_\Gamma$ containing $c_{\ell-1}$,
\item the next letter $f_\ell$ in the word $w$,
\item the codomain $V_j$ for $f_\ell|V_{i,k}$, and
\item the matrix $A_{i,k}^\ell$ so that $f_\ell|V_{i,k}$ is described by $(A_{i,k}^\ell,V_i,V_j)$.
\end{enumerate}

Here is the table:

\renewcommand{\arraystretch}{1.5} 
\begin{center}
 \begin{tabular}{c|llllrr}
 $\ell$ & $c_{\ell-1}$ & $V_{i,k}$ & $f_\ell$ & $V_j$ & $A_{i,k}^\ell$ & $c_\ell$ \\ \hline
 1 & $\left(\begin{smallmatrix} 1\\2\end{smallmatrix}\right)$ & $V_{1,2}$ & $\sigma_1$ & $V_2$ & $\left(\begin{smallmatrix}-1&1\\ \ \ 1& 0\end{smallmatrix}\right)$ & $\left(\begin{smallmatrix} 1\\1\end{smallmatrix}\right)$ \\
2 & $\left(\begin{smallmatrix} 1\\1\end{smallmatrix}\right)$ & $V_2$ & $\sigma_1$ & $V_3$ & $\left(\begin{smallmatrix} 0&1\\1&0\end{smallmatrix}\right)$ & $\left(\begin{smallmatrix} 1\\1\end{smallmatrix}\right)$ \\
3 & $\left(\begin{smallmatrix} 1\\1\end{smallmatrix}\right)$ & $V_3$ & $\sigma_1$ & $V_4$ & $\left(\begin{smallmatrix} 0&1\\1&1\end{smallmatrix}\right)$ & $\left(\begin{smallmatrix} 1\\2\end{smallmatrix}\right)$ \\
4 & $\left(\begin{smallmatrix} 1\\2\end{smallmatrix}\right)$ & $V_4$ & $\sigma_2$ & $V_1$ & $\left(\begin{smallmatrix} 0&1\\1&0\end{smallmatrix}\right)$ & $\left(\begin{smallmatrix} 2\\1\end{smallmatrix}\right)$ \\
5 & $\left(\begin{smallmatrix} 2\\1\end{smallmatrix}\right)$ & $V_1$ & $\sigma_1^{-1}$ & $V_1$ & $\left(\begin{smallmatrix} 1&1\\0&1\end{smallmatrix}\right)$ & $\left(\begin{smallmatrix} 3\\1\end{smallmatrix}\right)$ \\
6 & $\left(\begin{smallmatrix} 3\\1\end{smallmatrix}\right)$ & $V_1$ & $\sigma_2$ & $V_2$ & $\left(\begin{smallmatrix} 0&1\\1&1\end{smallmatrix}\right)$ & $\left(\begin{smallmatrix} 1\\4\end{smallmatrix}\right)$ 
 \end{tabular}
 \end{center}
\renewcommand{\arraystretch}{1} 
(For $\ell > 1$ we use $V_i$ instead of $V_{i,k}$ since the entire $V_i$ is a linear piece for the corresponding $f_\ell$.) From the table, we see that
\[
f(c) = \left({\begin{array}{rr} 1 \\ 4 \\ \end{array} } \right)  \in V_2.
\]

\p{An aside.} We can further see from the above computation that $c$ lies in a linear piece for $f$ described by a triple $(A,V_1,V_2)$, where $A$ is given by the product of the above matrices:
\begin{align*}
A &= A_1^{\sigma_2}A_1^{\sigma_1^{-1}}A_4^{\sigma_2}A_3^{\sigma_1}A_2^{\sigma_1}A_{1,2}^{\sigma_1} \\
&= \left({\begin{array}{rr} 0 & 1 \\ 1 & 1 \\ \end{array} } \right) \left({\begin{array}{rr} 1 & 1 \\ 0 & 1 \\ \end{array} } \right) \left({\begin{array}{rr} 0 & 1 \\ 1 & 0 \\ \end{array} } \right) \left({\begin{array}{rr} 0 & 1 \\ 1 & 1 \\ \end{array} } \right) \left({\begin{array}{rr} 0 & 1 \\ 1 & 0 \\ \end{array} } \right) \left({\begin{array}{rr} -1 & 1 \\ 1 & 0 \\ \end{array} } \right) \\
& = \left({\begin{array}{rr} -1 & 1 \\ -2 & 3 \\ \end{array} } \right) 
\end{align*}
We can check this by observing that
\[
\left({\begin{array}{rr} -1 & 1 \\ -2 & 3 \\ \end{array} } \right) \left({\begin{array}{rr} 1 \\ 2 \\ \end{array} } \right) = \left({\begin{array}{rr} 1 \\ 4 \\ \end{array} } \right) 
\]
We will not use this description of action of $f$ on the linear piece containing $c$; we include it just to illustrate how the computations work.

\p{Iterative-guess-and-check algorithm} We now describe in more detail the guess-and-check algorithm described at the start of the section.  This methid is the underpinning of some of the computer programs that predate our work.  The idea of the algorithm is that, starting with $k=1$, we ``guess'' that $f^k(c)$ lies in a Nielsen--Thurston eigenregion for $f$, and then
\begin{enumerate}
\item compute a triple $(A,V_i,V_j)$ for the action of $f$ on a linear piece containing the curve $f^k(c)$, 
\item find pairs $(\lambda,v)$ where $\lambda > 1$, where $v$ lies in $V_i \cap V_j$, and where $f(v) = \lambda v$ (if $V_i=V_j$ such $v$ can be computed as classical eigenvectors), and
\item check if one pair $(\lambda,v)$ with $\lambda > 1$ and $v \geq 0$ is a PL-eigenvalue/eigenvector pair for $f$.
\end{enumerate}
Since $f$ has exactly one PL-eigenvalue/eigenvector pair with $\lambda > 1$ (up to scale), the last step implies that $\lambda$ is the stretch factor and $v$ represents the unstable foliation for $f$.  If either the second or third step fails (meaning no pairs are found in the second step, or it is found in the third step that they are not PL-eigenvectors), we return to the first step, replacing $f^k(c)$ with $f^{k+1}(c)$.  If $f$ is pseudo-Anosov, then this process eventually terminates, since $f^k(c)$ must lie in a Nielsen--Thurston eigenregion for $k$ large.

Again, the complexity of the iterative-guess-and-check algorithm is impossible to determine without an upper bound on the power of $f$ required.  Our Theorem~\ref{thm:main} exactly gives such an upper bound, namely $Q$.  

\p{Applying the iterative-guess-and-check algorithm} We now apply the above algorithm to our specific $f \in \B_3$.  As we will see the algorithm terminates for $k=1$.  So this process terminates faster than if we were to apply the main algorithm, which would require computing $f^Q(c)$ with $Q = 2561$.  

\medskip

\noindent (1) For the first step we apply the basic computation algorithm to the action of $f$ on $f(c)$ and find the resulting acting matrix:
\begin{align*}
A &= A_{2}^{\sigma_2}A_{3}^{\sigma_1^{-1}}A_{4}^{\sigma_2}A_{4}^{\sigma_1}A_{3}^{\sigma_1}A_{2}^{\sigma_1} \\
&= \left({\begin{array}{rr} 0 & 1 \\ 1 & 0 \\ \end{array} } \right) \left({\begin{array}{rr} 0 & 1 \\ 1 & 1 \\ \end{array} } \right) \left({\begin{array}{rr} 1 & 1 \\ 0 & 1 \\ \end{array} } \right) \left({\begin{array}{rr} 0 & 1 \\ 1 & 0 \\ \end{array} } \right) \left({\begin{array}{rr} 1 & 1 \\ 0 & 1 \\ \end{array} } \right) \left({\begin{array}{rr} 0 & 1 \\ 1 & 1 \\ \end{array} } \right) \\
&= \left({\begin{array}{rr} 3 & 5 \\ 1 & 2 \\ \end{array} } \right) 
\end{align*}
This computation is done in the same way as the computation of $f(c)$.  We do not include the corresponding table, but this can be inferred from the given product.  During the computation, the curve $f(c) \in V_2$ gets mapped to $V_3$, to $V_4$, to $V_4$, to $V_1$, again to $V_1$, and back $V_2$, in that order.  In particular $f(f(c))$ lies in $V_2$.  This means that the result of the first step is the triple:
\[
\left( \left( {\begin{array}{cc}
    3& 5\\
    1& 2 \\
  \end{array} } \right) , V_2, V_2   \right) 
\]

\medskip

\noindent (2) We now proceed to the second step.  Since we have $V_i=V_j$, and we are guessing that $f(f(c))$ lies in a Nielsen--Thurston eigenregion for $f$, we compute the eigenvalue $\lambda > 1$ and corresponding eigenvector for the matrix $A$ found in the first step.  We find
\[
\lambda = \frac{5 + \sqrt{21}}{2} \quad \text{and} \quad v = \left( {\begin{array}{cc}
    1 + \sqrt{21}\\
    2 
  \end{array} } \right).
\]

\medskip

\noindent (3) For the third and final step, we again use the basic computation algorithm to compute the action of $f$ on a linear piece containing the eigenvector found in the last step.  This time the computation turns out to be exactly the same as the computation in the first step, as the curves computed during the computation lie in exactly the same linear pieces for the six letters in the word for $f$.  It follows that the pair $(\lambda,v)$ found above is a PL-eigenvalue/eigenvector pair for $f$.  As above, this means that $\lambda$ is the stretch factor for $f$ and the measured train track $(\tau_2,v)$ represents the unstable foliation.  

In order to find coordinates for the stable foliation for $f$, we can repeat the whole process, except with $f$ replaced by $f^{-1}$.  We find that the stable foliation is represented by $(-3+\sqrt{21} , 6)$ in $V_1$-coordinates.


\section{Slope}
\label{section:slope}

The main goal of this section is to prove Proposition~\ref{prop:boundedslope}, which bounds the diameter of the horizontality of a curve or leaf of a foliation.  We also introduce other tools related to slope and horizontality.  At the end of the section we give our application of slope to translation lengths in the curve graph, Corollary~\ref{cor:Lip}.

For the statement of Proposition~\ref{prop:boundedslope}, recall from the introduction that $\Qcurve(S_{g,p}) = 2K^2$ and $K=K(S_{g,p}) = 4(3g-3_p))$. Also recall that $\H_f$ is the horizontality function coming from any singular Euclidean surface associated to a pseudo-Anosov $f$.  

\begin{proposition}
\label{prop:boundedslope}
Let $S=S_{g,p}$ and let $f \colon S \to S$ be a pseudo-Anosov homeomorphism with stretch factor $\lambda$ and a singular Euclidean structure $X$.  If $Z$ is a set of pairwise non-crossing saddle connections in $\bar X$, then
\[
\diam \H_f(Z) \leq \Qcurve(S).
\]
\end{proposition}

In our applications of Proposition~\ref{prop:boundedslope} we will take $Z$ to be the set of saddle connections appearing in the (flat) geodesic representative $c_X$ of a curve $c$ or a leaf of a foliation in $S$.  This makes sense because the set of saddle connections appearing in the geodesic representative $c_X$ of a curve or a leaf are pairwise non-crossing.  More generally, Minsky and the third-named author of this paper proved that if $c$ and $d$ are two disjoint curves/leaves, then the geodesic representatives $c_X$ and $d_X$ consist of pairwise non-crossing saddle connections; see \cite[Lemma 4.3]{MinskyTaylor}.   

\p{Outline of the section} We begin in Section~\ref{sec:slope_prelim} with some preliminary definitions: completed universal covers, rectangles, and rectangle flexibility.  The latter, denoted $\flex(X)$, describes the amount that rectangles in a singular Euclidean surface can be extended in the horizontal and vertical directions.  

Next, in Section~\ref{sec:slope_disparity} we prove several lemmas that compare slopes of saddle connections in a fixed singular Euclidean surface.  We first prove Lemma~\ref{lemma:local-disparity}, a local slope disparity lemma, which bounds in terms of $\flex(X)$ the disparity between the slopes of saddle connections in a single rectangle (we obtain Lemma~\ref{lemma:local-disparity} as a consequence of Lemma~\ref{lem:slopeflex}, which compares the slope of a saddle connection to the slope of a rectangle containing it).  We then promote Lemma~\ref{lemma:local-disparity} to Lemma~\ref{lemma:global-disparity}, a global slope disparity lemma, which bounds in terms of $\flex(X)$ the disparity between the slopes of any non-crossing saddle connections in $X$.  

In Section~\ref{sec:slope_stretch} we consider singular Euclidean surfaces together with associated pseudo-Anosov maps.  In Lemma~\ref{lemma:flexibility} we relate $\flex(X)$ to the stretch factor of a pseudo-Anosov map.  We then turn to the proof of Proposition~\ref{prop:boundedslope}.  Since the slope of a curve is the set of slopes of a collection of non-crossing saddle connections, Proposition~\ref{prop:boundedslope} will follow from Lemmas~\ref{lemma:global-disparity} and~\ref{lemma:flexibility}.

Finally, in Section~\ref{sec:GT}, apply the notion of horizontality to prove Corollary~\ref{cor:Lip}.  This corollary bounds from below the stable translation length of the action of a pseudo-Anosov mapping class on the corresponding curve graph, extending the work of Gadre--Tsai. 


\subsection{Preliminaries}
\label{sec:slope_prelim}

In this section, we introduce here the definitions of completed universal covers, rectangles, and rectangle flexibility, along with several related notions.  

\p{Completed universal covers} Let $X$ be a singular Euclidean surface and let $\bar X$ be the associated closed singular Euclidean surface, obtained from $X$ by filling in the punctures.  As in the work of Minsky and the third-named author, we denote by $\wh{X}$ the metric completion of the universal cover of $X$.  We refer to $\wh{X}$ as the completed universal cover of $X$.  There is an induced projection $\wh{X} \to \bar X$.  This map can be regarded as a branched cover, with infinite-fold ramification over the points of $\bar X \setminus X$.  

The universal cover $\wt X$ can be identified with $\HH^2$, and as such the completed universal cover $\wh{X}$ can be identified with a subset (but not a subspace!) of the visual compactification of $\HH^2$.  The space $\wt{X}$ is a CAT(0) space.  In particular, given a pair of points in $\wt{X}$, there is a unique geodesic connecting them.  This geodesic may pass through arbitrarily many points of $\wh{X} \setminus \wt X$.  

\p{Rectangles and saddle connections}  Let $X$ be a singular Euclidean surface with horizontal and vertical foliations $\F_h$ and $\F_v$.   A rectangle in $X$ or $\bar X$ is defined to be an isometric immersion $R \colon E \to \bar X$, where
\begin{enumerate}
\item $E$ is a Euclidean rectangle,
\item each horizontal/vertical segment in $E$ maps to a horizontal/vertical segment, and
\item $R^{-1}(\Sigma) \subseteq \partial E$(recall $\Sigma$ is the set of singular points).
\end{enumerate} 
While $R$ is an immersion, it is always possible to lift $R$ to $\wh X$ and any such lift will be an embedding.  In all of our figures, the rectangles will be shown as embedded; we can (safely) imagine that these pictures describe lifts to $\wh X$.

By a diagonal of $R$ we mean the $R$-image of a diagonal of $E$.   When we say that a point $x \in X$ lies in the boundary or interior of $R$, we mean that $R^{-1}(x)$ lies in the boundary or interior of $E$, respectively.  The height and width of $R$ are defined to be the height and width of $E$.  

We say that a saddle connection $\sigma$ in $\bar X$ is rectangle spanning if it is the diagonal of a rectangle.  In this case the rectangle is denoted $R_\sigma$ and is called the spanning rectangle for $\sigma$.  The height and width of a rectangle-spanning saddle connection are defined to be the height and width of its spanning rectangle.  

Minsky and the third-named author proved that if $X$ a singular Euclidean surface associated to a pseudo-Anosov mapping class $f$, the triangulations of $X$ by rectangle-spanning saddle connections are exactly the sections of Agol's veering triangulation  \cite[Section 3]{MinskyTaylor}.  

\p{Slopes of rectangles and flexibility}  Let $X$ be a singular Euclidean surface.  As in the introduction, we define the slope of a rectangle $R \subseteq X$ as $\slope(\delta)$, where $\delta$ is either diagonal of $R$ (the diagonals have equal slope).  Also, we say that a rectangle $\hat R$ is an extension of $R$ if it is obtained by extending $R$ in one direction (up, down, left, or right).  Finally, we define the rectangle flexibility of $X$ as
\[
  \flex(X) = \sup_{(\sigma, \hat R_\sigma)} \max\left\{\frac{\slope(\sigma)}{\slope(\hat R_\sigma)}, \frac{\slope(\hat R_\sigma)}{\slope(\sigma)}\right\},
\]
where $\sigma$ is a rectangle-spanning saddle connection in $X$ and $\hat R_\sigma$ is an extension of the spanning rectangle $R_\sigma$.

\p{Finite flexibility versus the existence of horizontal and vertical saddle connections} For a singular Euclidean surface $X$ we have the following implication:
\[
\flex(X) < \infty \Longrightarrow \bar X \text{ has no horizontal or vertical saddle connection}
\]
Much of our work in this section relies on the work of Minsky and the third-named author of this paper \cite{MinskyTaylor}.  In their work on singular Euclidean surfaces, they assume the surfaces do not have any horizontal or vertical saddle connections.  As a proxy for this assumption, we will assume the condition that $X$ has finite flexibility.  It is a straightforward consequence of the definition of flexibility that this implies the non-existence of vertical and horizontal saddle connections.  The converse of the implication is also true, but we will not require it.  

A singular Euclidean structure for a pseudo-Anosov homeomorphism has no vertical or horizontal saddle connections.  Such a structure also has finite flexibility by Lemma~\ref{lemma:flexibility} (or by the converse of the above implication).  We will use this without mention below in order to apply our lemmas about singular Euclidean structures with finite flexibility to singular Euclidean structures associated to pseudo-Anosov homeomorphisms. 

While Lemmas~\ref{lemma:local-disparity}, \ref{lem:slopeflex}, and~\ref{lemma:global-disparity} hold vacuously in the case where the flexibility is infinite, we state them in the case of finite flexibility for the sake of clarity.


\subsection{Slope disparity lemmas}
\label{sec:slope_disparity}

The main goals of this section are to prove Lemmas~\ref{lemma:local-disparity} and~\ref{lemma:global-disparity}, the local and global slope disparity lemmas.  In what follows, we will write $F=F(X)$ for the quantity $2\flex(X)-1$; this quantity will appear often in what follows.  

\p{Local slope disparity} The next lemma bounds the disparity between slopes of saddle connections in a given rectangle.

\begin{lemma}
\label{lemma:local-disparity}
Let $X$ be a singular Euclidean surface with finite flexibility $F=F(X)$, let $R$ be a rectangle in $X$, and let $\sigma_0$ and $\sigma_1$ be saddle connections contained in $R$.  Then
  \begin{displaymath}
    \slope(\sigma_0) \leq F^2 \slope(\sigma_1).
  \end{displaymath}
\end{lemma}

Lemma~\ref{lemma:local-disparity} is an immediate consequence of the following lemma, which compares the slope of a rectangle to the slopes of saddle connections contained within.

\begin{lemma}
\label{lem:slopeflex}
Let $X$ be a singular Euclidean surface with finite flexibility $F=F(X)$, and let $R$ be a rectangle in $X$ and $\sigma$ a saddle connection in $R$.  Then
\[
\frac{1}{F} \leq \frac{\slope(\sigma)}{\slope(R)} \leq F.
 \]
\end{lemma}

\begin{proof}

We prove the second inequality.  The first is proved with a symmetric argument.  Let $R_\sigma^h$ be the maximal rectangle obtained by extending the spanning rectangle $R_\sigma$ horizontally in both directions.  Note that $\sigma$ connects the horizontal edges of $R_\sigma^h$. 

As $\slope(R_\sigma^h) \leq \slope(R)$ and $F=2\flex(X) -1$, it suffices to prove that
\[
\frac{\slope(\sigma)}{\slope(R_\sigma^h)} \leq 2\flex(X) -1.
\]
Since ratios of slopes are invariant under horizontal and vertical scaling, we may assume for simplicity that
\[
\height(R_\sigma) = \width(R_\sigma) = 1.
\]
Let $R_\sigma^-$ and $R_\sigma^+$ be the maximal extensions of $R_\sigma$ to the left and right.  By definition, we have
\[
\width(R_\sigma^\pm) \leq \flex(X).
\]
If $E$, $E^h$, $E^-$, and $E^+$ are the domains of $R_\sigma$, $R_\sigma^h$, $R_\sigma^-$, and $R_\sigma^+$, respectively, then $E^- \cup E^+ = E^h$ and $E^- \cap E^+ = E$.  Therefore
\begin{align*}
\frac{\slope(\sigma)}{\slope(R_\sigma^h)} &= \frac{1}{\slope(R_\sigma^h)} = \width(R_\sigma^h) \\
&= \width(R_\sigma^+) + \width(R_\sigma^-) - \width(R_\sigma) 
\leq 2\flex(X)-1,
\end{align*}
as desired.  
\end{proof}

\p{Global slope disparity} Our next goal is to prove the following lemma, which bounds the disparity between slopes of non-crossing saddle connections. In the statement, the constant $K=K(S_{g,p})$ is the number $4(3g-3+p)$ from the introduction, namely, our upper bound for the number of separatrices in a foliation on $S_{g,p}$,

\begin{lemma}
\label{lemma:global-disparity}
Let $X$ be a singular Euclidean surface with finite flexibility $F=F(X)$.  If $\sigma_0$ and $\sigma_1$ are non-crossing saddle connections in  $X$, then
\[
    \slope(\sigma_0) \leq F^{2K} \slope(\sigma_1).
\]
\end{lemma}

Our proof of Lemma~\ref{lemma:global-disparity} requires a subordinate lemma, Lemma~\ref{lemma:minskytaylor} below.  Before stating that lemma, we recall some preliminaries from the work of Minsky and the third-named author, in particular the notion of a rectangle hull; see \cite[Section 4.1]{MinskyTaylor}.  We also introduce the notion of a rectangle associated to a point.  

Let $X$ be a singular Euclidean surface with finite flexibility.  For any saddle connection $\sigma$ of $X$, let $\mathcal R(\sigma)$ be the collection of rectangles in $X$ that are maximal with respect to the property that their diagonals lie along $\sigma$.  We refer to $\mathcal R(\sigma)$ as the rectangle casing for $\sigma$.

The rectangle hull $\hull(\sigma)$ is defined to be the collection of saddle connections contained in some rectangle of $\mathcal R(\sigma)$. For example, the rectangle hull of the saddle connection in Figure~\ref{fig:enlarging-fitted-rectangles} contains the four saddle connections that join the singularities in the boundary of $R^+$ or the boundary of $R^-$.  Every rectangle hull is connected (the rectangles in $\mathcal R(\sigma)$ are ordered along $\sigma$ and consecutive rectangles share a singularity).  If $\sigma_0$ and $\sigma_1$ are non-crossing saddle connections, then the saddle connections of $\hull(\sigma_0) \cup \hull(\sigma_1)$ are also non-crossing \cite[Lemma 4.1]{MinskyTaylor}.

If $C$ is a singular point in a rectangle $R \in \mathcal R(\sigma)$ and does not lie on $\sigma$, then we define the associated rectangle $R_C$ as the rectangle in $R$ with corner $C$ and diagonal along $\sigma$.  In Figure~\ref{fig:enlarging-fitted-rectangles} we show a representative configuration.  

Let $R_C^-$ and $R_C^+$ be the rectangles of $\mathcal R(\sigma)$ obtained by sliding one of the points of $R_C \cap \sigma$ along $\sigma$.  The rectangle $R_C$ is contained in both of these.  Also it is possible that one or both is equal to $R_C$.  

The set $\mathcal R(\sigma)$ is finite, because the diagonals of the elements of $\mathcal R(\sigma)$ form a pairwise non-nested cover of $\sigma$ by closed intervals.  Also, each element of $\mathcal R(\sigma)$ contains at most one singularity on each edge, since we have assumed that $X$ has no horizontal or vertical saddle connections.  In particular, there are finitely many $R_C$ as above.  

\begin{lemma}
\label{lemma:minskytaylor}
Let $X$ be a singular Euclidean surface with finite flexibility.  For any saddle connection $\sigma$, there are rectangle-spanning saddle connections $\sigma^+,\sigma^-$ in the rectangle hull $\hull(\sigma)$ such that
\[
  \slope(\sigma^-) \leq \slope(\sigma) \leq \slope(\sigma^+).
\]
\end{lemma}

\begin{proof}

\begin{figure}[ht]
      \centering
      \def\rad{0.015}
      \begin{tikzpicture}[scale=4]
        \draw[very thick,red] (0.5,0.25+0.3) -- (0.5,0.25) -- (1.6,0.25);
        \draw[very thick,blue] (2,0.5) -- (2,1) -- (1.4,1);
        \filldraw[green!30!white] (1,0.5) rectangle ++(0.6,0.3);
        \filldraw (0,0) circle (\rad) -- (2,1) circle (\rad) node[right] {$C^-$};
        \draw (1,0.5) node[below] {} -- ++(0,0.3) node[above left] {$C$} -- ++(0.6,0) node[above] {} -- ++(0,-0.3) node[below right] {} -- ++(-0.6,0);
        \draw (1.575,0.525)  node[above left] {$R_C$} ;
        \draw[blue] (2.25,0.5)  node[above left] {$R_C^-$} ;
        \draw[red] (0.45,0.65)  node[above left] {$R_C^+$} ;
        \draw (0.30,0.125)  node[above left] {$\sigma$} ;

        \filldraw (1,0.8) circle (\rad);
        \draw[red] (1,0.8) -- (0.5,0.8) -- (0.5,0.25) -- (1.6,0.25) -- (1.6,0.5);
        \filldraw (0.5,0.41) circle (\rad) node[left] {$C^+$} ;
        \draw[blue] (1,0.8) -- (1,1) -- (2,1) -- (2,0.5) -- (1.6,0.5);
        \draw[dashed] (0.5,0.55) -- (1,0.8);
        \draw[dashed] (1,0.8) -- (1.4,1);
      \end{tikzpicture}
      \caption{The rectangles $R_C$, $R_C^+$, and $R_C^-$ from the proof of \Cref{lemma:minskytaylor}}
      \label{fig:enlarging-fitted-rectangles}
    \end{figure}
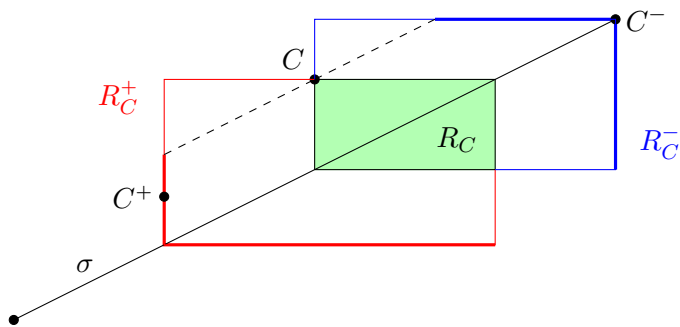

As above, there are finitely many rectangles $R_C$ associated to singularities in the rectangles of $\mathcal R(\sigma)$.  Thus there is such an $R_C$ of largest area.   

Without loss of generality, we assume that $\sigma$ passes through the bottom-left and top-right corners of $R_C$ and that $C$ lies above (or at an endpoint of) $\sigma$ (if needed, we precompose $R$ with a symmetry of its domain rectangle $E$).  Let $R_C^+$ and $R_C^-$ denote the elements of $\mathcal{R}(\sigma)$ that contain $R_C$ and are obtained by sliding, respectively, the bottom-left and top-right corners of $R_C$ along $\sigma$.  As above, it is possible for one or both of these two to equal $R_C$.  See Figure~\ref{fig:enlarging-fitted-rectangles} for an illustration.  

There must be a singularity $C^+$ on the left or bottom boundary of $R_C^+$ and also a singularity $C^-$ on the top or right boundary of $R_C^-$, for otherwise, $R_C^+$ and $R_C^-$ would not be maximal (as demanded by the definition of the rectangle hull).  

Let $\ell$ be the line segment in $R_C^+ \cup R_C^-$ that passes through $C$ and is parallel to $\sigma$.  The singularities $C^+$ and $C^-$ lie on the same side of $\ell$ as $\sigma$.  Indeed, if not, there would be a translate of $R_C^{\pm}$ (in the direction of $\sigma$) that strictly contains $R_C$, violating the maximality of the latter.  

Let $\sigma^+$ and $\sigma^-$ be the saddle connections connecting $C^+$ and $C^-$ to $C$ within $R_C^+$ and $R_C^-$.  We have by definition that $\sigma^+$ and $\sigma^-$ lie in $\hull(\sigma)$.  By the previous paragraph they satisfy $\slope(\sigma^-) \leq \slope(\sigma) \leq \slope(\sigma^+)$, as desired.
\end{proof}

We now turn towards the proof of Lemma~\ref{lemma:global-disparity}.  In the proof, we use the fact that $K(S)$, our upper bound for the number of separatrices in a foliation on $S$, is also an upper bound for the number of triangles in a triangulation of $S$ by saddle connections.  Indeed, the map associating each separatrix to the triangle it initially enters is bijective.

\begin{proof}[Proof of Lemma~\ref{lemma:global-disparity}]

We proceed in two steps, first proving the lemma in the case of rectangle-spanning saddle connections and then in the general case.

To this end, first suppose that $\sigma_0$ and $\sigma_1$ are non-crossing rectangle-spanning saddle connections in $X$.  Minsky and the third-named author proved that any collection of non-crossing, rectangle-spanning saddle connections can be extended to a triangulation by rectangle-spanning saddle connections \cite[Lemma 3.2]{MinskyTaylor}.  Consider then such triangulation of $X$. As above, the number of triangles is bounded above by $K=K(S)$.  Thus, any sequence of distinct triangles $(T_i)$ has length at most $K$.  It further follows that there is a sequence of edges of the triangulation $e_0 = \sigma_0,e_2,\dots,e_n=\sigma_1$ with $n \leq K$ and with each consecutive pair contained in a triangle.  Since each triangle is contained in a rectangle, the special case now follows from \Cref{lemma:local-disparity}. 

We now proceed to the general case. Let $\sigma_0$ and $\sigma_1$ be arbitrary non-crossing saddle connections. For $i=0,1$ let $\sigma_i^{\pm}$ be the rectangle-spanning saddle connections given by Lemma~\ref{lemma:minskytaylor}.  Minsky and Taylor proved that if two saddle connections are non-crossing then the elements of their rectangle hulls are pairwise non-crossing \cite[Lemma 4.1]{MinskyTaylor}.  Therefore, by Lemma~\ref{lemma:minskytaylor} and the special case of this lemma, we have
\[
\frac{\slope(\sigma_0)}{\slope(\sigma_1)} \leq \frac{\slope(\sigma_0^+)}{\slope(\sigma_1^-)} \leq F^{2K},
\]
as desired.  
\end{proof}


\subsection{Slopes and stretch factors}
\label{sec:slope_stretch}

In this section we incorporate the slope disparity lemmas together with basic properties of pseudo-Anosov maps in order to prove Proposition~\ref{prop:boundedslope}.  We begin with a comparison of rectangle flexibilities and stretch factors.  

\p{Flexibility versus stretch factor}  The following lemma bounds the rectangle flexibility of a singular Euclidean surface corresponding to a pseudo-Anosov map in terms of the stretch factor. Again, $K=K(S_{g,p})=4(3g-3+p)$.

\begin{lemma}
\label{lemma:flexibility}
  Let $S=S_{g,p}$ and let $f \colon S \to S$ be pseudo-Anosov with stretch factor $\lambda$ and singular Euclidean structure $X$.  Then
  \begin{displaymath}
    \flex(X) \le \lambda^{K}.
  \end{displaymath}
\end{lemma}

  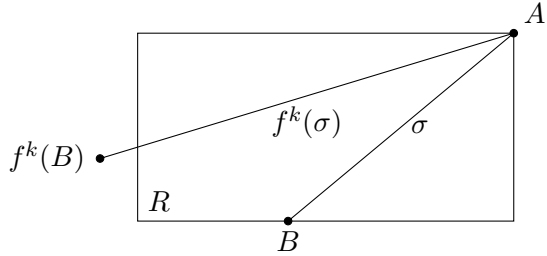
\begin{figure}[ht]
    \centering
    \begin{tikzpicture}[scale = 2.5]
      \draw (0,0) node[above right] {$R$} -- (2,0) -- (2,1) node[above right] {$A$} --  (0,1) -- cycle;
      \filldraw (2,1) circle (0.02) -- node[right] {$\sigma$} (0.8,0) circle (0.02) node[below] {$B$};
       \filldraw (2,1) -- node[below] {$f^k(\sigma)$} (-0.2,0.333) circle (0.02) ; 
               \draw (-0.225,0.20)  node[above left] {$f^k(B)$} ;

    \end{tikzpicture}
    \caption{The saddle connections $\sigma$ and $f^k(\sigma)$ from the proof of Lemma~\ref{lemma:flexibility}}
    \label{fig:stretch-from-corner}
  \end{figure}

\begin{proof}

Let $\hat R_\sigma$ be an extension of a spanning rectangle for a saddle connection $\sigma$ in $X$.  And let us denote $\hat R_\sigma$ by $R$.  By symmetry, we may assume that $R$ is a horizontal extension, as in \Cref{fig:stretch-from-corner} (for the other case we replace $f$ with $f^{-1}$).  Say that the endpoints of $\sigma$ are the singularities $A$ and $B$.  It must be the case that one of these, say $A$, is at a corner of $R$.  The singularity $B$ must then lie on a horizontal edge of $R$.  We need to show two inequalities:
\[
  \frac{\slope(R)}{\slope(\sigma)}  \le \lambda^{K}  \quad \text{and} \quad \frac{\slope(\sigma)}{\slope(R)} \leq \lambda^K.
 \]
The first inequality is immediate since $\slope(R) \le \slope(\sigma)$ and $\lambda > 1$.  We proceed to the second inequality. Choose $k \leq K$ so that $f^k$ fixes $A$ as well as the separatrices of the horizontal and vertical foliations at $A$. Since $R$ contains no singularities, $f^k(B)$ cannot lie in the interior of $R$. Thus
\[
\width(R) \le \width(f^k(\sigma)).
\]
On the other hand, we have
\[
\height(R) =  \height(\sigma) = \lambda^{k}\height(f^k(\sigma)).
\]
Dividing the first and last expressions here by the expressions in the previous inequality, we obtain
\[
 \slope(R)\ge   \lambda^{k} \slope(f^k(\sigma)).
\]
Since $\slope(\sigma) = \lambda^{2k} \slope(f^k(\sigma))$, we have
\[
\frac{\slope(\sigma)}{\slope(R)} \le \frac{\lambda^{2k}\slope(f^k(\sigma))}{\lambda^k\slope(f^k(\sigma))} = \lambda^k \le \lambda^{K},
\]
as desired.
\end{proof}

\p{Penner's lower bound} Penner proved the following lemma \cite{Penner91}. 

\begin{lemma}[Penner]\label{lem:penner} If $f \in \Mod(S)$ then
\[
2 \le \lambda^K.
\]
where $\lambda$ is the stretch factor of $f$ and $K=K(S)$.
\end{lemma}

Penner stated the equivalent formulation $\log \lambda(f) \geq (\log 2)/K$.  In the proof of Proposition~\ref{prop:boundedslope} we will use the inequality $4^K \leq \lambda^{2K^2}$, obtained from Penner's inequality by squaring both sides and raising to the $K$th power. 

\p{Completing the proof} We are now ready for the proof of Proposition~\ref{prop:boundedslope}, which gives an upper bound on the diameter of a set of non-crossing saddle connections.

\begin{proof}[Proof of Proposition~\ref{prop:boundedslope}]

As in the statement of the proposition, $S$ is a surface, $f \colon S \to S$ is pseudo-Anosov with stretch factor $\lambda$, $X$ is an associated singular Euclidean surface, and $Z$ is a collection of non-crossing saddle connections.  Also recall from the introduction that $\Qcurve = \Qcurve(S) = 2K^2$.  Let $F=F(X)$ be the flexibility.  The conclusion of the proposition is equivalent to the inequality
\[
\frac{\max \slope(Z)}{\min \slope(Z)} \leq (\lambda^2)^{\Qcurve}.
\] 
We proceed to the proof of this inequality.  Let $\sigma_0$ and $\sigma_1$ be any two saddle connections of $Z$.   We have
\begin{align*}
\slope(\sigma_0)
    &\le \slope(\sigma_1)F^{2K}
        &&\text{(Lemma~\ref{lemma:global-disparity})}\\
    &=  \slope(\sigma_1)(2\flex(X)-1)^{2K}
        &&\text{(definition of $F$)}\\
    &\le \slope(\sigma_1)(2\lambda^K-1)^{2K}
        &&\text{(Lemma~\ref{lemma:flexibility})}\\
    &< \slope(\sigma_1)4^K\lambda^{2K^2}
        &&\text{(}2\lambda^K-1 < 2\lambda^K\text{)}\\
    &\le \slope(\sigma_1)(\lambda^2)^{2K^2}
        &&\text{(Lemma~\ref{lem:penner})}
\end{align*}
as required.
\end{proof}

\subsection{Application to stable translation lengths in the curve graph}\label{sec:GT} As above, Proposition~\ref{prop:boundedslope} gives an upper bound of $\Qcurve$ to the diameter of the horizontality of a collection of pairwise non-crossing saddle connections.  This proposition has the following corollary.   In the statement, $\C(S)$ is the curve graph for $S$ and $\tau_{C(S)}(f)$ is the stable translation length for the action of a pseudo-Anosov $f$ on $\C(S)$.  The second statement of the corollary is a (slight) quantitative strengthening of a special case of a result of Gadre--Tsai \cite[Theorem 1.1]{GadreTsai}.

\begin{corollary}\label{cor:Lip}
Let $S = S_{g,p}$ and let $f : S \to S$ be pseudo-Anosov.  The horizontality function defines a $\Qcurve$--coarse Lipschitz map
\[
\H_f \colon \C(S) \to \R.
\]
Hence, for any pseudo-Anosov $f \in \Mod(S)$, we have
\[
\tau_{\C(S)}(f) \geq \frac{1}{\Qcurve}.
\]
\end{corollary}

For closed surfaces, Gadre--Tsai proved that for $g \geq 2$ and any pseudo-Anosov $f \in \Mod(S_g)$ we have
\[
\tau_{\C(S_g)}(f) \geq  \frac{1}{162(2g-2)^2+30(2g-2)}.
\]
Since $\Qcurve = 2K^2 < 162(2g-2)^2+30(2g-2)$ for $g > 1$, our Corollary~\ref{cor:Lip} is a strengthening of this.  For surfaces with punctures, Gadre---Tsai give a lower bound that is stronger than the one given by our corollary.


\section{Approximating rectangles}
\label{sec:rec}

The goal of this section is to prove Lemma~\ref{lem:covering_rectangles}.  This is the most technical result in the paper.  Let $X$ be a singular Euclidean surface and let $\wh X$ be its completed universal cover, as in Section~\ref{section:slope}.  What Lemma~\ref{lem:covering_rectangles} tells us is that if $\gamma_1$ and $\gamma_2$ are geodesics in $\wt X$ that cross each other and also cross a singularity-free horizontal leaf $\ell$, then there is a maximal rectangle $R$ in $\wh X$ with the following properties:
\begin{itemize}
\item $\gamma_1$ and $\gamma_2$ cross the horizontal sides of $R$, and
\item any saddle connection in $X$ with sufficiently small slope has a path lift that crosses the vertical sides of $R$.
\end{itemize}
Before stating Lemma~\ref{lem:covering_rectangles}, we introduce the notion of a $k$-approximating rectangle.  

\p{Approximating rectangles} As above let $X$ be a singular Euclidean surface with $\flex(X) < \infty$ and let $\wh X$ be its  completed universal cover.  Let $\gamma_1$ and $\gamma_2$ be intersecting geodesics in $\wh{X}$ whose endpoints (if any) are singular points of $\wh{X}$ (all points of $\wh{X} \setminus \wt X$ are singular). 

Let $\ell$ be a bi-infinite horizontal leaf in $\wh{X}$ that intersects both $\gamma_1$ and $\gamma_2$. We say a rectangle $R$ in $\wh{X}$ is a \emph{$k$-approximating rectangle} for the triple $(\gamma_1,\gamma_2,\ell)$  if
\begin{itemize}
\item $R$ is maximal (i.e. it contains a singularity in each edge), 
\item $\ell$ contains an interior horizontal leaf of $R$,
\item $\ga_1$ and $\ga_2$ connect the horizontal edges of $R$, and
\item  $k \slope(R) \ge \min \left\{ \slope(\gamma_1) \cup  \slope(\gamma_2) \right \}$.  
\end{itemize}

As in Section~\ref{section:slope}, we write $F(X)$ for the quantity $2\flex(X)-1$ when $X$ is a singular Euclidean surface.  Also, in what follows, by a lift of a saddle connection will mean a path lift. 

\begin{lemma}
\label{lem:covering_rectangles}
Let $X$ be a singular Euclidean surface with $\flex(X) < \infty$.  Let $\gamma_1$ and $\gamma_2$ be intersecting geodesics in $\wh{X}$ whose endpoints (if any) are singular points of $\wh{X}$.  Let $\ell$ be a singularity-free horizontal leaf in $\wh{X}$ that intersects both $\gamma_1$ and $\gamma_2$.  Let $F=2\flex(X)-1$ and let $K=K(X)$ be the separatrix complexity.  Then the following hold.
\begin{enumerate}
\item\label{cr1} There exists an $F^5$-approximating rectangle $R$ for $(\gamma_1,\gamma_2,\ell)$.
\item\label{cr2} If $R$ is an $F^5$-approximating rectangle for $(\gamma_1,\gamma_2,\ell)$ and $Z$ is a collection of pairwise non-crossing saddle connections in $X$ with
\[
\min \slope(Z) < F^{-(2K+6)} \slope(\gamma_i) \text{ for } i \in \{1,2\}
\]
then there is a $\sigma \in Z$ and a lift $\tilde \sigma$ that crosses the vertical edges of $R$.  
\end{enumerate}
\end{lemma}


\subsection{The existence of approximating rectangles} For the proof of the first statement of \Cref{lem:covering_rectangles}---on the existence of approximating rectangles---we introduce three ingredients: a basic fact about triangles in singular Euclidean surfaces, developing maps, and a graph of rectangles and rectangle-spanning saddle connections.

\p{Triangles in singular Euclidean surfaces} In the proof of \Cref{lem:covering_rectangles}, we will use that fact that a triangle in a (simply connected) singular Euclidean surface---a region bounded by three geodesics---has no singularities in its interior.  This follows readily from the Gauss--Bonnet formula for singular Euclidean surfaces with geodesic boundary, namely:
\[
2\pi \chi(Y) = -\pi \left (\sum_{p_k} \mathrm{deg}(p_k) \right ) + \sum_{i}(\pi - \theta_i),
\]
where the $\theta_i$ are the interior angles on the boundary and each $\mathrm{deg}(p_k) \geq 1$ is the degree of the interior singularity $p_k$ (the degree is two less than the number of prongs); see \cite[Theorem 3.3]{Rafi}.

\p{Developing maps} Let $X$ be a singular Euclidean surface.  By a disk in $\wh X$ we mean an isometric embedding of a closed polygonal disk in the Euclidean plane.  Examples of disks are (the images of) rectangles.  It follows from the Gauss--Bonnet formula that a triangle (as in the previous paragraph) is a disk in this sense.

Suppose we have a collection of disks $D_0,D_1,\dots$ in $\wh X$.  Suppose that each $D_i$ with $i > 0$ has the property that its intersection with $D_0 \cup \cdots \cup D_{i-1}$ is connected and contains a nonempty open set.  In this case there is a developing map from $\cup D_i$ to the Euclidean plane that respects horizontal and vertical directions and measures.  This map is defined inductively.  First we make a choice of embedding of $D_0$ in the plane; this choice is unique up to translation, rotation by $\pi$, and reflection through a horizontal or vertical line.  Then each subsequent $D_i$ is mapped to the plane in a way that agrees with the intersection with $D_0 \cup \cdots \cup D_{i-1}$.  Given a choice of embedding of $D_0$, the rest of the developing map is determined.  

In the proof of Lemma~\ref{lem:covering_rectangles}, we will use the (almost) uniqueness of the developing map as a way of relating certain configurations of triangles and rectangles in a singular-Euclidean surface to a configuration of disks in the Euclidean plane.  Once we fix a developing map it makes sense to use words like ``above'' and ``to the left'' when describing features of the configuration of disks.  

\p{The graph of rectangles and rectangle-spanning saddle connections} Let $X$ be a singular Euclidean surface.  We define $\G_X$ to be the graph whose vertices are the rectangles in $X$ and the rectangle-spanning saddle connections in $X$.  We connect a rectangle-spanning saddle connection $\sigma$ to a rectangle $R$ if $\sigma$ is contained in $R$.  It follows immediately from \Cref{lem:slopeflex} that if $v_0,\dots,v_n$ is a path in $\G_X$ then
\[
F^n \slope_X(v_0) \geq \slope_X(v_n),
\]
In the proof of Lemma~\ref{lem:covering_rectangles} we will bound the slope of the approximating rectangle by constructing a path in the graph $\G_{\wh X}$ associated to $\wh X$.  We will take $v_0$ to be the approximating rectangle $R$ and 
$v_n$ to be a saddle connection of $\gamma_1$ or $\gamma_2$ (or a saddle connection with slope greater than some such saddle connection).

\begin{proof}[Proof of Lemma~\ref{lem:covering_rectangles}(\ref{cr1})]

We break the proof into two main parts.  In the first part we construct a rectangle $R$ satisfying the first three conditions of the definition of an approximating rectangle, and in the second part we show that $R$ satisfies the slope condition.  We begin with some setup.   Throughout, the reader may refer to Figure~\ref{fig:separating-rectangle}. 

\p{Setup.}  Denote the unique points of intersection $\ell \cap \gamma_1$ and $\ell \cap \gamma_2$ by $x_1$ and $x_2$ (these points are unique because two geodesics must intersect in a connected union of saddle connections and because $\ell$ contains no singularities).  The geodesics $\gamma_1$, $\gamma_2$, and $\ell$ bound a triangle $\Delta$.  As above the interior of $\Delta$ is free of singularities.  We allow for the case where $x_1=x_2$; in this case $\Delta$ is degenerate, namely, it is the single point $x_1$.  In what follows, we assume that we are in the generic case where $x_1 \neq x_2$, and so $\Delta$ has nonempty interior.  

As per Figure~\ref{fig:separating-rectangle} we may choose a developing map for $\Delta$ so that $\Delta$ lies above $\ell$.  We may also assume that $\gamma_1$ and $\gamma_2$ lie to the left and right of $\Delta$, respectively.  These two choices determine the developing map up to translation.  

Below, we will define rectangles $R_1,\dots,R_5$, and at each stage it will be true that the union satisfies the criterion for the existence of a developing map, which is at each stage uniquely defined as an extension of the developing map defined on $\Delta$.  The images of the $R_i$ and $\Delta$ under the developing map (at least in a generic case) are shown in Figure~\ref{fig:separating-rectangle}.

\newcommand{\BaseMiniPicture}[1]{%
  \path[use as bounding box] (0.4,-1.3) rectangle (2.5,1.0);
  
  \draw[thin] (1.8,-1.15) rectangle (2,0.85);    
  \draw[thin] (2,0.3) rectangle (1.1,-0.9);      
  \draw[thin] (1.1,-0.4) rectangle (2.15,-0.9);  
  \draw[thin] (1.8,-.9) rectangle (2,0.65);      
  \draw[thin] (1.6,0.4) rectangle (2.1,0.65);    
  \draw[thin] (1.7,-0.6) rectangle (2.3,-0.9);   

  \draw[thin] (.55,.2) -- (.65,-.15);
  \draw[thin] (.65,-.15) -- (1,0.1) -- (2.2,0.7);
  \draw[thin] (.65,-.15) -- (0.9,-0.2) -- (2.4,-0.95);

\draw[thin,dash pattern=on 1.5pt off 1.5pt] (1.9,-1.2) -- (1.9,1.2);
  #1
}
  
\begin{figure}[ht]
  \centering

  \begin{minipage}{0.45\linewidth}
    \centering
    \begin{tikzpicture}[scale=4,rotate=270]
      \def\rad{0.02}
      \filldraw[fill=red!30!white,opacity=0.5,dotted,line width=0pt] (2,0.3) rectangle (1.1,-0.9);
      \filldraw[fill=blue!20!white,opacity=0.3,dotted,line width=0pt] (1.1,-0.4) rectangle (2.15,-0.9);
      \filldraw[fill=green!30!white] (1.8,-1.15) rectangle (2,0.85);
      \draw[line width=.5pt,fill=green!30!white] (1.8,-.9) rectangle (2,0.65);
      \filldraw (1.9,0.55) circle (\rad/2);
      \draw[line width=.5pt] (2,0.3) rectangle (1.1,-0.9);
      \draw[line width=.5pt] (1.1,-0.4) rectangle (2.15,-0.9);
      \node[left] at (1.95,0.6) {$x_2$};
      \filldraw (1.9,-0.7) circle (\rad/2);
      \node[left] at (1.95,-0.57) {$x_1$};
      \draw(1.95,-1.15) circle (\rad); 
      \draw[dashed] (1.9,-1.2) -- ++(0,2.4) node[right] {$\ell$};
       \draw (1,0.1) ++ (0.6,0.3) rectangle +(0.5,0.25) +(0.5,0) circle (\rad) + (0.5,0.1) circle (\rad);
      \draw (0.9, -0.2) -- ++(0.8,-0.4) rectangle +(0.6,-0.3) +(0.6,-0.3)  circle (\rad) node[above left] {$y_1$} +(0.4,0);
      \draw[color=gray,line width=.5pt] (.45,.75) -- node[above] {$\gamma_1$} (.55,.2);
      \draw[color=gray,line width=.5pt] (.45, -0.4) -- node[above] {$\gamma_2$} (.55,.2) -- (.65,-.15);
      \draw (.65,-.15) circle (\rad) -- node[right] {$\gamma_2$} (1,0.1) circle (\rad) -- ++(1.2, 0.6);
      \draw (.65,-.15) -- node[left] {$\gamma_1$} (0.9, -0.2) -- ++(1.5, -0.75);
      \draw (1.85, .85) circle (\rad); 
      \draw (2,0.3) circle (\rad) node[below] {$y_2$};
      \draw (1.1, -0.4) circle (\rad) node[above] {$y_3$}; 
      \draw (2.15, -0.54) circle (\rad);
      \node[left] at (2.22,-0.425) {$y_4$};
      \draw (2.3, -0.6) circle (\rad);
    \end{tikzpicture}
  \end{minipage}%
  \hfill
\begin{minipage}{0.225\linewidth}
  \scriptsize   
\vspace{-1.15in}

\hspace*{.5in}
  \begin{tikzpicture}[scale=0.75,rotate=270]
    \BaseMiniPicture{%
      \draw[line width=1pt] (1.6,0.4) rectangle (2.1,0.65);
      \draw[line width=1pt] (1.7,-0.6) rectangle (2.3,-0.9);
      \node at (1.35,.75) {$R_2$};
      \node at (1.45,-1.25) {$R_1$};
    }
  \end{tikzpicture}
  
\hspace*{.5in}
  \begin{tikzpicture}[scale=0.75,rotate=270]
    \BaseMiniPicture{%
      \draw[line width=1pt] (2,0.3) rectangle (1.1,-0.9);
      \node at (1.15,-1.25) {$R_4$};
    }
  \end{tikzpicture}

\vspace*{-.125in}
\hspace*{.5in}
  \begin{tikzpicture}[scale=0.75,rotate=270,baseline=(current bounding box.north)]
    \BaseMiniPicture{%
  \draw[line width=1pt] (.65,-.15) -- (0.9,-0.2) -- (1.9,-0.7) -- (1.9,0.6) -- (1,0.1) -- (.65,-.15);
      \node at (1.15,.75) {$\Delta$};
    }
  \end{tikzpicture}

  \end{minipage}
\begin{minipage}{0.225\linewidth}
  \scriptsize   
\vspace{-1.145in}

\hspace*{.25in}
  \begin{tikzpicture}[scale=0.75,rotate=270]
    \BaseMiniPicture{%
      \draw[line width=1pt] (1.8,-.9) rectangle (2,0.65);
      \node at (2.25,-0.1) {$R_3$};
    }
  \end{tikzpicture}

\hspace*{.25in}
  \begin{tikzpicture}[scale=0.75,rotate=270]
    \BaseMiniPicture{%
      \draw[line width=1pt] (1.1,-0.4) rectangle (2.15,-0.9);
      \node at (1.2,-1.25) {$R_5$};
    }
  \end{tikzpicture}

\hspace*{.25in}
  \begin{tikzpicture}[scale=0.75,rotate=270]
    \BaseMiniPicture{%
      \draw[line width=1pt] (1.8,-1.15) rectangle (2,0.85);    
      \node at (1.5,.95) {$R$};
    }
  \end{tikzpicture}
\end{minipage}

  \caption{The polygons in the proof of Lemma~\ref{lem:covering_rectangles}. The right column shows a key highlighting the individual rectangles and the triangle}
  \label{fig:separating-rectangle}
\end{figure}

\p{Part 1: Constructing $R$} In order to construct $R$, we first construct three intermediate rectangles, $R_1$, $R_2$, and $R_3$.

For each $i \in \{1,2\}$, we define $R_i$ to be a rectangle in $\mathcal R(\gamma_i)$ satisfying the property that $R_i$ contains singularities on either side of $\ell$.  To explain why such rectangles exist, we fix some $i \in \{1,2\}$.  Since $\ell$ intersects $\gamma_i$, it crosses some saddle connection $\sigma_i$ of $\gamma_i$.  As $\ell$ contains no singular points, the two singularities of $\sigma_i$ lie on opposite sides of $\gamma_i$.  The rectangle hull $\hull(\sigma_i)$ is a connected graph in $\wh X$ containing the two singular points of $\sigma_i$.  Therefore, there is a saddle connection $\rho_i$ of $\hull(\sigma_i)$ that connects singular points on opposite sides of $\ell$.  The rectangle $R_i \in \mathcal R(\sigma_i)$ associated to $\rho_i$ satisfies the desired property.

We note that Figure~\ref{fig:separating-rectangle} does not accurately reflect the case where $R_1 \cap R_2 \neq \emptyset$.  But our arguments still apply in this case.

We now proceed to construct the rectangle $R_3$.  Let $\ell_0$ be the minimal segment of $\ell$ containing $R_1\cap \ell$ and $R_2 \cap \ell$. Then let $R_3$ be the rectangle that is minimal with respect to the following properties: 
\begin{itemize}
\item $\ell_0$ is a horizontal leaf of $R_3$ and
\item $R_3$ contains singularities on both its top and bottom edges.
\end{itemize}
The rectangle $R_3$ has a singularity on both of its horizontal edges (possibly at the corners), but not in the interiors of its vertical edges.  Because $R_1$ and $R_2$ have singularities on both sides of $\ell$, the endpoints of both vertical edges of $R_3$ are contained in $R_1 \cup R_2$, as in Figure~\ref{fig:separating-rectangle}.

Finally, let $R$ be the rectangle that is maximal with respect to the property that the horizontal edges of $R_3$ are contained in $R$.  In other words, $R$ is the maximal horizontal extension of $R_3$. 

To complete the first part of the proof, we check that $R$ satisfies the first three conditions in the definition of an approximating rectangle.
\begin{itemize}
\item Since $R_3$ contains singularities on both of its horizontal edges and $R$ is a maximal horizontal extension of $R_3$, we have that $R$ is maximal.
\item Since $R_3$ contains a segment of $\ell$ in its interior and $R$ contains $R_3$, we have that $R$ contains a segment of $\ell$ in its interior.
\item Since $R_1$ and $R_2$ both contain singularities on both sides of $\ell$, the vertical edges of $R_3$ are contained in the vertical edges of $R_1$ and $R_2$.  Then, since the diagonals of $R_1$ and $R_2$ are contained in $\gamma_1$ and $\gamma_2$, respectively, it follows that $\gamma_1$ and $\gamma_2$ cross the horizontal edges of $R_3$, hence $R$, as desired.
\end{itemize}
This completes the first part of the proof.

\p{Part 2: Estimating the slope} Our strategy for estimating the slope of $R$, and hence verifying the fourth condition of a $k$-approximating rectangle, is to find a path in $\G_{\wh X}$ from $R$ to a saddle connection that has slope greater than or equal to the slope of some saddle connection of $\gamma_1$ or $\gamma_2$.  Since the goal is to show that $R$ is an $F^5$-approximating rectangle, it suffices (as above) to find such a path whose length is bounded above by 5.  To this end, we start by defining singular points $y_1$, $y_2$, $y_3$, and $y_4$ and rectangles $R_4$ and $R_5$; each of the saddle connections we use in our path will be of the form $y_iy_j$.

\begin{itemize}
\item Let $y_1$ and $y_2$ be the singularities on the horizontal edges of $R_3$ that lie above and below $\ell$, respectively. 
\end{itemize}
These exist by the definition of $R_3$, and they are unique since $\wh{X}$ has no horizontal saddle connections.  

Since $\Delta$, $R_1$, and $R_2$ have no singularities in their interiors, it follows that $y_1$ lies in the boundary of $R_1$ or $R_2$.  Up to relabeling the $\gamma_i$, we may assume that it lies in the boundary of $R_1$, as in Figure~\ref{fig:separating-rectangle}.  (Assuming $R_1 \neq R_2$ we have no further flexibility in the choice of developing map.)  Because $\Delta$ has no singularities in its interior, $y_1$ must lie on the left of $R_1$ (as in Figure~\ref{fig:separating-rectangle}) or the top of $R_1$.  

\begin{itemize}
\item Let $R_4$ be the rectangle that is maximal with respect to the properties that it contains $y_2$ at its bottom-right corner and $y_1$ along its left edge.
\item Let $y_3$ be the singularity on the top edge of $R_4$.
\item Let $R_5$ be the rectangle that is maximal with respect to the properties that it contains $y_3$ at its top-right corner and also contains $y_1$ along its left edge.
\item Let $y_4$ be the singularity on the bottom edge of $R_5$.  
\end{itemize}
In other words, $R_4$ is obtained from the rectangle spanned by the saddle connection $y_1y_2$ and extending upward as far as possible and $R_5$ is obtained from the rectangle spanned by $y_1y_3$ by extending it downward as far as possible.  

To complete the proof we treat three cases.  In each case we will choose a particular saddle connection $\sigma$ of one of the $\gamma_i$.  We will then find a saddle connection $y_iy_j$ or a rectangle $R_i$ with slope greater or equal to that of $\sigma$ and with distance at most 5 from $R$ in $\G_{\wh X}$.  

\p{Case 1: $y_3$ lies to the left of $\gamma_1$}  For this case let $\sigma$ be the saddle connection of $\gamma_1$ that crosses $\ell$.  There are two slightly different subcases, according to whether $y_1$ lies on the left-hand edge or the top edge of $R_1$.  We begin by treating the former case in detail and then explaining the minor changes needed for the latter case.  

In both subcases, it follows from the definition of $R_4$ that $y_1$ and $y_3$ lie on the left-hand and top edges of $R_4$, respectively.  It then follows that $y_1$ lies on the left-hand edge of $R_5$, and that $y_3$ is the top-right corner of $R_5$.  It must be that $y_3$, and hence the top edge of $R_5$, lies outside $\Delta$.  In particular, it must be that $\sigma$ intersects the right-hand edge of $R_5$.  It also follows from the definitions that $y_4$ lies on the bottom edge of $R_5$.  

For the case where $y_1$ lies on the left-hand edge of $R_1$, we claim that the following hold:
\begin{itemize}
\item $\sigma$ intersects the right-hand edge of $R_5$
\item $y_3$ is the top-right corner of $R_5$, 
\item $y_4$ lies on the bottom edge of $R_5$, 
\item $\sigma$ intersects the bottom edge of $R_5$,
\item along the bottom edge of $R_5$, $y_4$ lies to the right of $\sigma$.
\end{itemize}
We already verified the first three items.  Since $R_4$ contains no singularities in its interior and since $R_1$ contains a singularity below $\ell$, it must be that the bottom of $R_5$ passes through $R_1$, which gives the fourth item.  As $R_1$ contains no singularities in the interior, and has no singularities on its left edge besides $y_1$, the fifth and final item follows.

It follows from the claim that 
\[
\slope(y_3y_4) \geq \min \slope(\gamma_1).
\]
The path 
\[
R, y_1y_2, R_4, y_1y_3,R_5,y_3y_4
\]
in $\G_{\wh X}$ has length 5, and so the lemma follows in the case where $y_1$ lies on the left-hand edge of $R_1$.  

For the case where $y_1$ lies in the (interior of the) top edge of $R_1$, the proof is very similar.  The changes are that the left-hand edge of $R_5$ does not overlap the left-hand edge of $R_4$, and that the bottom edge of $R_5$ can be lower than the bottom edge of $R_1$ (this can happen when the singularity of $R_1$ below $\ell$ lies to the left of $R_5$).  The fourth and fifth items above become:
\begin{itemize}
\item $\sigma$ intersects the bottom edge or the left-hand edge of $R_5$,
\item if $\sigma$ intersects the bottom edge of $R_5$, $y_4$ lies to the right of $\sigma$ along this edge.
\end{itemize}
We can think of the case where $\sigma$ crosses the left-hand edge of $R_5$ as being a case where $\sigma$ crosses the bottom edge of $R_5$ very far to the left, then these two items are morally equivalent to their counterparts in the first subcase.  As such, we still obtain that $\slope(y_3y_4) \geq \min \slope(\gamma_1)$, and the proof concludes as in the first subcase.  

\p{Case 2: $y_3$ lies to the right of $\gamma_2$, $y_2$ lies to the right of $\gamma_2$} For this case let $\sigma$ be the saddle connection of $\gamma_2$ that crosses $\ell$.  We claim that in this case $y_2y_3$ is a saddle connection contained in $R_2$.  Since $R_2$ contains no singularities in its interior, it follows from the definition of $y_2$ that $y_2$ must lie along the right-hand edge of $R_2$.  It then follows from the maximality condition on $R_2$ and the fact that $\Delta$ contains no singularities in its interior that $R_2$ has a singularity along its top edge.  From the definition of $y_3$, the fact that $R_2$ contains no singularities in its interior, and the assumption that $y_3$ lies to the right of $\gamma_2$, it must be that the singularity on the top edge of $R_2$ is nothing other than $y_3$.  The claim follows.

Given the claim, we have the following path in $\G_{\wh X}$:
\[
R, y_1y_2, R_4, y_2y_3, R_2.
\]
Since $\slope(R_2)$ is equal to the slope of the saddle connection of $\gamma_2$ comprising the diagonal of $R_2$, the lemma follows in this case.

\p{Case 3: $y_3$ lies to the right of $\gamma_2$, $y_2$ lies to the left of $\gamma_2$} For this case let $\sigma$ be the saddle connection of $\gamma_2$ that crosses $R_4$. We claim that
\[
\slope(y_2y_3) \geq \slope(\sigma).
\]
To see this, we focus attention on the rectangle $R_4$.  By construction, $\gamma_2$ crosses $R_4$ along its top and right-hand edges.  Also, $y_2$ is the bottom-right corner of $R_4$ and $y_3$ lies along the top edge of $R_4$ to the right of $\gamma_2$ (by the way Case~3 is defined).  The claim follows.

We have the following path in $\G_{\wh X}$:
\[
R, y_1y_2, R_4, y_2y_3.
\]
Applying the claim, the lemma follows in this case.
\end{proof}


\subsection{Small-slope geodesics and approximating rectangles} The second statement of \Cref{lem:covering_rectangles} is a consequence of the first part of the lemma and the following auxiliary lemma, which uses the tools of Section~\ref{section:slope}.  

\begin{lemma} \label{lem:rectangles}
Let $X$ be a singular Euclidean surface, let $R$ be a maximal rectangle in $\wh X$, let $F=2\flex(X)-1$, let $K=K(X)$ be the separatrix complexity, and let $Z$ be a collection of pairwise non-crossing saddle connections in $X$.  Assume that
\[
\min \slope(Z) < F^{-(2K+1)} \slope(R).
\]
Then there is a $\sigma \in Z$ and a lift $\wt \sigma \subseteq \wh X$ that crosses the vertical edges of $R$.  
\end{lemma}

\begin{proof}

It suffices to consider the case where $Z$ is a singleton $\{\sigma\}$.  Let $\wt \sigma_-$ be the saddle connection in $R$ of smallest slope, that is, the one joining its vertical edges, and let $\sigma_-$ be its image in $X$.  Let $R_-$ be the rectangle spanned by $\wt \sigma_-$.  The vertical edges of $R_-$ are contained in those of $R$. 

By Lemma~\ref{lemma:minskytaylor} there is a saddle connection $\rho$ in $\hull(\sigma)$ that has slope bounded above by $\slope(\sigma)$.  In particular, $\slope(\rho) \leq F^{-(2K+1)} \slope(R)$.  By \Cref{lem:slopeflex}, $\slope(\rho) \leq F^{-2K} \slope(\sigma_-)$.  Then by Lemma~\ref{lemma:global-disparity}, $\rho$ crosses $\sigma_-$.  

Let $\wt \rho$ be a lift of $\rho$ to $\wh X$ that crosses $\wt \sigma_-$.  Since (as in the previous paragraph) $\slope(\rho) \leq \slope(\sigma_-)$, since $\wt \rho$ spans a rectangle, it must be that $\wt \rho$ crosses the vertical edges of $R_-$.  

Choose a lift $\wt \sigma$ of $\sigma$ with $\wt \rho \in \hull(\wt \sigma)$.  By the definition of $\hull(\wt \sigma)$, there is a rectangle $R_{\wt \rho}$ in $\wh X$ that contains $\wt \rho$ and has a segment of $\wt \sigma$ as its diagonal.   

 Since $R_{\wt \rho}$ contains no singularities in its interior and $\wt \rho$ crosses the vertical edges of $R_-$, the diagonal of $R_{\wt \rho}$ must also cross the vertical edges of $R_-$. But this diagonal is contained in $\wt \sigma$ and so $\wt \sigma$ crosses the vertical edges of $R_-$ and hence of $R$. This completes the proof.
\end{proof}

\begin{proof}[Proof of Lemma~\ref{lem:covering_rectangles}(\ref{cr2})] By the assumptions in the statement and the definition of a $k$-approximating rectangle we have
\[
\slope(Z) < F^{-(2K+6)} \min \{ \slope(\gamma_1) \cup \slope(\gamma_2)\} \leq F^{-(2K+1)} \slope(R).
\]
An application of Lemma~\ref{lem:rectangles} completes the proof.
\end{proof}


\section{The forcing lemma}
\label{sec:forcing}

The goal of the section is to prove Proposition~\ref{prop:forcing}, which we call the forcing lemma.   Roughly, it tells us that if we have a pseudo-Anosov mapping class $f$ (hence a notion of slope) and a train track $\tau$, then $\tau$ carries the curves of very small slope if and only if it carries the horizontal foliation for $f$. 

For the statement, we recall some basic definitions regarding train tracks.  Let $\tau$ be a train track in a surface $S$.  We say that $\tau$ is filling if each of its complementary regions in $S$ is a disk with at most one puncture; each such region can be regarded as a polygon.  An extension of $\tau$ is a train track $\tau'$ with $\tau$ a subtrack.  If $\tau'$ is obtained from $\tau$ by adjoining diagonals of complementary polygons, then we say that $\tau'$ is a diagonal extension. Next, we say that $\tau$ is transversely recurrent if for every branch $b$ of $\tau$ there is a curve $c$ in $S$ so that $c$ is transverse to $b$ and forms no bigons with $\tau$.  As in the introduction, if $c$ is a curve (or foliation) carried by $\tau$ then $\tau[c]$ is the sub-train track whose branches are traversed by $c$.  We say that $c$ is fully carried by $\tau$ if $\tau[c] = \tau$.  

\begin{proposition}
\label{prop:forcing}
Let $f \in \Mod(S)$ be pseudo-Anosov with horizontal foliation $\F_h$, and let $\tau$ be a filling, transversely recurrent train track in $S$.  If $c$ is a curve in $S$ or a foliation in $S$ with 
\[
\H_f(c) \geq \H_f(\tau) + \Qforce,
\]
then the following hold.  
\begin{enumerate}
\itemindent=-15pt \item If $c \carried \tau$, then $\F_h \carried \tau[c]$.
\itemindent=-15pt \item If $\F_h$ is fully carried by $\tau$, then $c$ is fully carried by a diagonal extension of $\tau$.  
\end{enumerate}
 \end{proposition}
 
We begin in Section~\ref{sec:forcing_prelim} by defining left- and right-turning paths and explaining some of their properties.  We then proceed to define the slope of a train track and relate it to the slopes of the left- and right-turning paths  (Lemma~\ref{lemma:trackslope}).  In Section~\ref{sec:forcing_shadows} we introduce the notion of a shadow for the lift of a train track to the universal cover of the ambient surface.  We then use the notion of a shadow to define obstructing branches and agreeable branches, and use these to give criterion for a foliation to be carried on a train track and a criterion for a carried foliation to traverse a branch of a train track, respectively.  In Section~\ref{sec:forcing_proof} we prove Proposition~\ref{prop:forcing} by first proving a version for train track branches, Lemma~\ref{lem:same_as}.  The latter is an application of the approximating rectangles lemma, \Cref{lem:covering_rectangles}


\subsection{Left- and right-turning paths and slopes of train tracks}
\label{sec:forcing_prelim}

The goal of this section is to define left- and right-turning paths for a lifted train track, and to compare their slopes to those of the corresponding train track.  We begin with preliminaries from the theory of train tracks: transverse recurrence and quasi-geodesic train tracks.  

\p{Transverse recurrence and quasi-geodesic train tracks} If $\tau$ is a transversely recurrent train track in a hyperbolic surface $S$, then the preimage $\wt \tau$ of $\tau$ in $\HH^2$ is quasi-geodesic, in the sense that every bi-infinite train path in $\wt \tau$ is a quasi-geodesic with uniform quasi-isometry constants that do not depend on the path \cite{PennerHarer92}.  This fact is the main reason why we frequently employ the assumption on transverse recurrence. 

Because of this, we may apply the tools and techniques from hyperbolic geometry to $\wt \tau$.  For instance, every bi-infinite train path $\gamma$ in $\wt \tau$ has well-defined endpoints in $\partial \HH^2$.  For a singular Euclidean surface $X$ there is a geodesic in the completed universal cover $\wh X$ with the same endpoints as $\gamma$  (this geodesic is unique except when $\gamma$ corresponds to the lift of the core of a flat annulus).  We will use the fact that two such paths $\gamma_1$ and $\gamma_2$ are linked at infinity if and only if their geodesic representatives in $\wh X$ are.  

As in the work of Masur--Minsky \cite[Lemma 4.5]{MasurMinsky99}, a transversely recurrent track $\wt \tau \subset \HH^2$ does not admit any generalized bigons, by which we mean a region in $\HH^2$ bounded by either \emph{(i)} distinct, bi-infinite train paths with common endpoints  in $\partial \HH^2$ or \emph{(ii)} distinct rays with common starting points in $\HH^2$ and common ending points in $\partial \HH^2$.

\p{Left- and right-turning paths and cycles} Let $\tau$ be a train track in either a hyperbolic surface $S$ or in the universal cover $\HH^2$.  Let $b$ be a branch of $\tau$. The \emph{left-turning path} in $\tau$ through $b$ is the bi-infinite train path $\gamma_b^L$ in $\tau$ that contains $b$ and has the following property: if we travel along $\gamma_b^L$ from $b$ to another branch, then at each switch we traverse the left-most branch in the direction of travel.  

When $\tau$ is a train track in a surface $S$ and $b$ is an oriented branch in $\tau$, we may also define the \emph{left-turning cycle} $\beta_b^L$ via the following construction.  We start from $b$ and proceed in the direction of the orientation, always turning left, until we arrive at a branch $b'$ that we have already traversed in the direction we are about the traverse it again. Then $\beta_b^L$ is the closed train path obtained by discarding the initial path from $b$ to $b'$. Right-turning paths and cycles $\gamma_b^R$  and $\beta_b^R$ are defined similarly.  

We have the following basic facts about a transversely recurrent train track $\tau$ in a surface $S$:
\begin{itemize}
\item for branch $b$ of $\tau$ the paths $\gamma_b^L$ and $\gamma_b^R$ are simple,
\item for each oriented branch $b$ of $\tau$ the curves $\beta_b^L$ and $\beta_b^R$ are simple, and
\item for any oriented edge $b$, the paths and curves $\gamma_b^L$ and $\beta_b^L$ are disjoint, and similarly for $\gamma_b^R$ and $\beta_b^R$.
\end{itemize}
All three statements are verified by checking that the lifts to $\HH^2$ can be realized disjointly, equivalently that there are no two (components of) lifts whose endpoints separate each other in $\partial \HH^2$.  If this were the case, then at the first and last intersections of the first path with the second, the first path would be turning in the same direction with respect to the second, which is impossible.

One other fact that we will use is that each curve of the form $\beta_b^L$ or $\beta_b^R$ is a vertex cycle for $\tau$.  This is straightforward from the definition; see, e.g. \cite[Lemma 2.12]{GS}.  

\p{Slopes of train tracks versus slopes of left- and right-turning paths} We define the slope of a train track $\tau$ to be the union of the slopes of its vertex cycles, and we define $\H_f(\tau)$ in the usual way from the slope.  Recall that $\Qcurve = 2K^2$, that $E=14$, and that $\Qtrack = E \Qcurve$.  We prove now two lemmas in turn.

\begin{lemma}\label{lemma:trackslope1}
Let $\tau$ be a transversely recurrent train track in a surface $S$.  Then
\[
\diam \H_f(\tau) \leq \Qtrack.
\]
\end{lemma}

\begin{proof}

As in the introduction, $E(S)$ is an upper bound for the diameter in $\C(S)$ of the set of vertex cycles of $\tau$.  By \Cref{prop:boundedslope}, each edge $(v,w)$ in $\C(S)$ has $\diam \H_f(v) \cup \H_f(w) \leq \Qcurve$.  Finally, for any edge in the curve graph, the corresponding saddle connections are pairwise non-crossing (this is even true in the exceptional cases of $S_{1,1}$ and $S_{0,4}$ where edges do not correspond to disjoint curves but to curves that intersect once and twice, respectively).  The lemma follows.
\end{proof}

\begin{lemma}\label{lemma:trackslope}
Let $f \in \Mod(S)$ be pseudo-Anosov with singular Euclidean surface $X$ and let $K=K(X)$.  If $\tau \subseteq X$ is a transversely recurrent train track then for each branch $b$ of $\tau$ we have
\[
\H_f(\tau) - (\Qtrack + \Qcurve) \leq \H_f(\gamma_b^L) \leq \H_f(\tau) + (\Qtrack + \Qcurve)
\]
and similarly for $\gamma_b^R$.
\end{lemma}

\begin{proof}

As above, we have that every $\gamma_b^L$ is disjoint from the corresponding vertex cycle $\beta_b^L$.  Again by \Cref{prop:boundedslope} we have $\diam \H_f(\gamma_b^L) \cup \H_f(\beta_b^L) \leq \Qcurve$.  By Lemma~\ref{lemma:trackslope1} and the fact that $\beta_b^L$ is a vertex cycle of $\tau$, the lemma follows.  
\end{proof}


\subsection{Obstructing and agreeable branches}
\label{sec:forcing_shadows}

In this section we introduce the notion of a shadow for a train track.  We then define the notion of an obstructing branch and use it to give a criterion for a foliation to be carried by a train track.  Finally, we define the notion of an agreeable branch and use it to give a criterion for a carried foliation to traverse a certain branch of a train track.  

\p{Shadows, obstructing branches, and agreeable branches} Let $\tau$ be a transversely recurrent train track in a hyperbolic surface $S$ and $\wt \tau$ the preimage in $\HH^2$.  Let $b$ be a branch of $\wt \tau$.  We define the shadow $\Lambda_b$ as the union of the two disjoint intervals of $\partial \HH^2$ that run counterclockwise from an endpoint of $\gamma_b^R$ to an endpoint of $\gamma_b^L$.

Next, let $\gamma$ be a bi-infinite path in $\HH^2$ with well-defined endpoints at infinity (for instance a component of the preimage of a curve in $S$ or a leaf of a foliation of $S$).  We say that $b$ obstructs $\gamma$ if the endpoints of $\gamma$ in $\partial \HH^2$ separate the two components of the shadow $\Lambda_b$ (in the language of Masur--Minsky \cite[Proof of Lemma 4.5]{MasurMinsky99}, $b$ consistently separates $\gamma$).  We then say that $b$ obstructs a foliation $\F$ of $S$ if it is obstructing for some component of the preimage of some singularity-free leaf of $\F$.  

Finally, we say that the branch $b$ is agreeable to the path $\gamma$ if the endpoints of $\gamma$ in $\partial \HH^2$ lie in different components of $\Lambda_b$.  As in the obstructing case, we say that $b$ is agreeable to a foliation $\F$ if it is agreeable to some component of the preimage of some singularity-free leaf of $\F$.  

\p{Obstructing branches and a criterion for carrying}  The next lemma uses obstructing branches to give a criterion for a curve or a foliation to be carried by a train track.  

\begin{lemma}\label{lemma:diag-extension}
Let $\tau$ be a filling transversely recurrent train track in a hyperbolic surface $S$, and let $\F$ be a measured foliation of $S$.  Then $\F$ is carried on a diagonal extension of $\tau$ if and only if $\tau$ has no obstructing branches for $\F$.  
\end{lemma}

\begin{proof}

A proof of the reverse implication can be found in the work of Masur--Minsky \cite[Proof of Lemma 4.5]{MasurMinsky99}.  The argument there is framed in the special case where $\F$ is a simple closed curve (as opposed to a general foliation) but the same argument applies.  (The hypotheses of the Masur--Minsky lemma include the assumption that $\tau$ is bi-recurrent; however, this part of the proof only uses the transverse recurrence.)

It remains to prove the forward implication.  Suppose that $\F$ is carried by $\tau$.  Let $\ell$ be a singularity-free leaf of the preimage of $\F$ in $\HH^2$.  Also let $\wt \tau$ be the preimage of $\tau$ in $\HH^2$ and let $b$ be a branch of $\wt \tau$.  We would like to show that $b$ is not an obstructing branch for $\ell$.  

Let $\gamma$ be the bi-infinite train path in $\wt \tau$ whose endpoints in $\partial \HH^2$ are the same as those of $\ell$.  Consider now the intersection $\gamma \cap \gamma_b^L$.  If this intersection is empty, then the endpoints of $\ell$ and $\gamma_b^L$ do not link in $\partial \HH^2$ and it follows that $b$ is not obstructing for $\ell$.  

Suppose then that $b'$ is a branch of $\gamma \cap \gamma_b^L$.  There are two sub-paths of $\gamma$ that start at $b'$.  Choose a sub-path $\gamma'$ that moves away from $b$ (if $b=b'$ then both choices move away from $b$).  Since $b'$ lies on $\gamma_b^L$, the endpoint of $\gamma'$ at infinity lies in the shadow $\Lambda_b$.  Thus $b$ is not obstructing for $\ell$, as desired.
\end{proof}

\p{Agreeable branches and a criterion for edge traversal} The next lemma uses agreeable branches to give a criterion for a carried curve or foliation to traverse a particular edge of a train track.  

\begin{lemma}\label{lemma:carrying-points}
Let $\tau \subset S$ be a transversely recurrent train track in a hyperbolic surface, let $\wt \tau$ be the preimage in $\HH^2$, and let $\gamma$ be a bi-infinite train path in $\wt \tau$. Suppose that the endpoints of $\gamma$ in $\partial \HH^2$ lie in different components of $\Lambda_b$ for some branch $b$ of $\wt \tau$.  Then $\gamma$ traverses $b$.
\end{lemma}

\begin{proof}

Let $\alpha$ be the intersection of $\gamma_b^L$ and $\gamma_b^R$.  It must be that $\alpha$ is a connected train path in $\wt \tau$ (otherwise $\wt \tau$, hence $\tau$, would bound an embedded bigon).  So $\gamma_b^L \cup \gamma_b^R$ divides $\HH^2$ into four regions, two that meet $\Lambda_b$ and two others.  If $\gamma$ does not traverse $\alpha$, then it must enter the interior of one of the latter two regions.  The intersection of $\gamma$ with this region is a train path that is either
\begin{itemize}
\item a compact train path,
\item an infinite train path asymptotic to $\partial \Lambda_b$, or 
\item a bi-infinite train path (namely, all of $\gamma$) asymptotic to $\partial \Lambda_b$ at both ends.
\end{itemize}
In the first case, we see that $\wt \tau$, hence $\tau$, bounds a bigon, and in the second and third cases we see that $\wt \tau$ has a generalized bigon. These contradictions complete the proof.
\end{proof}


\subsection{Proof of the forcing lemma}
\label{sec:forcing_proof}

In this section we prove \Cref{prop:forcing}, the forcing lemma.  This proposition concludes that certain curves or leaves of a foliation are carried by a train track and also that every branch of the train track is traversed by this carrying.  Before proceeding to the proof of the forcing lemma, we give a local version, that is, a version for a single branch of a (lifted) train track.  

\begin{lemma} \label{lem:same_as}
Let $f \in \Mod(S)$ be pseudo-Anosov with singular Euclidean surface $X$ and horizontal foliation $\F_h$.  Let $\tau \subset S$ be a filling transversely recurrent train track and let $\eta$ be a curve or a leaf of a foliation in $S$ with
\[
\H_f(\eta) \geq \H_f(\tau) + \Qforce.
\]
Let $\wt \tau$, $\wt \F_h$, and $\wt \eta$ be the preimages in $\HH^2$ of $\tau$, $\F_h$, and $\eta$.  Let $b$ be a branch of $\wt \tau$, and let $\ell$ be a singularity-free leaf of $\wt \F_h$.
\begin{enumerate}[leftmargin=2em]
\item If $b$ is agreeable to $\ell$ then $b$ is agreeable to some component of $\wt \eta$.  
\item If $b$ is obstructing for $\ell$ then $b$ is obstructing for some component of $\wt \eta$.
\end{enumerate}
\end{lemma}

\begin{proof}

We would like to apply \Cref{lem:covering_rectangles} to the triple $(\ga_b^L, \ga_b^R, \ell)$, where the paths $\ga_b^L$ and $\ga_b^R$ have been replaced with their geodesic representatives in $\wh X$.  To this end, we verify the hypotheses of that lemma.  We have assumed in the hypotheses of this lemma that $\flex(X) < \infty$ and that $\ell$ contains no singularities.  We also have that $\ga_b^L$ and $\ga_b^R$ intersect, since they are the geodesic representatives of left- and right-turning paths through $b$ (and hence must be linked at infinity).    

We claim that $\ell$ crosses both $\ga_b^L$ and $\ga_b^R$.  Indeed, the only other possibility is that an endpoint of $\ell$ meets $\Lambda_b$ in its boundary.  This is impossible because the points of $\partial \Lambda_b$ are endpoints of lifts of the vertex cycles of $\tau$ corresponding to the left- and right-turning cycles through the image of $b$. 

Since the hypotheses of \Cref{lem:covering_rectangles} are satisfied, it follows from the first statement of the lemma that there is an $F^5$-approximating rectangle $R$ for the triple $(\ga_b^L, \ga_b^R, \ell)$ where $F$ denotes the quantity $2\flex(X)-1$.

We now would like to apply the second statement of \Cref{lem:covering_rectangles} to $\eta$, which requires us to prove 
\[
\min \slope(\eta) < F^{-(2K+6)} \min \left\{ \slope(\ga_b^L) \cup \slope(\ga_b^R) \right\} 
\]
By our hypothesis on $\eta$ and Lemma~\ref{lemma:trackslope}, we have
\begin{align*}
\H_f(\eta) &\geq \H_f(\tau) + \Qforce \\
&= \H_f(\tau) + \Qcurve + \Qtrack + (2K+6) \cdot 2K \\
&\geq \max \{ \H_f(\ga_b^L) \cup \H_f(\ga_b^R)\} + (2K+6) \cdot 2K
\end{align*}
Converting horizontalities to slope we have
\[
\slope(\eta) \leq (\lambda^{2K})^{-(2K+6)} \min \left\{ \slope(\ga_b^L) \cup \slope(\ga_b^R) \right\} 
\]
Hence to verify the hypothesis for $\eta$ it suffices to show that $F < \lambda^{2K}$.  But this follows from Lemmas~\ref{lemma:flexibility} and~\ref{lem:penner}:
\[
F =  2\flex(X)-1 \le 2\lambda^K-1 < 2 \lambda^K \le \lambda^{2K}.
\]
Having verified the hypothesis on $\eta$, \Cref{lem:covering_rectangles}(2) gives a geodesic $\wt \eta$ that is the geodesic representative of a component of the preimage of $\eta$ and that crosses the vertical edges of $R$.  Since $R$ is an approximating rectangle, the geodesics $\ga_b^L$ and $\ga_b^R$ cross the horizontal edges of $R$.  It follows that $\wt \eta$ crosses both $\ga_b^L$ and $\ga_b^R$ transversely within $R$.

In a $\CAT(0)$ space, the intersection of two geodesics is connected, and so $\wt \eta$ intersects each of $\ga_b^L$ and $\ga_b^R$ in a single point.  Also, two geodesics that cross in a $\CAT(0)$ space cannot have the same endpoints at infinity.  Therefore, the endpoints of $\wt \eta$ must be contained in the same components of $\partial \HH^2 \setminus \left (\partial \ga_b^L \bigcup \partial \ga_b^R \right )$ as the endpoints of $\ell$. This implies both statements of the lemma.
\end{proof}

\begin{proof}[Proof of Proposition~\ref{prop:forcing}]

As in the statement, $\F_h$ is the horizontal foliation for a pseudo-Anosov $f \in \Mod(S)$.  Also, $\tau$ is a filling transversely recurrent train track in $S$, and $c$ is a singularity-free leaf of a foliation in $S$ with
\[
\H_f(c) \geq \H_f(\tau) + \Qforce.
\]
In what follows, $X$ is a singular Euclidean structure associated to $f$, and $\wt \tau$ and $\wt \F_h$ are the preimages of $\tau$ and $\F_h$ in the universal cover $\wt X$.  We prove the two statements of the proposition in turn.

\smallskip

\p{Statement (1)} Assume that $c \carried \tau$.  We first claim that $\F_h$ is carried by a diagonal extension $\tau'$ of $\tau$.  Assume not.  By \Cref{lemma:diag-extension} there is a leaf $\ell$ of $\wt \F_h$ and a branch $b$ of $\wt \tau$ that is obstructing for $\ell$.  Then by \Cref{lem:same_as}, $c$ is obstructed by $b$.  By \Cref{lemma:diag-extension}, $c$ is not carried by $\tau$, the desired contradiction. 

We now show that $\F_h$ is carried by $\tau'[c]$.  Since $\tau'[c] = \tau[c]$, this will complete the proof of the first statement.   Let $b$ be a branch of $\wt \tau'$ traversed by some singularity-free leaf $\ell$ of $\wt \F_h$.  The branch $b$ is agreeable to $\ell$.  By \Cref{lem:same_as}, the branch $b$ is agreeable to some component $\wt c$ of the preimage of $c$. By \Cref{lemma:carrying-points}, the train path of $\wt \tau'$ corresponding to $\wt c$ traverses the branch $b$ and the first statement follows.

\smallskip

\p{Statement (2)} As per the second statement, we assume that $\F_h$ is fully carried by $\tau$.  Let $b$ be a branch of $\wt \tau$ traversed by (a leaf of) $\wt \F_h$.  As in the proof of Statement~(1), \Cref{lem:same_as} implies there is an component of the preimage of $c$ that traverses $b$.  Since the components of $\wt c$ do not cross, it follows that $b$ is not obstructing for any component of the preimage of $c$.  Since $b$ was arbitrary, \Cref{lemma:diag-extension} implies that $c$ is carried by a  diagonal extension $\tau'$ of $\tau$.  Applying \Cref{lem:same_as} to $\tau'$, we conclude that $c$ is fully carried by $\tau'[\F_h]$.  Since the latter is equal to $\tau$, the proof is complete.  
\end{proof}


\section{The fitting lemma}
\label{sec:back}

The goal of this section is to prove the following, which we call the fitting lemma.  

\begin{proposition}\label{prop:squeeze}
Let $f \in \Mod(S)$ be pseudo-Anosov with horizontal foliation $\F_h$, and let $M \in \R$.  Let $T$ be the set of transversely recurrent train tracks in $S$ that carry $\F_h$ and have horizontality at most $M$.  There is an $f$-invariant train track $\nu$ so that $\nu$ is carried by each element of $T$ and
\[
\H_f(\nu) \leq M + \Qfit.
\]
\end{proposition}

The $\nu$ in the conclusion of Proposition~\ref{prop:squeeze} is fitted in the sense that it lies between $\F_h$ and each $\tau \in T$  (in the sense that $\F_h \carried \nu \carried \tau$) and it has horizontality close to $\H_f(T)$.  

The starting point for the proof of \Cref{prop:squeeze} is the work of Agol \cite{Agol11}.  Given a pseudo-Anosov $f$, Agol gives a bi-infinite sequence of train tracks and multi-splittings 
\[
\cdots \split  \nu_i \split \nu_{i+1} \split \cdots \split \nu_{i+k} \split \cdots
\]
and a $k > 0$ so that
\[
f(\nu_i) = \nu_{i+k}
\]
for all $i$; where by a multi-splitting we mean a collection of simultaneous train track splits along one or more large branches.  In particular, each $\nu_i$ is $f$-invariant and the $\H_f(\nu_i)$ tend to $\pm \infty$ as $i \to \pm \infty$.  Also, each $\nu_i$ fully carries the unstable foliation for $f$, and the set of large branches split at each step is precisely the set of branches of largest weight with respect to the unstable foliation.  The train track $\nu$ in the conclusion of \Cref{prop:squeeze} will be chosen from among the $\nu_i$ in Agol's sequence.

Agol proves that if $\tau$ is a train track carrying the horizontal foliation for $f$, then $\tau$ splits to some $\nu_i$ in the sequence \cite[Theorem 3.5]{Agol11}.  The basic idea of the proof of \Cref{prop:squeeze} is to use this fact to find a $\nu_i$ that all the $\tau_i$ split to, and then to show that we can move backwards along the sequence of $\nu_i$ until the slope condition is also satisfied.  

We divide the proof into two steps.  In Section~\ref{sec:backfolding} we prove \Cref{lem:baby}, the backfolding lemma, which gives conditions under which the fold of a train track carried by a train track $\tau$ is again carried by $\tau$; this is the main technical lemma of the section.  Then in Section~\ref{sec:agol} we focus specifically on the Agol train tracks described above, and use them to prove \Cref{prop:squeeze}.  


\subsection{The backfolding lemma}
\label{sec:backfolding}

Before proceeding to the statement and the proof of the backfolding lemma, we introduce the following criterion for when a single fold does not does not disturb the carrying relation.  In the statement, we say that a shadow $\Lambda_b$ is contained in a shadow $\Lambda_a$ by a component-respecting inclusion if each component of $\Lambda_a$ contains one component of $\Lambda_b$.  

\begin{lemma}\label{lem:fold_to_carry}
Let $S$ be a hyperbolic surface and let $\sigma$ and $\tau$ be transversely recurrent train tracks in $S$ with $\sigma \prec \tau$.  Let $\sigma \fold \sigma'$ be a single fold creating a new large branch $b'$.  Let $\wt \sigma'$ and $\wt \tau$ be the preimages of $\sigma'$ and $\tau$ in $\HH^2$, and let $\wt b' \subseteq \wt \sigma'$ be a lift of $b'$ to $\HH^2$. Suppose there is a branch $\tilde a$ of $\wt \tau$ with
\[
\Lambda_{\wt b'} \subseteq \Lambda_{\wt a}
\]
with a component-respecting inclusion.  Then $\sigma'$ is carried by $\tau$.
\end{lemma}

\begin{proof}

Let $e$, $f$, and $g$ be the branches that are folded together in $\sigma$ to produce the branch $b'$ of $\sigma'$, where $f$ is the small branch.  There are two cases: we can either choose the carrying map $\sigma \to \tau$ so that a small neighborhood of $f$ in $\sigma$ maps homeomorphically or we cannot (here we are referring only to the carrying map, not the map to a tie neighborhood).  In the latter case $\sigma \prec \tau$ factors through $\sigma \mapsto \sigma'$, contrary to assumption.  So we may assume that a small neighborhood of $f$ does map homeomorphically; say that the branch in the image is $h$. 

One way to visualize the statements in the preceding paragraph is to draw a tie neighborhood $N$ of $\tau$ and to consider $\sigma \subset N$ so that collapsing along the fibers give a carrying map $\sigma \prec \tau$.  If some train track tie of $N$ passes through all three branches $e$, $f$, and $g$, the carrying map $\sigma \prec \tau$ factors through $\sigma \mapsto \sigma'$. The only other possibility is that $e$, $f$, and $g$ are configured in $N$ as in \Cref{fig:sigma_in_N}.  

Let $\wt e$, $\wt f$, and $\wt g$ be preimages of $e$, $f$, and $g$ in $\widetilde \sigma$ that are folded to create $\wt b'$.  Because of the folding relation, we have
\[
\Lambda_{\wt e} \cup \Lambda_{\wt f} \cup \Lambda_{\wt g} \subseteq \Lambda_{\wt b'} 
\]
by component-respecting inclusions. By assumption we have an edge $\tilde a$ of $\wt \tau$ with $\Lambda_{\wt b'} \subseteq \Lambda_{\wt a}$ and hence $\Lambda_{\wt a}$ contains $\Lambda_{\wt e}$, $\Lambda_{\wt f}$, and $\Lambda_{\wt g}$.  Let $\wt h$ be the branch of $\wt \tau$ that is the homeomorphic image of $\wt f$.  

\begin{figure}[hbt!]
\labellist
\pinlabel  {$\wt e$} [ ] at 50 220
\pinlabel  {$\wt f$} [ ] at 155 205
\pinlabel  {$\wt g$} [ ] at 570 195
\endlabellist  
 \includegraphics[width=.6667\textwidth]{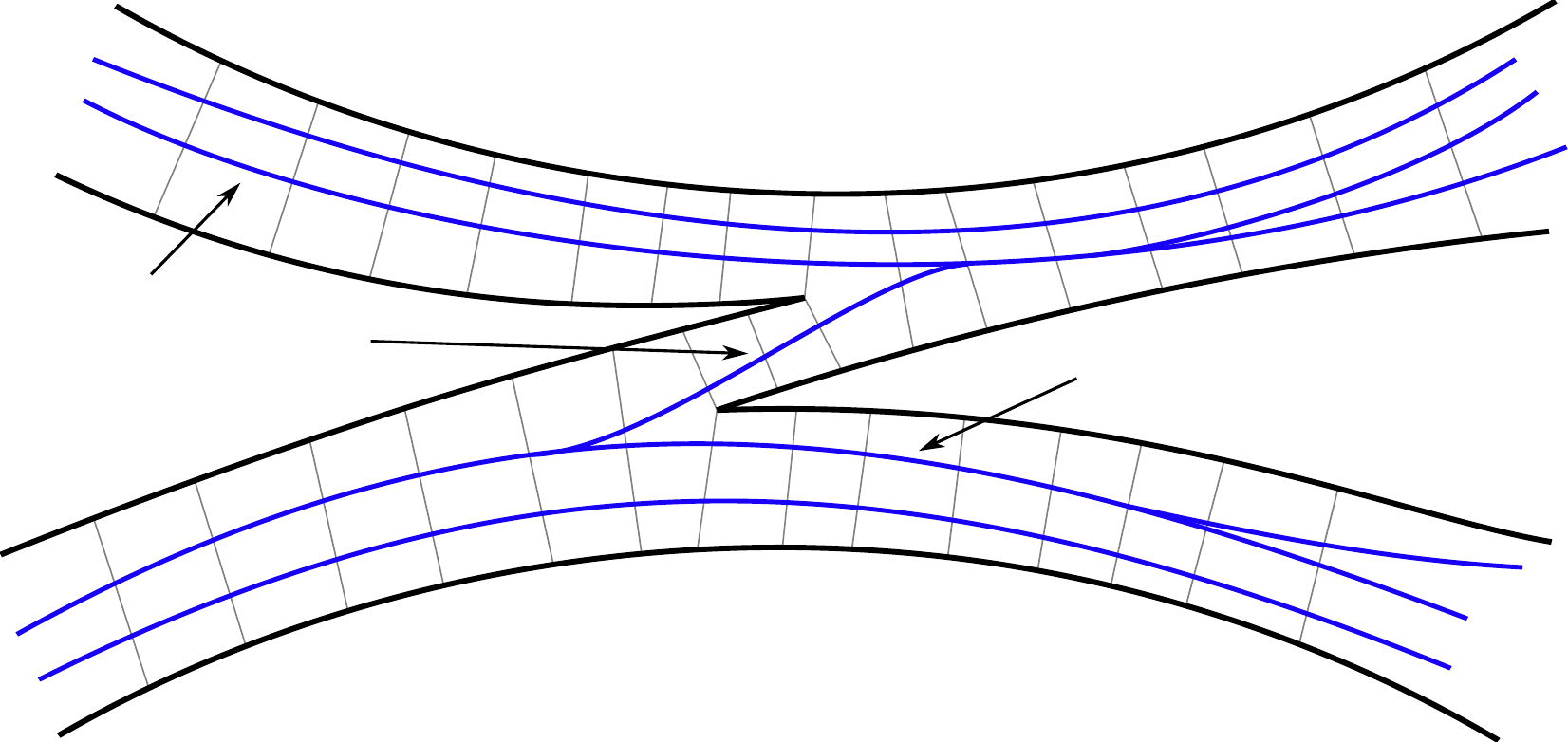}
\caption{The configuration of $\sigma \subset N$ if $\sigma \prec \tau$ does not factor through $\sigma \mapsto \sigma'$}
\label{fig:sigma_in_N}
\end{figure}

By the above (component-respecting) inclusions, the path 
\[
\gamma = \gamma_{\wt f}^L
\]
has endpoints in both components of $\Lambda_{\wt a}$.  By Lemma~\ref{lemma:carrying-points}, the image of $\gamma$ in $\wt \tau$ passes through $\wt a$.  In particular, $\wt h$ and the branch $\wt a$ are connected by a train path in $\wt \tau$ (a sub-path of $\gamma$). 

Finally, observe from \Cref{fig:sigma_in_N} that $\Lambda_{\wt h}$ contains neither $\Lambda_{\wt e}$ nor $\Lambda_{\wt g}$.
 Since moving along a train path does not increase the shadow in the forward direction, we must have that one of the inclusions $\Lambda_{\wt e} \subseteq \Lambda_{\wt a}$ or $\Lambda_{\wt g} \subseteq \Lambda_{\wt a}$ is false.  This is the desired contradiction.
\end{proof}

We require one more lemma about slopes before proceeding to the backfolding lemma.  Recall that $E=14$ and $\Qcurve(S) = 2K^2$.  

\begin{lemma}
\label{lem:foldslope}
Let $f \in \Mod(S)$ be pseudo-Anosov, let $\sigma \fold \sigma'$ be a single fold of transversely recurrent train tracks, and let $b' \subseteq \sigma'$ be the large branch created in the fold.  Then
\[
\min \H_f(\sigma) - (E+3)\Qcurve \leq \H_f(\gamma_{b'}^{L/R}) \leq \max \H_f(\sigma) + (E+3)\Qcurve.
\]
\end{lemma}

\begin{proof}

By symmetry, it suffices to prove the first inequality.  Let $b$ be the branch of $\sigma$ used to do the fold.  Depending upon the orientation of the fold, it is either the case that $\gamma_{b'}^L$ and $\gamma_b^L$ are homotopic or $\gamma_{b'}^R$ and $\gamma_b^R$ are (up to modifying the fold itself by an ambient isotopy, we can in fact make the two paths equal).  Say it is the former that are homotopic. As in the proof of Lemma~\ref{lemma:trackslope}, we have
\[
\H_f(\gamma_{b'}^L) = \H_f(\gamma_b^L) \geq \min \H_f(\sigma) - \Qcurve.
\]
Using the same reasoning and applying Lemma~\ref{lemma:trackslope}, we have that
\[
\H_f(\gamma_{b'}^R) \geq \H_f(\gamma_{b'}^L) - (E+2)\Qcurve \geq \min \H_f(\sigma) - (E+2)\Qcurve.
\]
Combining these gives the lemma.
\end{proof}

We now arrive at the backfolding lemma.

\begin{lemma}\label{lem:baby}
Let $f \in \Mod(S)$ be pseudo-Anosov with singular Euclidean surface $X$ and horizontal and vertical foliations $\F_h$ and $\F_v$.  Let $\tau$ be a transversely recurrent train track in $S$ that fully carries $\F_h$, and let $\sigma$ be a train track in $S$ with $\sigma \prec \tau$ and
\[
\H_f(\sigma) \geq \H_f(\tau) + \Qfit
\]
Further suppose that $\sigma \fold \sigma'$ is a single fold producing a large branch $b' \subseteq \sigma'$ and that $b'$ obstructs $\F_v$. Then $\sigma'$ is carried by $\tau$.
\end{lemma}

\begin{proof}

Our goal is apply Lemma~\ref{lem:fold_to_carry}.  As such, our main work is to verify the hypotheses of that lemma, with $\tau$, $\sigma$, $\sigma'$, and $b'$ in this lemma corresponding to the objects with the same names in the former.  For starters, since $f$ is pseudo-Anosov, we have that $S$ is hyperbolic, as required by Lemma~\ref{lem:fold_to_carry}.  

Let $\wt b'$ be a lift of $b'$ to $\HH^2$.  The only hypothesis of Lemma~\ref{lem:fold_to_carry} that it remains to verify is that there is a branch $\wt a$ in $\wt \tau$, the preimage of $\tau$ in $\HH^2$, with
\[
\Lambda_{\wt b'} \subseteq \Lambda_{\wt a}.
\]
We first address the case where $\tau$ is complete, meaning that $\tau$ has no proper extensions, and then we use this special case to derive the general case.  

\p{Complete case} We first claim that $\gamma_{b'}^L$ and $\gamma_{b'}^R$ are carried by $\tau$.  We will prove this using \Cref{prop:forcing}, and so we would like to bound the horizontalities of both paths.  First, the assumption $\H_f(\sigma) \geq \H_f(\tau) + \Qfit$ can be rephrased as $\min \H_f(\sigma) \geq \H_f(\tau) + \Qfit$.  Recall that $\Qfit = \Qforce + (E+3)\Qcurve$.  Using Lemma~\ref{lem:foldslope} and the rephrasing of the assumption we have
\begin{align*}
\H_f(\gamma_{b'}^{L/R}) &\geq \min \H_f(\sigma) - (E+3)\Qcurve \\
&= \min \H_f(\sigma) - (\Qfit - \Qforce) \\
&\geq \H_f(\tau) + \Qforce.
\end{align*}
By \Cref{prop:forcing}, the assumption that $\tau$ fully carries $\F_h$ and the assumption that $\tau$ is complete, the claim now follows.  

By the previous claim the paths $\gamma_{\wt b'}^L$ and $\gamma_{\wt b'}^R$ in $\HH^2$ are carried by $\wt \tau$.  Also, these two paths are linked at infinity in $\HH^2$.  Thus, there is a branch $\wt a$ in $\wt \tau$ contained in both paths.  We will show that this $\wt a$ satisfies the desired containment $\Lambda_{\wt b'} \subseteq \Lambda_{\wt a}$.  

Since $\gamma_{\wt b'}^L$ and $\gamma_{\wt b'}^R$ both pass through $\wt a$, we have
\[
\partial \Lambda_{\wt b'} \subseteq \Lambda_{\wt a}.
\]
From this containment, there are two possibilities for $\Lambda_{\wt b'}$: either $\Lambda_{\wt b'} \subseteq \Lambda_{\wt a}$ (which is what we want to show) or $\partial \Lambda_{\wt a} \subseteq \Lambda_{\wt b'}$.  As such, it remains to rule out the latter.

By assumption there is a leaf $\ell$ of $\F_v$ that separates the components of $\Lambda_{\wt b'}$. Any such $\ell$ must cross both $\gamma_{\wt b'}^L$ and $\gamma_{\wt b'}^R$.  By (the unstated vertical version of) \Cref{lem:covering_rectangles} there is a maximal rectangle $R$ so that  $\ell$ meets $R$ in a segment, so that the geodesic representative of $\gamma_{\wt b'}^L$ in $\wh X$ crosses the vertical edges of $R$, and so that every geodesic $\gamma$ in $X$ with
\[
\max \slope(\gamma) > F^{2K+6} \slope(\gamma_{\wt b'}^i) \text{ for } i \in \{L,R\}
\]
there is a component $\wt \gamma$ of the preimage of $\gamma$ the crosses the horizontal edges of $R$.  Here $F = 2 \flex(X)-1$, as usual.

We next claim that $\gamma_{\wt a}^L$ satisfies the criterion
\[
\max \slope(\gamma_{\wt a}^L) > F^{2K+6} \slope(\gamma_{\wt b'}^i) \text{ for } i \in \{L,R\}.
\]
First, by \Cref{prop:boundedslope}, the assumption on $\H_f(\sigma)$, \Cref{lem:foldslope}, the definition of $E$, and the fact that $\Qfit \geq 0$, we have
\begin{align*}
\min \H_f(\gamma_{\wt a}^L) &\leq \max \H_f(\tau) + \Qcurve \\
&\leq \min \H_f(\sigma) + \Qcurve -  \Qfit \\ 
&\leq \H_f(\gamma_{b'}^{L/R})  + (E+4)\Qcurve -  \Qfit  \\
&= \H_f(\gamma_{b'}^{L/R})  + 18\,\Qcurve - \Qfit \\
& \leq \H_f(\gamma_{b'}^{L/R})  + 18\,\Qcurve.
\end{align*}
We convert to slopes and use Penner's inequality $\lambda^{2K} > F$ (Lemma~\ref{lem:penner}):
\begin{align*}
\max \slope(\gamma_{\wt a}^L) &\geq (\lambda^{2K})^{18K} \slope(\gamma_{b'}^{L/R}) \\
&\geq F^{18K} \slope(\gamma_{b'}^{L/R}). 
\end{align*}
Since $18K > 2K+6$, the claim follows.  (This argument does not use $\H_f(\sigma) \geq \H_f(\tau) + \Qfit$, but the weaker assumption that $\H_f(\sigma) \geq \H_f(\tau) + (E+2)\Qcurve - 12K$.)

Let us replace $\gamma_{\wt a}^L$, $\gamma_{\wt b}^L$, and $\gamma_{\wt b}$ with their geodesic representatives in $\wh X$.  By the previous claim, the description of $R$ from \Cref{lem:covering_rectangles}, and the CAT(0) geometry of $\wh X$ (specifically that geodesics in $\wh X$ intersect in a connected set), there is a translate $\gamma'$ of $\gamma = \gamma_{\wt a}^L$ that crosses both $\gamma_{\wt b}^L$ and $\gamma_{\wt b}^R$ each in exactly one point.  It follows that $\gamma'$ also separates the components of $\Lambda_{\wt b'}$. However, under our assumption $\gamma$ has its boundary in $\Lambda_{\wt b'}$ and so $\gamma$ and $\gamma'$ must cross. This contradicts the fact that left turning paths are simple and we are done.

\p{General case} Suppose that $\sigma' \split \sigma$ is a left split, so that $\gamma_b^L$ is automatically carried by $\tau$. Now by \Cref{prop:forcing} and the assumption that $\tau$ fully carries $\F_h$, we have that $\gamma_{\wt b}^R$ is carried by $\wt \tau'$, where $\tau'$ is some diagonal extension of $\tau$. Let $\wt a$ be a common branch of the train paths $\gamma_{\wt b}^{L/R}$ as carried by $\wt \tau'$. As $\gamma_{\wt b}^L \prec \wt \tau$, we must have that $\wt a$ is contained in $\wt \tau$. Note that since a diagonal of a diagonal extension is never the left/right-most turn, we have that $\Lambda_{\wt a}$ is independent of whether $\wt a$ is considered as a branch of $\wt \tau$ or $\wt \tau'$. By the special case of the lemma, we have $\Lambda_{\wt b} \subseteq \Lambda_{\wt a}$ and we are done.
\end{proof}


\subsection{Proof of the fitting lemma}
\label{sec:agol}

Before proving \Cref{prop:squeeze} we have a technical lemma about Agol's train tracks, Lemma~\ref{lem:veering_dual}.  This lemma will use a fact from the work of Gu\'eritaud, Minsky, and the third-named author, which we now explain.  

\p{Gu\'eritaud and transversality} Using the work of Gu\'eritaud, Minsky and the third author observed that each Agol train track $\nu_i$ can be realized as a smoothing of the dual graph to a triangulation of the surface by rectangle-spanning saddle connections.  Moreover, the large half-branch in each triangle meets the tallest side.  See Figure~\ref{fig:agol} for an illustration of the local picture in one such triangle.  It follows that $\nu_i$ can be chosen to be transverse to the vertical foliation $\F_v$ for the corresponding pseudo-Anosov mapping class.  

\begin{figure}[htbp]
\begin{center}
\includegraphics[width = .2 \textwidth]{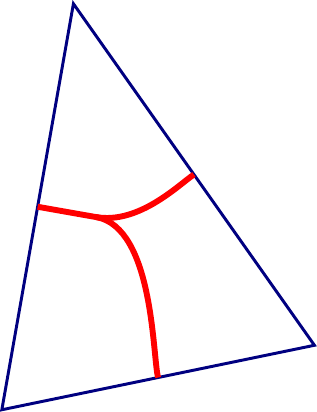}
\caption{A piece of an Agol train track inside a triangle of a triangulation by rectangle-spanning saddle connections}
\label{fig:agol}
\end{center}
\end{figure}

\begin{lemma}\label{lem:veering_dual}
Let $f \in \Mod(S)$ be pseudo-Anosov with vertical foliation $\F_v$, and let $\nu$ be an Agol train track for $f$.  Then 
every branch of $\nu$ obstructs $\F_v$. 
\end{lemma}

\begin{proof}

Fix a singular Euclidean structure on $S$ associated to $f$, and let $\wt \nu$ be the preimage of $\nu$ in the corresponding completed universal cover $\wh X$.  Let $\wt \F_v$ be the induced vertical foliation of $\wh X$.  As above, we may assume that $\wt \nu$ is transverse to $\wt \F_v$.

We claim that a train path in $\wt \nu$ meets a nonsingular leaf of $\wt \F_v$ at most once.  Indeed, suppose that $\ell$ is a nonsingular leaf of $\wt \F_v$ that meets some train path of $\wt \nu$ in more than one point. There is an embedded (innermost) disk $D$ in $\HH^2$ bounded by a train path $p$ in $\wt \nu$ and a segment of $\ell$, and so that the interior of $D$ does not meet $\wt \nu$. The disk $D$ is foliated by (the restriction of) $\wt \F_v$.  If we double $D$ across the segment of $\ell$, we obtain foliated disk where the leaves are transverse to the boundary.  Doubling this resulting disk along its boundary, we obtain a foliated sphere where each singularity corresponds to a singularity in the interior of $D$.  Each such singularity has negative index, contradicting the Gauss--Bonnet formula.  The claim follows.

Now, let $b$ be a branch of $\wt \nu$, and let $\ell$ be a nonsingular leaf of $\F_v$ that passes through $b$.  By the claim, $\ell$ meets $\gamma_b^L$ and $\gamma_b^R$ exactly once each.  Thus the endpoints of $\ell$ separate the components of $\Lambda_b$ and so $b$ obstructs $\ell$, hence $\F_v$, as required.
\end{proof}

\begin{proof}[Proof of \Cref{prop:squeeze}]

As above, the $\H_f(\nu_i)$ tend to $\pm \infty$ as $i \to \pm \infty$.  Thus there is an $i \in \Z$ so that $\H_f(\nu_i) \leq M + \Qfit$ and $\H_f(\nu_{i+k}) \geq M + \Qfit$ for all $k > 0$.  To prove the proposition, it suffices to show that $\nu_i$ is carried by all elements of $T$.

So let $\tau \in T$.  As above, this means that $\tau$ is a transversely recurrent train track in $S$ that carries $\F_h$ and has $\H_f(\tau) \leq M$.  Let $\tau'$ denote $\tau[\F_h]$, the sub-train track of $\tau$ fully carrying $\F_h$.  We will show that $\nu_i \carried \tau'$.  

By the work of Agol discussed at the start of the section, $\tau'$ splits to some $\nu_j$ in the Agol sequence.  Assume that $j$ is the smallest integer with $\nu_j \carried \tau'$.  If $j \leq i$ we are done, because then $\nu_i \carried \nu_j \carried \tau'$.  So we assume $j > i$.  By our assumption in the first paragraph of the proof, $\H_f(\nu_j) \geq M + \Qfit$.  \Cref{lem:veering_dual} implies that every branch of $\nu_j$ obstructs the vertical foliation $\F_v$ for $f$.  

Lemma~\ref{lem:baby} then implies that each individual fold of the multifold $\nu_j \mapsto \nu_{j-1}$ produces a train track carried by $\tau'$.  Since these individual folds are far apart from each other in the sense that the large branches they create are disjoint, this further implies that $\nu_{j-1}$ is carried by $\tau'$, contradicting the minimality of $j$.  Thus, $j \leq i$ and so $\nu_i \carried \nu_j \carried \tau'$, as desired. 
\end{proof}


\section{Proof of Theorem~\ref{thm:main}}
\label{sec:proof}

Before proving our main technical theorem, we require several technical lemmas: a criterion for carrying, a construction of a train track neighborhood of a curve, an upper bound on the radius of a Dehn--Thurston cell structure, an upper bound for the radius of the set of vertex cycles in a Dehn--Thurston train track, and a linear algebraic lemma about projective sinks.

\p{Another carrying criterion} The following lemma gives a criterion for the extension operation to preserve the carrying relation between train tracks.  

\begin{lemma}
\label{lem:still_carry}
Suppose that $\nu \prec \tau$, where $\nu$ and $\tau$ are large train tracks in a surface. Let $\wt \nu$ be a diagonal extension of $\nu$ that fully carries a measured foliation $y$ that is also carried by $\tau$. Then a carrying map $\nu \to \tau$ extends to a carrying map $\wt \nu \to \tau$.
\end{lemma}

\begin{proof}

A carrying map from $\nu$ to $\tau$ takes $\wt \nu$ to a diagonal extension $\wt \tau$ of $\tau$ (which we assume to be a minimal such extension).  Since $y$ is carried by $\tau$ and $\wt \nu$ and since the carrying position of $y$ in $\tau$ is unique, we have $\wt \tau = \tau$, as desired. 
\end{proof}

\begin{figure}[htbp]
\begin{center}
\includegraphics[width = .5 \textwidth]{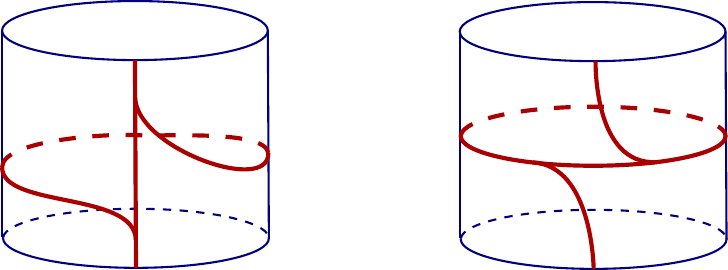}
\caption{Two model train tracks on an annulus carrying a common path connecting the boundaries}
\label{fig:dt}
\end{center}
\end{figure}

\p{A train track neighborhood} We require the following standard construction of an open neighborhood of a point in $\MF(S)$.  

\begin{lemma}
\label{lem:ttnbd}
Suppose $x \in \MF(S)$ is fully carried by a filling train track $\sigma$, which is in turn carried by a train track $\tau$. Let $\sigma_1,\dots,\sigma_k$ be the set of diagonal extensions of $\sigma$ that are carried by $\tau$. Then the union
\[
V(\sigma_1) \cup \cdots \cup V(\sigma_k)
\]
contains an open neighborhood of $x$ in $V(\tau)$.
\end{lemma}


\begin{proof}

The train track $\sigma$ can be pinched along complementary components to produce a complete track $\wt \sigma$ with $\sigma \prec \wt \sigma$ and so that each curve fully carried by $\sigma$ is fully carried by $\wt \sigma$; see \cite[Proposition 1.4.9]{PennerHarer92} where it is said that $\wt \sigma$ arises from $\sigma$ by ``trivial collapses along admissible arcs.''

Say that $\sigma_1,\dots,\sigma_\ell$ are the diagonal extensions of $\sigma$ (note $\ell \geq k$). The set $V(\wt \sigma)$ is equal to the union $V(\sigma_1) \cup \cdots \cup V(\sigma_\ell)$ (see the ``improvement'' following Theorem 5.12 in the book by Casson and Bleiler \cite{CassonBleilerNotes}). The lemma follows. (Alternatively, the construction of the pinched complete track $\wt \sigma$ above can be done within a tie neighborhood of $\tau$, yielding the same conclusion.)
\end{proof}

%
%
%
%

\p{The radius for a Dehn--Thurston cell structure} The following lemma justifies our choice of $D(S)=1$ in the introduction.  Again, see the book by Penner--Harer for background on Dehn--Thurston train tracks \cite[{\S}2.6]{PennerHarer92}.  

\begin{lemma}
\label{lem:rad}
Let $S$ be a surface, let $Y$ be a Dehn--Thurston cell structure on $\PMF(S)$, and let $c$ be a curve of the underlying pants decomposition.  Then the radius of $Y$ with respect to $c$ is 1.  
\end{lemma}
\begin{proof}

There are two types of train tracks in the given collection of Dehn--Thurston train tracks: (1) those that contain $c_1$ as a vertex cycle and (2) those that do not carry $c_1$.  To prove the lemma we must show that each train track of the second type carries a curve in common with some train track of the first type.  

So let $\tau$ be of the second type.  This means that the model train track $\tau_A$ for $\tau$ in the annulus $A$ around $c_1$ does not carry $c_1$.  Of the two model tracks on $A$ that carry $c_1$, there is exactly one with the property that there is a homotopy class $\alpha$ of paths in $A$ (relative to $\partial A$) carried both by this model and $\tau_A$; see Figure~\ref{fig:dt}.  Let $\tau'$ be the train track obtained from $\tau$ by changing the model on $A$ to this model.   

Let $c$ be a curve in $S$ that is carried by $\tau$ and so that each component of the intersection with $A$ is parallel to $\alpha$ (relative to $\partial A$); for instance, there is a vertex cycle for $\tau$ that passes through $A$ once and satisfies this property.  Then $c$ is also carried by $\tau'$ since $\tau$ agrees with $\tau'$ outside $A$.  This completes the proof.  
\end{proof}

\p{Sinks and eigenvalues} Finally, we give a lemma that allows us to conclude the second statement of Theorem~\ref{thm:main} from the first.  As in the introduction, a vertex ray for an integral cone in an integral cone complex $Y$ is a vertex of the corresponding polyhedron in $\P(Y)$.  

\begin{lemma}
\label{lem:specrad}
Suppose that $Y$ is an integral cone complex and that the group $G$ acts on $Y$ with property $Q$.  Let $f : Y \to Y$ be a projectively source-sink integral cone complex map with sink $v$, associated eigenvalue $\lambda$, and extended dynamical map given by $(V_i,D)$.  Suppose that $v$ does not lie on a rational ray in $V_i$.  Then the largest real eigenvalue of $D$ is equal to $\lambda$ and any corresponding positive eigenvector represents $v$.  
\end{lemma}

\begin{proof}[Proof of Lemma~\ref{lem:specrad}]

As in the extended dynamical matrix algorithm, the extended dynamical matrix $D$ induces a linear map of a subspace $W$ of $\R^d$, and moreover $v$ lies in $W$.  Suppose that $w$ is an eigenvector of $D$ with real eigenvalue $\mu > 1$.  Then $D$ acts on the subspace $W'$ of $W$ spanned by $v$ and $w$.  Also, by the description of $D$ from the extended dynamical matrix algorithm, $w$ lies in the subspace of $\R^d$ spanned by $V_i$.  

Let $U \subseteq V_i$ be the PL-eigenregion used to determine $D$.  Since $v$ is not rational, it does not lie in a vertex ray of $U$, the intersection $W' \cap U$ contains a two-dimensional cone $C$ (for the same reason that the intersection of a line with a polytope is either a vertex or contains an open interval).   

The actions of $f$ and $D$ agree on $C$.  Since $f$ acts with projective source-sink dynamics on $Y$, we can choose a smaller cone $C'$ around the span of $v$ with the property that $f(C')$ is properly contained in $C'$.  It follows that $\mu < \lambda$, as desired.  The lemma follows.  
\end{proof}

\begin{proof}[Proof of Theorem~\ref{thm:main}] 

We begin with the first statement.  Let $\{\tau_i\}$ be a finite set of complete, transversely recurrent train tracks giving the integral cone complex structure $Y$ on $\MF(S)$, and say that $c$ is a vertex cycle of $\tau_0$.  Let $C$ be a top-dimensional linear cone for the $f$-action on $Y$ that contains $f^Q(c)$.  

We will show that $C$ shares an open subset with a Nielsen--Thurston eigenregion for $f$, namely, the region $V(\wt \nu)$ below, and that both are mapped by $f$ into the same $V(\tau_j)$.  From this it follows that the two linear actions must be equal.  In particular, there is a single $f$-linear piece containing both. This linear piece is still a Nielsen--Thurston eigenregion. The first statement of the theorem follows from this. 

To fix notation, suppose that $C \subseteq V(\tau_i)$ and that $f$ maps $C$ linearly into $V(\tau_j)$ (while our proof applies in all cases, the proof simplifies in the case $i=j$, which is the generic case; we leave this simplification to the reader).  

We claim that both $V(\tau_i)$ and $V(\tau_j)$ contain $\F_h$.   Using the basic properties of horizontality, the definition of $Q$, and the definition of $\Qdt$ with Proposition~\ref{prop:boundedslope}, we have: 
\begin{align*}
\H_f(f^Q(c)) &=  \H_f(c) + Q \\
&= \H_f(c) + \Qforce + \Qfit + \Qdt + 1 \\
&\geq \H_f(\tau_i) + \Qforce + \Qfit + 1,
\end{align*}
and similarly for $\tau_j$.  As $\Qforce + \Qfit + 1 > \Qforce$, the claim now follows from Proposition~\ref{prop:forcing}.

Let
\[
M = \max  \H_f(\tau_i) \cup \H_f(\tau_j).
\]
By the previous claim, \Cref{prop:squeeze} now gives an $f$-invariant train track $\nu$ that fully carries $\F_h$, that is carried by both $\tau_i$ and $\tau_j$, and that satisfies 
\[
\H_f(\nu) \leq M + \Qfit.
\]
We now claim that
\[
\H_f(f^Q(c)) > \H_f(\nu) + \Qforce.
\]
Indeed, let $\tau$ be the element of $\{\tau_i,\tau_j\}$ realizing $M$.  Using the inequality obtained in the last claim, the definition of $M$, and the definition of $\nu$ we have
\begin{align*}
\H_f(f^Q(c)) &\geq \max \H_f(\tau) + \Qforce + \Qfit +1 \\
&= M + \Qforce + \Qfit +1 \\ 
&\geq \H_f(\nu) + \Qforce +1,
\end{align*}
and the claim follows.

Combining the previous claim with \Cref{prop:forcing}, we obtain that $f^Q(c)$ is fully carried by some diagonal extension $\nu_1$ of $\nu$. As $f^Q(c)$ and $\nu$ are each carried by $\tau_i$, it follows from \Cref{lem:still_carry} that $\nu_1$ is also carried by $\tau_i$.  It then follows from Lemma~\ref{lem:ttnbd} that there is a diagonal extension $\wt \nu$ of $\nu_1$ that is carried by $\tau_i$ and so that the interiors of $V(\wt \nu)$ and $C$ have nonempty intersection ($\wt \nu$ is also an extension of $\nu$, hence the notation).  Let $y$ be a point in this intersection. This means in particular that $\wt \nu$ fully carries $y$. As above, our goal is to show $V(\wt \nu)$ is a Nielsen--Thurston eigenregion.

To this end, we claim $f(\wt \nu) \prec \tau_j$. First, we have $\nu \prec \tau_j$ so $f(\nu) \prec \tau_j$. Also, since $f(C)$ lies in $V(\tau_j)$ we have $f(y) \prec \tau_j$. Since $\wt \nu$ fully carries $y$, $f(\wt \nu)$ fully carries $f(y)$, and so Lemma~\ref{lem:still_carry} implies that $\tau_j$ carries the unique diagonal extension of $f(\nu)$ that fully carries $f(y)$. But this unique extension must be $f(\wt \nu)$, since $\wt \nu$ is the unique extension of $\nu$ fully carrying $y$ and since $f$ is (represented by) a homeomorphism.  

We may now conclude that $V(\wt \nu)$ is a Nielsen--Thurston eigenregion for $f$.  Indeed, $V(\wt \nu)$ contains $\F_h$  since $V(\nu)$ does, and the last claim implies that $f$ is linear on $V(\wt \nu)$.  As discussed, this completes the proof of the first statement.

The second statement follows from the first statement and Lemma~\ref{lem:specrad}.  Indeed, for the $\Mod(S)$-action on the Dehn--Thurston integral cone complex, the foliation $\F_h$ automatically satisfies the condition of not lying in a vertex ray.  The reason is that in constructing the linear pieces for $f$ as in the basic computational algorithm, the subdivisions of the cells of $Y$ are made using rational planes.  It follows that the vertices of the resulting subdivision of $\P(Y)$ are rational, which means they correspond to multicurves in $S$.  On the other hand, the unstable foliation for a pseudo-Anosov mapping class does not correspond to a multicurve.  This completes the proof.
\end{proof}

\bibliographystyle{plain}
\bibliography{ntc_pA}

\end{document}